\documentclass[11pt,british,reqno]{amsart}
\usepackage{charter}
\usepackage{comment}
\usepackage[T1]{fontenc}
\usepackage[latin9]{inputenc}
\usepackage{geometry}
\usepackage{color}
\usepackage{babel}
\usepackage{prettyref}
\usepackage{amstext}
\usepackage{amsthm}
\usepackage{amssymb}
\usepackage{setspace}
\usepackage[numbers]{natbib}
\usepackage[unicode=true,
 bookmarks=true,bookmarksnumbered=false,bookmarksopen=false,
 breaklinks=false,pdfborder={0 0 1},%backref=page,
 colorlinks=true]
 {hyperref}
\hypersetup{
 pdfauthor={Adam Skalski and Ami Viselter},
 pdfstartview={XYZ null null 1.25},citecolor=OliveGreen,linkcolor=RoyalBlue,urlcolor=blue}

\makeatletter
\numberwithin{equation}{section}
\theoremstyle{plain}
\newtheorem{thm}{\protect\theoremname}[section]
  \theoremstyle{remark}
  \newtheorem{rem}[thm]{\protect\remarkname}
  \theoremstyle{plain}
  \newtheorem{prop}[thm]{\protect\propositionname}
\theoremstyle{remark}
\newtheorem*{convention*}{Convention}
  \theoremstyle{definition}
  \newtheorem{defn}[thm]{\protect\definitionname}
  \theoremstyle{plain}
  \newtheorem{lem}[thm]{\protect\lemmaname}
  \theoremstyle{plain}
  \newtheorem{cor}[thm]{\protect\corollaryname}
  \theoremstyle{definition}
  \newtheorem{example}[thm]{\protect\examplename}
  \theoremstyle{plain}
  \newtheorem{question}[thm]{\protect\questionname}

\usepackage{eucal}
\let\mathcal=\CMcal
\usepackage{dsfont}
\usepackage{graphicx}
\usepackage{mathtools} % for \prescript
\usepackage[all]{xy}
\usepackage[usenames,dvipsnames]{xcolor}
\usepackage{upgreek}

\let\ldash=\l

\renewcommand{\theenumi}{(\alph{enumi})}

\newrefformat{eq}{\eqref{#1}}
\newrefformat{def}{Definition \ref{#1}}
\newrefformat{lem}{Lemma \ref{#1}}
\newrefformat{pro}{Proposition \ref{#1}}
\newrefformat{prop}{Proposition \ref{#1}}
\newrefformat{cor}{Corollary \ref{#1}}
\newrefformat{thm}{Theorem \ref{#1}}
\newrefformat{exa}{Example \ref{#1}}
\newrefformat{rem}{Remark \ref{#1}}
\newrefformat{sec}{Section \ref{#1}}
\newrefformat{sub}{Subsection \ref{#1}}
\newrefformat{conj}{Conjecture \ref{#1}}
\newrefformat{desc}{\ref{#1}}
\newrefformat{ques}{Question \ref{#1}}

\newcommand*{\doi}[1]{\href{http://dx.doi.org/#1}{doi: #1}}

\DeclareSymbolFont{YHlargesymbols}{OMX}{yhex}{m}{n}
\DeclareMathAccent{\wideparen}{\mathord}{YHlargesymbols}{"F3}

\makeatother

  \providecommand{\corollaryname}{Corollary}
  \providecommand{\definitionname}{Definition}
  \providecommand{\examplename}{Example}
  \providecommand{\lemmaname}{Lemma}
  \providecommand{\propositionname}{Proposition}
  \providecommand{\questionname}{Question}
  \providecommand{\remarkname}{Remark}
\providecommand{\theoremname}{Theorem}

\begin{document}
\global\long\def\e{\varepsilon}
\global\long\def\N{\mathbb{N}}
\global\long\def\Z{\mathbb{Z}}
\global\long\def\Q{\mathbb{Q}}
\global\long\def\R{\mathbb{R}}
\global\long\def\C{\mathbb{C}}
\global\long\def\G{\mathbb{G}}
\global\long\def\QG{\mathbb{G}}
\global\long\def\QH{\mathbb{H}}
\global\long\def\bn{\mathbb{N}}
\global\long\def\br{\mathbb{R}}
\global\long\def\bc{\mathbb{C}}
\global\long\def\bt{\mathbb{T}}

\global\long\def\H{\EuScript H}
\global\long\def\J{\mathcal{J}}
\global\long\def\K{\mathcal{K}}
\global\long\def\KHilb{\EuScript K}
\global\long\def\a{\alpha}
\global\long\def\be{\beta}
\global\long\def\l{\lambda}
\global\long\def\om{\omega}
\global\long\def\z{\zeta}
\global\long\def\gnsmap{\upeta}
\global\long\def\Aa{\mathcal{A}}
\global\long\def\Aalg{\mathsf{A}}
\global\long\def\Sant{\mathtt{S}}
\global\long\def\Rant{\mathtt{R}}

\global\long\def\Ree{\operatorname{Re}}
\global\long\def\Img{\operatorname{Im}}
\global\long\def\linspan{\operatorname{span}}
\global\long\def\slim{\operatorname*{s-lim}}
\global\long\def\clinspan{\operatorname{\overline{span}}}
\global\long\def\co{\operatorname{co}}
\global\long\def\pres#1#2#3{\prescript{#1}{#2}{#3}}

\global\long\def\tensor{\otimes}
\global\long\def\tensormin{\mathbin{\otimes_{\mathrm{min}}}}
\global\long\def\tensorn{\mathbin{\overline{\otimes}}}

\global\long\def\A{\forall}

\global\long\def\i{\mathrm{id}}
\global\long\def\tr{\operatorname{tr}}

\global\long\def\one{\mathds{1}}
\global\long\def\Ww{\mathds{W}}
\global\long\def\wW{\text{\reflectbox{\ensuremath{\Ww}}}\:\!}
\global\long\def\op{\mathrm{op}}
\global\long\def\WW{{\mathds{V}\!\!\text{\reflectbox{\ensuremath{\mathds{V}}}}}}
\global\long\def\Vv{\mathds{V}}
\global\long\def\vV{\text{\reflectbox{\ensuremath{\Vv}}}\:\!}

\global\long\def\B#1{\mathcal{B}(#1)}
\global\long\def\M#1{\operatorname{M}(#1)}
\global\long\def\Linfty#1{L^{\infty}(#1)}
\global\long\def\Lone#1{L^{1}(#1)}
\global\long\def\Lp#1{L^{p}(#1)}
\global\long\def\Lq#1{L^{q}(#1)}
\global\long\def\LoneSharp#1{L_{\sharp}^{1}(#1)}
\global\long\def\Ltwo#1{L^{2}(#1)}
\global\long\def\Cz#1{\mathrm{C}_{0}(#1)}
\global\long\def\CzU#1{\mathrm{C}_{0}^{\mathrm{u}}(#1)}
\global\long\def\CzUSSharp#1{\mathrm{C}_{0}^{\mathrm{u}}(#1)_{\sharp}^{*}}
\global\long\def\CU#1{\mathrm{C}^{\mathrm{u}}(#1)}
\global\long\def\Cb#1{\mathrm{C}_{b}(#1)}
\global\long\def\CStarF#1{\mathrm{C}^{*}(#1)}
\global\long\def\CStarR#1{\mathrm{C}_{\mathrm{r}}^{*}(#1)}
\global\long\def\Cc#1{\mathrm{C}_{c}(#1)}
\global\long\def\CC#1{\mathrm{C}(#1)}

\global\long\def\linfty#1{\ell^{\infty}(#1)}
\global\long\def\lone#1{\ell^{1}(#1)}
\global\long\def\ltwo#1{\ell^{2}(#1)}
\global\long\def\Pol#1{\mathrm{Pol}(#1)}
\global\long\def\Ltwozero#1{L_{0}^{2}(#1)}
\global\long\def\Irred#1{\mathrm{Irred}(#1)}
\global\long\def\conv{\star}

\global\long\def\Ad#1{\mathrm{Ad}(#1)}
\global\long\def\VN#1{\mathrm{VN}(#1)}
\global\long\def\d{\,\mathrm{d}}
\global\long\def\t{\mathrm{t}}

\global\long\def\tie#1{\wideparen{#1}}

\title[Convolution semigroups on quantum groups and noncommutative Dirichlet
forms]{Convolution semigroups on locally compact quantum groups and noncommutative
Dirichlet forms}

\author{Adam Skalski}

\address{Institute of Mathematics of the Polish Academy of Sciences, ul.~\'Sniadeckich
8, 00-656 Warszawa, Poland}

\email{a.skalski@impan.pl}

\author{Ami Viselter}

\address{Department of Mathematics, University of Haifa, 31905 Haifa, Israel}

\email{aviselter@univ.haifa.ac.il}
\begin{abstract}
The subject of this paper is the study of convolution semigroups of states on a locally compact quantum group,
generalising classical families of distributions of a L\'{e}vy process
on a locally compact group.  In particular a definitive one-to-one correspondence between symmetric
convolution semigroups of states and noncommutative Dirichlet forms
satisfying the natural translation invariance property is established, extending earlier partial results and providing a powerful tool to analyse such semigroups.
This is then applied to provide new characterisations of the Haagerup
Property and Property (T) for locally compact quantum groups, and some
examples are presented. The proofs of the main theorems require developing
certain general results concerning Haagerup's $L^{p}$-spaces.
\end{abstract}


\maketitle

\section*{Introduction}

The connection between the convolution operation, defined on probability
measures on $\br$, and independence of random variables is one of
the key elementary features of classical probability theory. In particular
convolution semigroups of probability measures are precisely families
of distributions of $\br$-valued stochastic processes with independent
and identically distributed increments, i.e.~of \emph{L\'{e}vy processes}.
It is not too difficult to see that defining the convolution operation
requires only that the underlying space admits a semigroup structure;
in fact the natural setup for studying L\'{e}vy processes, on one
hand sufficiently rich to allow plenty of important examples, and
on the other hand sufficiently specific to facilitate the application
of strong functional-analytic methods, is that of a \emph{locally
compact group}, denoted henceforth by $G$. Here convolution semigroups
of measures generate convolution semigroups of operators, acting either
on the algebra of continuous functions vanishing at infinity, $\Cz G$,
on the von Neumann algebra of essentially bounded measurable functions, $\Linfty G$,
or on the scale of the $\Lp G$-spaces. Among recent monographs describing
(some of) this vast area of study we recommend \citep{Applebaum__Levy_proc_stock_calc}
and \citep{Liao__book}. The convolution semigroups of operators associated
to a L\'{e}vy process form a specific subclass of \emph{Markov semigroups}
\citep{Bakry_Ledoux_Saloff-Coste__Markov_sgrps_SF}. The latter are
often studied via their generators; a related key tool, which will
play a very important role in this paper, is that of \emph{Dirichlet
forms} \textendash{} quadratic forms on the appropriate $L^{2}$-space satisfying
the specific conditions identified by Beurling and Deny (see \citep{Fukushima_Oshima_Takeda__Dirichlet_forms_Markov_proc}
and references therein). %??? Mention generating function(al)s here and the  L\'evy-Khinchine/Hunt formula???

The development of mathematical approaches to quantum mechanics using
the language of operator algebras, dating back to von Neumann, has
led to the study of \emph{quantum Markov semigroups}, understood as
semigroups of completely positive unital maps on a von Neumann algebra
equipped with a reference state or weight and representing quantum
stochastic evolutions of open systems (see for example \citep{Davies__quantum_thy_open_sys,Meyer__quant_prob_for_prob}).
These, as their classical counterparts, are studied via their generators,
and once again the Dirichlet forms become an indispensable tool. In
the noncommutative context Dirichlet forms were first studied by Albeverio
and H{\o}egh-Krohn, then by Davies and Lindsay in the case where the
reference state is tracial, and later by Goldstein and Lindsay \citep{Goldstein_Lindsay__KMS_symm_Markov_sgs}
and independently by Cipriani \citep{Cipriani__Dirichlet_forms_Markov_sgs_standard_forms}
for general non-tracial states. Finally in \citep{Goldstein_Lindsay__Markov_sgs_KMS_symm_weight}
noncommutative Dirichlet forms were investigated in the most general
context, in which the role of the reference `measure' was played by
an arbitrary \emph{normal semi-finite faithful weight}. For the history,
motivations behind the introduction of quantum Dirichlet forms and
several examples we refer to the survey \citep{Cipriani__Dirichlet_forms_noncomm_space}.
It has to be noted that the passage from the tracial to the non-tracial
case vastly increases the technical complexity of the problem, as
for example one needs to consider Haagerup's $L^{p}$-spaces \citep{Terp__L_p_spaces,Terp__interpolation}
instead of the `tracial' $L^{p}$-spaces of Nelson \citep{Pisier_Xu__noncomm_Lp_spaces}.
In the non-tracial context we may also consider several different
natural ways of passing from the maps at the von Neumann algebra level
to the maps on the corresponding $L^{2}$-space, with the two most
prominent ones being the so-called \emph{GNS-} and \emph{KMS-implementations},
see for example \citep{Caspers_Skalski__Haagerup_AP_Dirichlet_forms}.

In view of the first paragraph above it is natural to ask about the
convolution structures which might play an important role also in
the quantum context. This was recognised relatively early in the development
of the theory of quantum L\'{e}vy process initiated by Sch\"{u}rmann
(see \citep{Schurmann__white_noise_bialg,Franz__Levy_proc_QGs_dual_grps}),
in which the underlying quantum probability space was represented
by a Hopf $*$-algebra, or more generally a $*$-bialgebra. The huge
success of the theory of compact quantum groups due to Woronowicz
\citep{Woronowicz__CMP,Woronowicz__symetries_quantiques}, and more
generally locally compact quantum groups due to Kustermans and Vaes
\citep{Kustermans_Vaes__LCQG_C_star}, with the latter objects denoted below by $\QG$,
has opened a possibility to study analogous problems in the much richer
analytic setting. Here, in parallel to the situation described in
the previous paragraph, the level of technical difficulty involved
in the locally compact theory by far exceeds that of the compact case,
where many questions can be still investigated via purely algebraic
means, exploiting the Hopf $*$-algebra $\textup{Pol}(\QG)$. Kustermans's
and Vaes's theory admits perfect duality, with $\hat{\QG}$ denoting
the dual of $\QG$. Discrete quantum groups are duals of compact ones,
which means that also in the discrete case one can take advantage
of certain algebraic techniques. Recall here only that quantum groups
are in fact studied indirectly, via associated algebras of functions,
such as $\Cz{\QG}$, $\CzU{\QG}$ or $\Linfty{\QG}$ (see \citep{Kustermans_Vaes__LCQG_C_star,Kustermans_Vaes__LCQG_von_Neumann,Kustermans__LCQG_universal}).

Recent years brought a significant interest in convolution
operators on arbitrary locally compact quantum groups: these featured
prominently in the work of Junge, Neufang, Ruan, Daws, Hu and others
(see e.g.\ \citep{Junge_Neufang_Ruan__rep_thm_LCQG} and \citep{Daws__mult_LCQG_Hilbert_C_mod}).
It should be noted that the quantum group context allows one to treat
the usual convolution operators and Herz\textendash Schur multipliers
on a locally compact group $G$ within the same framework, with the
latter viewed as convolution operators on the dual quantum group $\hat{G}$.
Perhaps slightly surprisingly, there has been less work on the convolution
semigroups of states on quantum groups beyond the compact case: some
initial facts were established in \citep{Lindsay_Skalski__conv_semigrp_states}
(where in particular such semigroups were shown to be determined by
densely defined \emph{generating functionals}) and then applied in
\citep{Lindsay_Skalski__quant_stoch_conv_cocyc_3} to study the analytic
theory of quantum L\'{e}vy processes. On the other hand, in an important
recent paper \citep{Cipriani_Franz_Kula__sym_Levy_proc}, Cipriani,
Franz and Kula continued the study of convolution semigroups of states
on compact quantum groups, earlier investigated in the context of
algebraic quantum L\'{e}vy processes, for the first time involving
in the analysis the Dirichlet form techniques.

In this work we undertake a deeper study of the convolution semigroups
associated with a locally compact quantum group $\QG$. As in the
classical case these turn out to have several avatars: semigroups
of states of $\CzU{\QG}$, or semigroups of completely positive operators
on each of the algebras $\CzU{\QG}$, $\Cz{\QG}$, $\Linfty{\QG}$, etc. % semigroups of (KMS- or GNS-implemented) Hilbert space operators on $L^2(\QG)$, interpolated maps acting on the Haagerup $L^p$-spaces. 
%The following theorem (Theorem 3.4 (??) in the main body of the paper) describes some of these connections, drawing to a large extent on the existing work of Junge, Neufang, Ruan, Daws and others mentioned above.
%\begin{thmmain} There exist 1-1 correspondences between the following objects:
%	\begin{itemize}
%		\item $w^{*}$-continuous convolution
%		semigroups %$\left(\mu_{t}\right)_{t\geq0}$ 
%		of states of $\C_0^u(\QG)$;
%		\item$C_{0}$-semigroups %$\left(T_{t}^{\mathrm{u}}\right)_{t\geq0}$
%		of completely positive maps of norm $1$ on $\C_0^u(\QG)$ that commute
%		with the left convolution operators;
%		\item $C_{0}^*$-semigroups %$\left(T_{t}\right)_{t\geq0}$
%		of normal, unital, completely positive maps on $\Linf(\QG)$ that
%		satisfy the intertwining relation with the co-product;
%		\item $C_{0}$-semigroups $\left(M_{t}\right)_{t\geq0}$
%		of norm $1$ left module maps on $L^1(\QG)$ with completely positive
%		adjoints.
%	\end{itemize}
%	\end{thmmain}
Here we show further that every convolution operator on $\Linfty{\QG}$
leads, in two different ways (via the KMS- and GNS-implementations
mentioned above) to a bounded operator on $\Ltwo{\QG}$. In presence of the natural
symmetry (again, visible on all the levels at which we can view the
respective convolution operator or convolution semigroup) we use the
theory of noncommutative Dirichlet forms for weights to obtain
the central result of the article (\prettyref{thm:corres_conv_smgrps_Dirichlet_forms}
in the main body of the paper).
\begin{thm}
\label{thm:intro__corres_conv_smgrps_Dirichlet_forms}There exist
$1-1$ correspondences between the following objects: 
\begin{itemize}
\item $w^{*}$-continuous convolution semigroups %$\left(\mu_{t}\right)_{t\geq0}$ 
 of states of $\CzU{\QG}$ invariant under the unitary antipode; 
\item $C_{0}^{*}$-semigroups of normal, unital, completely positive maps
on $\Linfty{\QG}$ that satisfy the intertwining relation with the
co-product and are KMS-symmetric with respect to the left Haar weight
of $\QG$; 
\item $C_{0}$-semigroups of selfadjoint completely Markov operators on
$\Ltwo{\QG}$ (with respect to the left Haar weight of $\QG$) that
belong to $\Linfty{\hat{\QG}}$; 
\item completely Dirichlet forms (with respect to the left Haar weight of
$\QG$) on $\Ltwo{\QG}$ that are invariant under $\mathcal{U}(\Linfty{\hat{\QG}}')$, modulo multiplication of forms by a positive number. 
\end{itemize}
\end{thm}

All the notions featuring in the above result will be made precise in the remainder of the paper. Here
note only that the special property of the Dirichlet forms in the last
point above can be naturally interpreted as translational invariance,
bringing us back to the classical context studied in \citep[Sections 1.4 and 1.5]{Fukushima_Oshima_Takeda__Dirichlet_forms_Markov_proc}
and in \citep[Section 6]{Deny__methods_hilbertiennes_thy_potentiel}.

The theorem above provides a definitive extension of the results of \citep{Cipriani_Franz_Kula__sym_Levy_proc}.
The technical difficulties in the locally compact case are far
more daunting: the generators of the respective semigroups admit no `obvious'
domain and one needs to deal with weights as opposed to states. The
key step of the proof uses ideas of Goldstein and Lindsay; however as
the proof of one of the main results in \citep{Goldstein_Lindsay__Markov_sgs_KMS_symm_weight}
appears to be incorrect, we have to provide a new argument to obtain
the passage from the Dirichlet form to the semigroup of operators acting on the von Neumann
algebra, requiring slightly stronger assumptions than those stated
in \citep{Goldstein_Lindsay__Markov_sgs_KMS_symm_weight}. Nevertheless, these are
automatically satisfied in our quantum group context. We believe this
general result is of independent interest. Returning to the quantum group
framework, we furthermore present a very explicit connection between
the relevant noncommutative Dirichlet form and the respective generating
functional as studied in \citep{Lindsay_Skalski__conv_semigrp_states}.

The above theorem and its technical extensions turn out to have significant
applications to geometric quantum group theory. In particular we deduce
that a second countable locally compact quantum group $\QG$ does
not have Kazhdan's Property (T) (see \citep{Daws_Skalski_Viselter__prop_T}
and references therein) if and only if there exists a symmetric (in
the sense of the first bullet point in Theorem \ref{thm:intro__corres_conv_smgrps_Dirichlet_forms}) convolution semigroup
of states of $\CzU{\hat{\QG}}$ with an unbounded (equivalently, not
everywhere defined) generating functional. On the other hand $\QG$
has the Haagerup Property if and only if there exists a semigroup
of Markov convolution operators on $\Ltwo{\QG}$ (in the sense of
the third bullet point of the Theorem above) belonging to $\Cz{\QG}$.
Classically the passage between the representation theoretic aspects
of properties such as Property (T) and their geometric guises goes through
the Sch\"onberg correspondence: suitable positive-definite functions
are first assembled into a semigroup, whose generator is a conditionally
negative-definite function $\psi$, and then one identifies the desired affine
action of the group on a Hilbert space with the cocycle associated
with $\psi$; the converse direction requires Sch\"onberg's result. 
For discrete quantum groups one no longer has a direct notion of affine
actions, and works only with the respective cocycles \textendash{}
understood as certain derivations of the underlying Hopf algebra (instances of such techniques
can be found in \citep{Kyed__cohom_prop_T_QG} and
\citep{Daws_Fima_Skalski_White_Haagerup_LCQG}). 
We emphasise that this approach relies crucially on the algebraic nature of discrete quantum groups, which is absent from the general locally compact case.
Thus in this paper we provide a yet different perspective, encompassing all the results mentioned above: namely, we
encode the representation theoretic properties of a locally compact
quantum group via the existence (or non-existence) of convolution
semigroups with particular properties.

Also the construction of examples in the locally compact setting is
far more complicated than it was in \citep{Cipriani_Franz_Kula__sym_Levy_proc},
once again mainly due to the lack of a canonical `purely algebraic'
domain for the generators. We describe in detail the dual to the classical
case, and for the classical case refer again to  \citep[Sections 1.4 and 1.5]{Fukushima_Oshima_Takeda__Dirichlet_forms_Markov_proc} and \citep[Section 6]{Deny__methods_hilbertiennes_thy_potentiel}.
We further show how to use cocycle twisting to generate interesting
examples acting on genuine locally compact quantum groups, e.g.~on
the quantised Heisenberg group of Enock and Vainerman \citep{Enock_Vainerman__deform_Kac_alg_abel_grp}.

Our work opens several future directions of research. It is natural
to ask about the possibility of extending the key construction of
derivations out of quantum Dirichlet forms due to Cipriani and Sauvageot
\citep{Cipriani_Sauvageot__der_sqroot_Dirichlet_forms} beyond the
tracial case, exploiting the additional quantum group structure. We
intend also to study further examples of non-trivial convolution semigroups
using the Rieffel deformation techniques of \citep{Kasprzak__Rieffel_deform_crossed_prod},
in a sense dual to the procedure of cocycle twisting discussed in
this paper. Finally one might investigate how the properties of the
noncommutative Dirichlet form affect the long-term behaviour of the
quantum Markov semigroup in question both in the quantitative and
the qualitative senses, and exploit the Dirichlet form techniques
to produce interesting perturbations of the objects studied here (perhaps
landing outside of the class of the convolution semigroups).

The concrete plan of the paper is as follows: in Section \ref{sec:prelim} we discuss preliminaries concerning von Neumann algebras, the associated Haagerup $L^p$-spaces, and noncommutative Dirichlet forms,  prove several technical results related to the $L^p$-embeddings introduced in \cite{Goldstein_Lindsay__Markov_sgs_KMS_symm_weight} and introduce notation and terminology concerning locally compact quantum groups. Section \ref{sec:L_R_oper} treats convolution operators associated with quantum groups, first recalling known results, and then focusing on the existence and properties of $L^2(\QG)$-implementations and equivalences between various modes of convergence. A short Section \ref{sec:semigroups} contains the main general theorems of the paper, in particular Theorem \ref{thm:intro__corres_conv_smgrps_Dirichlet_forms}, and connects the noncommutative Dirichlet form associated with a convolution semigroup with the generating functional of \cite{Lindsay_Skalski__conv_semigrp_states}. Section \ref{sec:prop_T_Haagerup} provides applications to geometric quantum group theory, first concerning Property (T) and then the Haagerup Property, and Section \ref{sec:examples} discusses several examples. Finally in the 
Appendix we provide a correct proof of %(one direction of) 
\citep[Theorem 4.7]{Goldstein_Lindsay__Markov_sgs_KMS_symm_weight} under two different sets of mild additional assumptions.

\subsection*{Acknowledgements}

The first author was partially supported by the National Science Centre
(NCN) grant no.~2014/14/E/ST1/00525. We thank Martin Lindsay and Stanis{\ldash}aw Goldstein for encouragement, and Stuart White and the anonymous referee for valuable comments.
During the last stage of the preparation of this paper the second author was visiting Sutanu Roy at the School of Mathematical Sciences of NISER, Bhubaneswar, India. He is grateful to him and to his colleagues for their warm hospitality.

\section{Preliminaries}\label{sec:prelim}

We begin with some basic notation and conventions. For a Banach space
$\mathsf{X}$, let $B(\mathsf{X})$ stand for the algebra of bounded
operators on $\mathsf{X}$. A (one-parameter) \emph{operator semigroup}
on $\mathsf{X}$ is a family $\left(T_{t}\right)_{t\ge0}$ in $B(\mathsf{X})$
satisfying $T(0)=I$ and the semigroup identity: $T(t+s)=T(t)T(s)$
for all $t,s\ge0$. We say that $\left(T_{t}\right)_{t\ge0}$ is a
\emph{$C_{0}$-semigroup} if it is continuous at $0^{+}$ in the strong
operator topology, namely $\lim_{t\to0^{+}}T(t)x=x$ for every $x\in\mathsf{X}$
\citep{Davies__semigroups,EngelNagel,Kantorovitz__oper_semi}. We will frequently
use the fact that this continuity condition is equivalent to continuity
at $0^{+}$ in the weak operator topology, namely that $\lim_{t\to0^{+}}T(t)x=x$
weakly for every $x\in\mathsf{X}$ \citep[Proposition 1.23]{Davies__semigroups}.

For an algebra $\mathsf{A}$ we denote by $\i:\mathsf{A}\to\mathsf{A}$
the identity map and by $\one$ the unit of $\mathsf{A}$, if it exists. For subsets $X,Y\subseteq\mathsf{A}$
we write $XY:=\linspan\left\{ xy:x\in X,y\in Y\right\} $.

Let $\mathsf{A}$ be a C$^{*}$-algebra. The multiplier algebra of
$\mathsf{A}$ is denoted by $\M{\mathsf{A}}$. For $\om\in\mathsf{A}^{*}$
we define $\overline{\om}\in\mathsf{A}^{*}$ by $\overline{\om}(x):=\overline{\om(x^{*})}$,
$x\in\mathsf{A}$. Representations of C$^{*}$-algebras are always
assumed nondegenerate.

We use $\tensor,\tensormin,\tensorn$ for the algebraic or Hilbert
space tensor product, the minimal tensor product of C$^{*}$-algebras,
and the normal tensor product of von Neumann algebras, respectively.
We write $\sigma$ for the flip map at the C$^{*}$- or the von Neumann
algebra level.

For a Hilbert space $\H$, write $\K(\H)$ for the C$^{*}$-algebra
of compact operators on $\H$ and $\mathcal{U}(\H)$ for the group
of all unitary operators on $\H$. Given $u\in\mathcal{U}(\H)$, set
$\Ad u(x):=uxu^{*}$ for $x\in B(\H)$. Inner products are linear
in the left variable. For vectors $\z,\eta\in\H$, $\om_{\z,\eta}\in B(\H)_{*}$
is the functional given by $\om_{\z,\eta}(x):=\left\langle x\z,\eta\right\rangle $,
$x\in B(\H)$. We let $\om_{\z}:=\om_{\z,\z}$. The ultraweak operator topology, 
resp.~ultrastrong operator topology, on a von Neumann algebra
will be called simply the ultraweak topology, resp.~ultrastrong topology. An automorphism
group of a C$^{*}$- or von Neumann algebra is assumed by definition
to be continuous with respect to the suitable topology, namely point\textendash norm
or point\textendash ultraweak, respectively.

Subsections \ref{sub:prelim_Haagerup_Lp}\textendash \ref{sub:prelim_new_results}
discuss preliminaries on von Neumann algebras and in particular on the Haagerup $L^p$-spaces. \prettyref{sub:prelim_LCQGs}
introduces locally compact quantum groups.

\subsection{\label{sub:prelim_Haagerup_Lp}The Haagerup noncommutative $L^{p}$-spaces
of von Neumann algebras and the maps $\mathfrak{j}^{(q)}$ and $\mathfrak{i}^{(p)}$}

In this subsection we introduce Haagerup's construction of $L^{p}$-spaces
of von Neumann algebras as announced in \citep{Haagerup__L_p_spaces}
and described in full detail by Terp \citep{Terp__L_p_spaces}, and
the subsequent work of Goldstein and Lindsay \citep[Sections 1--2]{Goldstein_Lindsay__Markov_sgs_KMS_symm_weight}.
We assume that the reader is familiar with modular theory \citep{Stratila__mod_thy,Takesaki__Tomita_thy,Takesaki__book_vol_2}. 

We will rely heavily on the theory of unbounded operators on Hilbert
spaces \citep[Chapter XII]{Dunford_Schwartz_2}. For linear operators
$a,b$ on a Hilbert space $\H$, not necessarily densely defined or
closable, denote by $D(a)$ the domain of $a$, and write $a+b$ and
$ab$ for the sum, respectively product, of $a$ and $b$ with maximal
domains. If $a$ is densely defined, we denote its adjoint by $a^{*}$;
if $a$ is closable, we denote its closure by $\overline{a}$. 

Fix a von Neumann algebra $\Aa$ acting, not necessarily standardly,
on a Hilbert space $\H$. Recall that a (not necessarily densely defined
or closable) linear operator $a$ on $\H$ is \emph{affiliated} with
$\Aa$ if $y'a\subseteq ay'$ for every $y'\in\Aa'$. 

Fix a normal, semi-finite, faithful (n.s.f.) weight $\varphi$ on
$\Aa$. Write $\sigma^{\varphi}=\left(\sigma_{t}^{\varphi}\right)_{t\in\R}$
for the modular automorphism group of $\varphi$, let 
\[
\begin{split}\mathcal{N}_{\varphi} & :=\left\{ a\in\Aa:\varphi(a^{*}a)<\infty\right\} ,\\
\mathcal{M}_{\varphi} & :=\linspan\left\{ a\in\Aa_{+}:\varphi(a)<\infty\right\} =\mathcal{N}_{\varphi}^{*}\mathcal{N}_{\varphi},
\end{split}
\]
let $\mathcal{T}_{\varphi}$ denote the Tomita $*$-algebra of all
elements $a\in\Aa$ that are entire analytic with respect to $\sigma^{\varphi}$
and satisfy $\sigma_{z}^{\varphi}(a)\in\mathcal{N}_{\varphi}\cap\mathcal{N}_{\varphi}^{*}$
for every $z\in\C$, and let $\mathcal{M}_{\varphi,\infty}:=\mathcal{T}_{\varphi}^{*}\mathcal{T}_{\varphi}=\mathcal{T}_{\varphi}\mathcal{T}_{\varphi}$.
The GNS construction for the pair $(\Aa,\varphi)$ gives a Hilbert
space $\Ltwo{\Aa,\varphi}$ and a map $\gnsmap_{\varphi}:\mathcal{N}_{\varphi}\to\Ltwo{\Aa,\varphi}$.
When viewing $\Aa$ as acting on $\Ltwo{\Aa,\varphi}$ we do not use
any additional symbol for the representation. Write $\nabla_{\varphi}$
and $J_{\varphi}$ for the modular operator and modular conjugation
of $\varphi$, respectively, both acting on $\Ltwo{\Aa,\varphi}$.
In the sequel we usually omit $\varphi$ from the notation. 

We begin with an approximation lemma, providing a `good' net of elements in the Tomita algebra converging to $\one$ in the $*$-strong operator topology.

\begin{lem}
\label{lem:Str_Terp_approx}For each $\delta>0$ there exists a net
$\left(e_{j}\right)_{j\in\mathcal{J}}$ in $\mathcal{T}$ such that
for every $z\in\C$ we have: 
\begin{enumerate}
\item \label{enu:Str_Terp_approx__1}$\left\Vert \sigma_{z}(e_{j})\right\Vert \le e^{\delta(\Im z)^{2}}$
for all $j\in\mathcal{J}$; 
\item \label{enu:Str_Terp_approx__2}$\sigma_{z}(e_{j})\xrightarrow[j\in\mathcal{J}]{}\one$
in the $*$-strong operator topology.
\end{enumerate}
\end{lem}

\begin{proof}
Combine \citep[Lemma 9]{Terp__interpolation} and \citep[Proposition 2.16]{Stratila__mod_thy},
whose constructions agree: the former proves the existence of a net
satisfying \prettyref{enu:Str_Terp_approx__1}, as well as \prettyref{enu:Str_Terp_approx__2}
for $z=0$, while the latter treats \prettyref{enu:Str_Terp_approx__1}
and \prettyref{enu:Str_Terp_approx__2} for $\delta=1$. Unfortunately,
in proving \prettyref{enu:Str_Terp_approx__2}, \citep{Stratila__mod_thy}
uses the Lebesgue dominated convergence theorem, which is applicable
only when one can choose in the proof sequences instead of nets (e.g.,
when $\Aa_{*}$ is separable). This can be fixed via the following argument (see also the proof of \citep[Proposition 2.25]{Kustermans__one_param_rep}).
Assume that $\Aa$ is in standard form in
$\H$ and fix $z\in\C$, $\z\in\H$. Write $a:=\Re z$, $b:=\Im z$.
Given $f\in\Aa$ with $\left\Vert f\right\Vert \le1$ let $e:=\frac{\sqrt{\delta}}{\sqrt{\pi}}\int_{\R}e^{-\delta t^{2}}\sigma_{t}(f)\d t$
(convergence in the $*$-strong operator topology). Then we have $\frac{\sqrt{\pi}}{\sqrt{\delta}}\left\Vert \sigma_{z}(e)\z-\z\right\Vert \le e^{\delta b^{2}}\int_{\R}e^{-\delta(t-a)^{2}}\left\Vert \sigma_{t}(f)\z-\z\right\Vert \mathrm{d}t$.
For $\e>0$ one can approximate the last integral up to $\e$ by the
integral on some (bounded) interval $I$ in $\R$ independently of
$f$. For every $t\in\R$ we have $\left\Vert \sigma_{t}(f)\z-\z\right\Vert =\left\Vert f\nabla^{-it}\z-\nabla^{-it}\z\right\Vert $
because $\sigma_{t}$ is implemented by $\nabla^{it}$. By uniform
continuity of $t\mapsto\nabla^{-it}\z$ on $I$ one can approximate
the integral on $I$ up to $\e$ by Riemann sums, again independently of
$f$. This implies that if $\left(f_{j}\right)_{j\in\mathcal{J}}$
in the closed unit ball of $\Aa$ converges in the strong operator
topology to $\one$ then the associated $\left(e_{j}\right)_{j\in\mathcal{J}}$
will satisfy $\sigma_{z}(e_{j})\z\xrightarrow[j\in\mathcal{J}]{}\z$.
\end{proof}

The following material is taken from \citep{NelsonIntegration, Terp__L_p_spaces} unless
indicated otherwise. Let $\mathbb{A}$ be a semi-finite von Neumann
algebra acting, not necessarily standardly, on a Hilbert space $\KHilb$,
and $\tau$ an n.s.f.~trace on $\mathbb{A}$. A linear subspace $E\subseteq\KHilb$
is called $\tau$-\emph{dense} if for every $\delta>0$ there exists
a projection $p\in\mathbb{A}$ such that $\Img p\subseteq E$ and
$\tau(\one-p)<\delta$. A closed, densely defined operator $a$ affiliated
with $\mathbb{A}$ is called $\tau$-\emph{measurable} if $D(a)$
is $\tau$-dense. The set of all $\tau$-measurable operators is denoted by $\pres{\tau}{}{\mathbb{A}}$.
Evidently, $\mathbb{A}\subseteq\pres{\tau}{}{\mathbb{A}}$. In fact,
$\pres{\tau}{}{\mathbb{A}}$ is a $*$-algebra with respect to the
sum $a\dotplus b:=\overline{a+b}$, product $a\cdot b:=\overline{ab}$
and scalar product $\l\cdot a:=\overline{\l a}$ (the `strong operations'),
and the involution being the adjoint $a^{*}$ ($a,b\in\pres{\tau}{}{\mathbb{A}}$,
$\l\in\C$).  
\begin{rem}[{see \citep[proof of Proposition I.24]{Terp__L_p_spaces} or \citep[Corollary 1.2]{Goldstein_Lindsay__Markov_sgs_KMS_symm_weight}}]
\label{rem:tau_meas_poly}Every polynomial in elements of $\pres{\tau}{}{\mathbb{A}}$
formed using sums and products of unbounded operators (without closures) has a unique extension
in $\pres{\tau}{}{\mathbb{A}}$, namely the one given at each step by the operations
in $\pres{\tau}{}{\mathbb{A}}$ (with closures). It also equals
the closure of the original polynomial (which is densely defined).
\end{rem}

To introduce the Haagerup $L^{p}$-spaces, we 
recall that $\Aa$ and $\varphi$ have been fixed above, and
set henceforth $\mathbb{A}:=\Aa\rtimes_{\sigma^{\varphi}}\R$,
the crossed product of $\Aa$ by $\sigma^{\varphi}$ acting on $\Ltwo{\R}\tensor\H\cong\Ltwo{\R,\H}$.
It is generated as a von Neumann algebra by $\pi(\Aa)\cup\lambda(\R)$,
where $\pi$ is the canonical embedding of $\Aa$ in $\mathbb{A}$
given by $(\pi(a)f)(t):=\sigma_{-t}^{\varphi}(a)(f(t))$ ($a\in\Aa$,
$f\in\Ltwo{\R,\H}$, $t\in\R$) and $\l$ is the left regular representation
of $\R$ tensored by $\one_{\H}$, namely $(\l(s)f)(t):=f(t-s)$ ($f\in\Ltwo{\R,\H}$,
$s,t\in\R$). Write $h$ for the injective, positive selfadjoint
operator on $\Ltwo{\R,\H}$ such that $\l(s)=h^{is}$ for all $s\in\R$, whose existence and uniqueness follow from 
Stone's theorem. For simplicity, we usually suppress $\pi$. By
construction, $\sigma_{t}^{\varphi}=\Ad{\lambda(t)}=\Ad{h^{it}}$
for every $t\in\R$. The crossed product $\mathbb{A}$ is equipped
with an automorphism group $\theta=\left(\theta_{s}\right)_{s\in\R}$
that is the action of $\R$ dual to $\sigma^{\varphi}$, characterised
by the equalities  $\theta_{s}(a)=a$ and $\theta_{s}(h^{it})=e^{-ist}h^{it}$ for
all $s,t\in\R$ and $a\in\Aa$. For an n.s.f.~weight $\psi$ on $\Aa$,
we denote by $\widetilde{\psi}$ the (n.s.f.) dual weight on $\mathbb{A}$
\citep{Haagerup__dual_weights_1,Haagerup__dual_weights_2}. There
is a unique n.s.f.~trace $\tau$ on $\mathbb{A}$ that, using the
Pedersen\textendash Takesaki Radon\textendash Nikodym derivative notation,
satisfies $\frac{\mathrm{d}\widetilde{\varphi}}{\mathrm{d}\tau}=h$. 

Denote by $\hat{\mathbb{A}}_{+}$ the extended positive part of $\mathbb{A}$.
The definition of $\widetilde{\psi}$ in \citep{Haagerup__dual_weights_2}
extends naturally to every normal, not necessarily faithful or semi-finite,
weight $\psi$ on $\Aa$, yielding a dual (normal, not necessarily
faithful or semi-finite) weight on $\mathbb{A}$. For every such $\psi$,
let $k_{\psi}:=\frac{\mathrm{d}\widetilde{\psi}}{\mathrm{d}\tau}\in\hat{\mathbb{A}}_{+}$
in the sense of \citep[Theorem 1.12]{Haagerup__oper_val_weights_1};
by definition, $k_{\varphi}=h$. Then $k_{\psi}$ is a genuine (unbounded,
positive, selfadjoint) operator if and only if $\psi$ is semi-finite,
and in this case, $k_{\psi}$ is $\tau$-measurable if and only if
$\psi$ is bounded, that is, belongs to $\Aa_{*}$. Furthermore, the
map ${(\Aa_{*})}_{+}\ni\psi\mapsto k_{\psi}$ extends linearly to
an injection $\psi\mapsto k_{\psi}$ from $\Aa_{*}$ into $\pres{\tau}{}{\mathbb{A}}$.

For each $t\in\R$ the map $\theta_{t}$ extends naturally to a $*$-automorphism
of $\pres{\tau}{}{\mathbb{A}}$, also denoted by $\theta_{t}$. We
define, for $p\in\left[1,\infty\right]$, 
\[
\Lp{\Aa}:=\left\{ x\in\pres{\tau}{}{\mathbb{A}}:\theta_{t}(x)=e^{-t/p}x\text{ for all }t\in\R\right\} .
\]
The set $\Lp{\Aa}$ is a selfadjoint linear subspace of $\pres{\tau}{}{\mathbb{A}}$,
and it is spanned by $\Lp{\Aa}_{+}:=\Lp{\Aa}\cap{(\pres{\tau}{}{\mathbb{A}})}_{+}$.
Up to isomorphism, this construction does not depend on the n.s.f.~weight
$\varphi$. Several special cases are of particular importance. The
set $\Linfty{\Aa}$ coincides with $\Aa$ (that is, its image in $\mathbb{A}$),
while $\Lone{\Aa}=\left\{ k_{\psi}:\psi\in\Aa_{*}\right\} $. One
can thus define a functional $\tr$ on $\Lone{\Aa}$ by the formula $\tr(k_{\psi}):=\psi(\one)$,
$\psi\in\Aa_{*}$. For $p\in[1,\infty)$ and a closed, densely defined
linear operator $a$ affiliated with $\mathbb{A}$ with polar decomposition
$a=u\left|a\right|$, we have $a\in\Lp{\Aa}$ if and only if $u\in\Aa$
and $\left|a\right|^{p}\in\Lone{\Aa}$. This means one can define a norm
on $\Lp{\Aa}$, setting $\left\Vert a\right\Vert _{p}:=\tr(\left|a\right|^{p})^{1/p}$,
and thus making $\Lp{\Aa}$ a Banach space. The canonical linear map $\psi\mapsto k_{\psi}$
is an isometry from $\Aa_{*}$ onto $\Lone{\Aa}$. The functional
$\tr$ is `trace-like': for conjugate exponents $p,q\in\left[1,\infty\right]$,
$a\in\Lp{\Aa}$ and $b\in\Lq{\Aa}$, we have $a\cdot b,b\cdot a\in\Lone{\Aa}$
and $\tr(a\cdot b)=\tr(b\cdot a)$. The Banach space $\Ltwo{\Aa}$
is actually a Hilbert space with respect to the inner product $\left\langle a,b\right\rangle :=\tr(a\cdot b^{*})$,
$a,b\in\Ltwo{\Aa}$. Using the product of $\pres{\tau}{}{\mathbb{A}}$,
we see that $\Aa$ acts on $\Ltwo{\Aa}$ by left multiplication. With
respect to this representation of $\Aa$, the quadruple $(\Aa,\Ltwo{\Aa},*,\Ltwo{\Aa}_{+})$
is a standard representation for $\Aa$. 
More generally, for every $p\in[1,\infty)$, $\Linfty{\Aa}$ acts on $\Lp{\Aa}$ by left and right multiplication in $\pres{\tau}{}{\mathbb{A}}$. These actions are contractive: $\Vert x \cdot a \Vert_p, \Vert a \cdot x \Vert_p \leq \Vert x \Vert_\infty \Vert a \Vert_p$ for all $x\in\Linfty{\Aa}$ and $a\in\Lp{\Aa}$.

We proceed to describe some of \citep[Sections 1--2]{Goldstein_Lindsay__Markov_sgs_KMS_symm_weight}.
For $q\in[2,\infty)$, define a left ideal of $\Aa$ by
\[
\mathcal{N}^{(q)}:=\bigl\{ a\in\Aa:ah^{1/q}\text{ is closable and }\overline{ah^{1/q}}\in\Lq{\Aa}\bigr\}.
\]
We have $\mathcal{N}=\mathcal{N}^{(2)}\subseteq\mathcal{N}^{(q)}$
\citep[p.~46]{Goldstein_Lindsay__Markov_sgs_KMS_symm_weight}. Further define
a map $\mathfrak{j}^{(q)}:\mathcal{N}^{(q)}\to\Lq{\Aa}$ by $\mathcal{N}^{(q)}\ni a\mapsto\overline{ah^{1/q}}$.
It is linear and injective \citep[p.~48]{Goldstein_Lindsay__Markov_sgs_KMS_symm_weight}.
\begin{rem}
\label{rem:N_q}We have $\mathcal{N}^{(q)}=\bigl\{ a\in\Aa:h^{1/q}a^{*}\in\pres{\tau}{}{\mathbb{A}}\bigr\}$.
Indeed, for $a\in\Aa$, observe that `$ah^{1/q}$ is closable and
$\overline{ah^{1/q}}\in\pres{\tau}{}{\mathbb{A}}$' if and only if
`the adjoint of $ah^{1/q}$, namely $h^{1/q}a^{*}$, belongs to $\pres{\tau}{}{\mathbb{A}}$'.
This is because if $ah^{1/q}$ is closable and $\overline{ah^{1/q}}\in\pres{\tau}{}{\mathbb{A}}$,
then also $h^{1/q}a^{*}=\bigl(\overline{ah^{1/q}}\bigr)^{*}\in\pres{\tau}{}{\mathbb{A}}$;
on the other hand, if $(ah^{1/q})^{*}=h^{1/q}a^{*}$ belongs to $\pres{\tau}{}{\mathbb{A}}$,
then in particular it is densely defined, so that $ah^{1/q}$ is closable,
hence $\overline{ah^{1/q}}=(h^{1/q}a^{*})^{*}\in\pres{\tau}{}{\mathbb{A}}$.
Next, when this is the case, we have $h^{1/q}a^{*}\in\Lq{\Aa}$ (equivalently:
$\overline{ah^{1/q}}\in\Lq{\Aa}$) automatically: this is showed in
the second part of the proof of \citep[Proposition 2.2]{Goldstein_Lindsay__Markov_sgs_KMS_symm_weight}
(after `it remains only to show that').
\end{rem}

Given $p\in[1,\infty)$, define the following $*$-subalgebra of $\Aa$:
\[
\mathcal{M}^{(p)}:=\linspan\left\{ a\in\Aa_{+}:a^{1/2}\in\mathcal{N}^{(2p)}\right\} =\linspan\left\{ b^{*}c:b,c\in\mathcal{N}^{(2p)}\right\} .
\]
We have $\mathcal{M}=\mathcal{M}^{(1)}\subseteq\mathcal{M}^{(p)}$.
Also, define a linear map $\mathfrak{i}^{(p)}:\mathcal{M}^{(p)}\to\Lp{\Aa}$
as the unique linear extension of the map that takes $a\in\Aa_{+}$
such that $a^{1/2}\in\mathcal{N}^{(2p)}$ to $\mathfrak{j}^{(2p)}(a^{1/2})^{*}\mathfrak{j}^{(2p)}(a^{1/2})$.
Then $\mathfrak{i}^{(p)}$ is injective and \emph{positivity preserving}.
For all this see \citep[p.~49]{Goldstein_Lindsay__Markov_sgs_KMS_symm_weight}.
The notation $\mathfrak{i}^{(\infty)}:=\i_{\Aa}$ will be useful.
Let us list some key properties of these maps.
\begin{prop}
\label{prop:GL_i}\mbox{}
\begin{enumerate}
\item \label{enu:GL_i_Lem_2_9}For every $p\in[1,\infty)$ and $a,b\in\mathcal{N}^{(2p)}$
we have $\mathfrak{i}^{(p)}(a^{*}b)=\mathfrak{j}^{(2p)}(a)^{*}\cdot\mathfrak{j}^{(2p)}(b)$
\citep[Lemma 2.9]{Goldstein_Lindsay__Markov_sgs_KMS_symm_weight}.
\item \label{enu:GL_i_Prop_2_10}Let $a,b\in\Aa$. All scalars $\tr\left(\mathfrak{i}^{(p)}(a)\cdot\mathfrak{i}^{(q)}(b)\right)$,
for conjugate exponents $p,q\in\left[1,\infty\right]$ such that $a\in\mathcal{M}^{(p)}$
and $b\in\mathcal{M}^{(q)}$, are equal \citep[Proposition 2.10]{Goldstein_Lindsay__Markov_sgs_KMS_symm_weight}.
\item \label{enu:GL_i_Prop_2_11}For every $p\in[1,\infty)$, $\mathfrak{i}^{(p)}(\mathcal{M}_{\infty}\cap\Aa_{+})$
is dense in $\Lp{\Aa}_{+}$ \citep[Proposition 2.11 (b)]{Goldstein_Lindsay__Markov_sgs_KMS_symm_weight}.
\item \label{enu:GL_i_Prop_2_12}For every $p\in[1,\infty)$, the map $\mathfrak{i}^{(p)}:\mathcal{M}^{(p)}\to\Lp{\Aa}$,
viewed as a densely defined operator from $\Aa$ with the ultraweak
topology to $\Lp{\Aa}$ with the weak topology, is closable. The pertinent
closure, which we denote by $\overline{\mathfrak{i}^{(p)}}$, is injective
and positivity preserving \citep[Proposition 2.12]{Goldstein_Lindsay__Markov_sgs_KMS_symm_weight}.
\item \label{enu:GL_i_Prop_2_13}For every $a\in\mathcal{N}^{*}$ and $b\in\mathcal{M}_{\infty}$
we have $\varphi(a\sigma_{-i/2}(b))=\tr\left(a\cdot\mathfrak{i}^{(1)}(b)\right)$
\citep[Proposition 2.13 (c)]{Goldstein_Lindsay__Markov_sgs_KMS_symm_weight}.
\end{enumerate}
\end{prop}

Remark that for $p,q\in\left[1,\infty\right]$ conjugate exponents,
$x\in\Lp{\Aa}$ is positive if and only if $\tr(x\cdot y)\ge0$ for
all $y\in\Lq{\Aa}_{+}$. Consequently, if $p\in[1,\infty)$ and $a\in\mathcal{M}^{(p)}$,
then $a\ge0$ if (and only if) $\mathfrak{i}^{(p)}(a)\ge0$, because
for every $b\in\mathcal{M}_{+}$, $0\le\tr\left(\mathfrak{i}^{(p)}(a)\cdot\mathfrak{i}^{(q)}(b)\right)=\tr\left(a\cdot\mathfrak{i}^{(1)}(b)\right)$,
thus $\om(a)\ge0$ for every $\om\in\Aa_{*}^{+}$ (use \prettyref{enu:GL_i_Prop_2_10},
\prettyref{enu:GL_i_Prop_2_11} above).
\begin{convention*}
From now on we will tacitly identify: 
\begin{itemize}
\item the $*$-algebra $\Linfty{\Aa}$ with $\Aa$ (that is, with its canonical
image in $\mathbb{A}$);
\item $\Lp{\Aa}^{*}$ with $\Lq{\Aa}$ for $p\in[1,\infty),q\in(1,\infty]$
conjugate exponents by the correspondence sending $b\in\Lq{\Aa}$
to the functional $\Lp{\Aa}\ni a\mapsto\tr(a\cdot b)$;
\item (in particular) the Banach space $\Lone{\Aa}$ with $\Aa_{*}$ by
the correspondence sending $x\in\Lone{\Aa}$ to the functional $\Linfty{\Aa}\ni y\mapsto\tr(xy)$;
\item the GNS representation $(\Ltwo{\Aa,\varphi},\text{the canonical left action of }\Aa,\gnsmap_{\varphi})$
with the semi-cyclic representation
\[
(\Ltwo{\Aa},\text{the action of }\Linfty{\Aa}\cong\Aa\text{ on }\Ltwo{\Aa}\text{ by multiplication on the left},\mathfrak{j}_{\varphi}^{(2)});
\]
indeed, we have $x\cdot\mathfrak{j}_{\varphi}^{(2)}(a)=\mathfrak{j}_{\varphi}^{(2)}(xa)$
for every $x\in\Aa\cong\Linfty{\Aa}$ and $a\in\mathcal{N}$, $\mathfrak{j}_{\varphi}^{(2)}$
has dense range \citep[Proposition 2.11 (a)]{Goldstein_Lindsay__Markov_sgs_KMS_symm_weight},
and it satisfies $\Vert\mathfrak{j}_{\varphi}^{(2)}(a)\Vert_{2}=\varphi(a^{*}a)^{1/2}$
for all $a\in\mathcal{N}$ by \citep[Lemma 2.1 (c)]{Goldstein_Lindsay__Markov_sgs_KMS_symm_weight}
and the invariance of the $\Ltwo{\Aa}$-norm under the adjoint map, essentially following from the trace-like property of the functional $\tr$.
\end{itemize}
\end{convention*}

\subsection{\label{sub:Markov__KMS_sym__Dirichlet}Markov operators, KMS-symmetry
and Dirichlet forms}
In this subsection we discuss several notions and results from \citep[Sections 3--5]{Goldstein_Lindsay__Markov_sgs_KMS_symm_weight} that will be fundamental for us.
We continue to use the same notation as in the previous subsection,
so in particular, $\Aa$ and $\varphi$ are fixed.

Let $p\in[1,\infty)$. Write $\left[0,h^{1/p}\right]_{\Lp{\Aa}}:=\left\{ x\in\Lp{\Aa}:0\le x\le h^{1/p}\right\} $,
the right inequality being in the sense of unbounded positive selfadjoint
operators (recall that generally $h^{1/p}$ does not belong to $\pres{\tau}{}{\mathbb{A}}$!).
By \citep[Lemma 3.1, Proposition 3.2 and Corollary 3.4]{Goldstein_Lindsay__Markov_sgs_KMS_symm_weight},
we have 
\begin{equation}
\mathfrak{i}^{(p)}(\left[0,\one\right]_{\mathcal{M}^{(p)}})=\left[0,h^{1/p}\right]_{\Lp{\Aa}},\label{eq:key_convex_set}
\end{equation}
and this convex set is norm-closed in $\Lp{\Aa}$. Notice that the
linear span of $\left[0,h^{1/p}\right]_{\Lp{\Aa}}$ equals $\mathfrak{i}^{(p)}(\mathcal{M}^{(p)})$.
\begin{defn}[{\citep[p.~57]{Goldstein_Lindsay__Markov_sgs_KMS_symm_weight}}] \label{def:Markov,KMS}
Let $p\in[1,\infty)$ and let $S,T$ be linear operators on, respectively,
$\Lp{\Aa},\Aa$, not necessarily everywhere defined. Then
\begin{itemize}
\item $S$ is \emph{Markov with respect to $\varphi$} if $\left[0,h^{1/p}\right]_{\Lp{\Aa}}$
is contained in $D(S)$ and is invariant under $S$;
\item $T$ is \emph{Markov with respect to $\varphi$} if $\mathcal{M}\subseteq D(T)$
and $T\left(\left[0,\one\right]_{\Aa}\cap D(T)\right)\subseteq\left[0,\one\right]_{\Aa}$;
\item $T$ is \emph{KMS-symmetric with respect to $\varphi$} if
$\mathcal{M}\subseteq D(T)$ and
\[
\tr\left(Ta\cdot\mathfrak{i}^{(1)}(b)\right)=\tr\left(\mathfrak{i}^{(1)}(a)\cdot Tb\right)\qquad(\forall_{a,b\in\mathcal{M}});
\]
\item $T$ is \emph{$p$-integrable with respect to $\varphi$}
if $\mathcal{M}\subseteq D(T)$, $T(\mathcal{M})\subseteq\mathcal{M}^{(p)}$
and the map $\mathfrak{i}^{(p)}(a)\mapsto\mathfrak{i}^{(p)}(Ta)$,
$a\in\mathcal{M}$, is continuous (with respect to the $\Lp{\Aa}$-norm).
When this is the case, we denote by $\widetilde{T}^{(p)}$ the unique
bounded extension of that map to $\Lp{\Aa}$.
\end{itemize}
\end{defn}

Remark that when $T$ is everywhere defined, it is Markov if and only
if it is positive (= positivity preserving) and contractive. Also,
for $p=2$, if $S$ is symmetric and Markov, then for every $a,b\in\mathcal{M}^{(2)}$
we have
\begin{equation}
\begin{split}\tr\left(S\mathfrak{i}^{(2)}(a)\cdot\mathfrak{i}^{(2)}(b)\right) & =\left\langle S\mathfrak{i}^{(2)}(a),\mathfrak{i}^{(2)}(b)^{*}\right\rangle =\left\langle \mathfrak{i}^{(2)}(a),S(\mathfrak{i}^{(2)}(b)^{*})\right\rangle \\
 & =\left\langle \mathfrak{i}^{(2)}(a),(S\mathfrak{i}^{(2)}(b))^{*}\right\rangle =\tr\left(\mathfrak{i}^{(2)}(a)\cdot S\mathfrak{i}^{(2)}(b)\right),
\end{split}
\label{eq:L2_symmetry_tr}
\end{equation}
where we used the fact that $S$ is positivity preserving, thus adjoint
preserving, on $\mathfrak{i}^{(2)}(\mathcal{M}^{(2)})$.

We make a short detour to discuss the general theory of quadratic
forms \citep{Davies__semigroups,Kato__perturb_lin_op}. Fix a Hilbert
space $\H$. A non-negative \emph{quadratic form} on $\H$ arises from a semi-inner
product $Q:D(Q)\times D(Q)\to\C$ for a linear subspace $D(Q)$ of
$\H$. Specifically, $Q$ is assumed to be linear in the first variable, hermitian: $Q(\eta,\z)=\overline{Q(\z,\eta)}$
for all $\z,\eta\in D(Q)$, and positive semi-definite: $Q(\z,\z)\ge0$
for all $\z\in D(Q)$. All quadratic forms in this paper will be non-negative.
By the polarisation identity, $Q$ is determined by the function $Q':D(Q)\to[0,\infty)$
given by $Q'\z:=Q(\z,\z)$, $\z\in D(Q)$ (usually it is $Q'$ which is called the quadratic form). It is also useful to consider
the function $Q'':\H\to\left[0,\infty\right]$ given by 
\[
Q''\z:=\begin{cases}
Q'(\z) & \z\in D(Q)\\
\infty & \text{else}
\end{cases}\qquad(\z\in\H).
\]
The form $Q$ is said to be \emph{densely defined} if $D(Q)$ is dense in $\H$, and \emph{closed} if for every sequence $\left(\z_{n}\right)_{n=1}^{\infty}$
in $D(Q)$ that converges to $\z\in\H$ and satisfies that $Q'(\z_{n}-\z_{m})\xrightarrow[n,m\to\infty]{}0$
we have $\z\in D(Q)$ and $Q'(\z_{n}-\z)\xrightarrow[n\to\infty]{}0$.
This is equivalent to $Q''$ being lower semicontinuous.

Using the $Q''$ `face' of quadratic forms, one can add and compare
them, and form limits of ascending sequences of them. For instance,
if $Q_{1},Q_{2}$ are quadratic forms on $\H$, then the quadratic
form $Q_{1}+Q_{2}$ on $\H$ is given by $(Q_{1}+Q_{2})''=Q_{1}''+Q_{2}''$.

There are $1-1$ correspondences between: (generally unbounded) positive
selfadjoint operators $A$ on $\H$; $C_{0}$-semigroups $\left(S_{t}\right)_{t\ge0}$
of selfadjoint contractions on $\H$; and closed, densely-defined
(non-negative) quadratic forms $Q$ on $\H$. They are given by $S_{t}=e^{-tA}$
($t\geq 0$) -- with the equality understood either via the functional calculus or in the standard semigroup theory sense, and $Q=Q_{A}$ where $Q_{A}:=\left\Vert A^{1/2}\cdot\right\Vert ^{2}$,
that is, $D(Q_{A})=D(A^{1/2})$ and $Q_{A}'(\z)=\left\Vert A^{1/2}\z\right\Vert ^{2}$
for all $\z\in D(A^{1/2})$.

In the sequel we will not distinguish between $Q,Q',Q''$, and denote
these three maps by $Q$.

Finally, having completed the detour, we introduce Dirichlet forms
using the previous notation; so again, $\Aa$ and $\varphi$ are fixed.
Let $\pi_{I}:\Ltwo{\Aa}\to\left[0,h^{1/2}\right]_{\Ltwo{\Aa}}$ be
the nearest-point projection onto the closed convex set $\left[0,h^{1/2}\right]_{\Ltwo{\Aa}}$
(see \prettyref{eq:key_convex_set}). A closed densely-defined quadratic form $Q$ on $\Ltwo{\Aa}$
is said to be \emph{Dirichlet with respect to $\varphi$} if $Q\circ\pi_{I}\le Q$
\citep[p.~62]{Goldstein_Lindsay__Markov_sgs_KMS_symm_weight}.

At this point we refer the reader to the Appendix, where we point
out a gap in the proof of \citep[Theorem 4.7]{Goldstein_Lindsay__Markov_sgs_KMS_symm_weight}
and propose two solutions under additional hypotheses. The first one
is more practical for our purposes. It is used to prove \prettyref{cor:corres_compl_Markov},
which plays an important role in our paper.

\subsection{New results concerning $L^p$-embeddings\label{sub:prelim_new_results}}

We now present several  new results pertaining to the material
of the previous two subsections. %Once again, the same notation is used.

The first lemma is not used later on, but it complements Lemma \ref{lem:i_2_gns}.

\begin{lem}
\label{lem:Goldstein_Lindsay_i_1}Let $b,c\in\mathcal{N}$, and consider
$a:=b^{*}c\in\mathcal{M}$. Then $\mathfrak{i}^{(1)}(a)\in\Lone{\Aa}$,
viewed as an element of $\Aa_{*}$, equals $\om_{J\gnsmap(b),J\gnsmap(c)}$.
\end{lem}

\begin{proof}
Denote by $\om$ the element of $\Aa_{*}$ that corresponds to $\mathfrak{i}^{(1)}(a)$.
Assume for the moment that $b,c\in\mathcal{T}$. Write $S:=J\nabla^{1/2}=\nabla^{-1/2}J$.
By \prettyref{prop:GL_i} \prettyref{enu:GL_i_Prop_2_13}, for every
$x\in\mathcal{N}\cap\mathcal{N}^{*}$,
\[
\begin{split}\om(x) & =\tr(x\cdot\mathfrak{i}^{(1)}(a))=\varphi(x\sigma_{-i/2}(a))=\left\langle \gnsmap(\sigma_{-i/2}(a)),\gnsmap(x^{*})\right\rangle =\left\langle \nabla^{1/2}\gnsmap(a),S\gnsmap(x)\right\rangle \\
 & =\left\langle \nabla^{1/2}\gnsmap(a),\nabla^{-1/2}J\gnsmap(x)\right\rangle =\left\langle \gnsmap(a),J\gnsmap(x)\right\rangle =\left\langle \gnsmap(x),J\gnsmap(a)\right\rangle \\
 & =\left\langle JbJ\gnsmap(x),J\gnsmap(c)\right\rangle =\left\langle xJ\gnsmap(b),J\gnsmap(c)\right\rangle .
\end{split}
\]
Since $\om$ is normal and $\mathcal{N}\cap\mathcal{N}^{*}$ is ultraweakly
dense in $\Aa$, we get $\om=\om_{J\gnsmap(b),J\gnsmap(c)}$.

For the general case of $b,c\in\mathcal{N}$, let $\left(b_{n}\right)_{n=1}^{\infty}$,
$\left(c_{n}\right)_{n=1}^{\infty}$ be bounded sequences in $\mathcal{T}$
such that $b_{n}\xrightarrow[n\to\infty]{}b$, $c_{n}\xrightarrow[n\to\infty]{}c$
in the strong operator topology and $\gnsmap(b_{n})\xrightarrow[n\to\infty]{}\gnsmap(b)$,
$\gnsmap(c_{n})\xrightarrow[n\to\infty]{}\gnsmap(c)$ (\citep[Theorem VI.1.26]{Takesaki__book_vol_2}
and \citep[Section 10.21, Corollary 2]{Stratila_Zsido__lectures_vN};
or \citep[p.~29, (16)]{Stratila__mod_thy}). Then $a_{n}:=b_{n}^{*}c_{n}\xrightarrow[n\to\infty]{}b^{*}c=a$
ultraweakly and $\mathfrak{i}^{(1)}(a_{n})=\om_{J\gnsmap(b_{n}),J\gnsmap(c_{n})}\to\om_{J\gnsmap(b),J\gnsmap(c)}$
in norm. Since $\mathfrak{i}^{(1)}$ is ultraweak\textendash $\sigma(\Aa_{*},\Aa)$-closable
(\prettyref{prop:GL_i} \prettyref{enu:GL_i_Prop_2_12}), $\mathfrak{i}^{(1)}(a)=\om_{J\gnsmap(b),J\gnsmap(c)}$.
\end{proof}

The next lemma's assertion is precisely that of \citep[Lemma 2.5]{Goldstein_Lindsay__Markov_sgs_KMS_symm_weight}
but it uses weaker assumptions.

\begin{lem}
\label{lem:h_impl_sigma}Let $\a\in\C$ be such that $s:=\Re\a\geq0$,
and let $a\in D(\sigma_{i\a})=D(\sigma_{is})\subseteq\Aa$ be such
that $a,\sigma_{i\a}(a)^{*}\in\mathcal{N}^{(1/s)}$. Then $\overline{ah^{\a}}=h^{\a}\sigma_{i\a}(a)$.
\end{lem}

\begin{proof} By \citep[9.24, Proposition]{Stratila_Zsido__lectures_vN},
$a\in D(\sigma_{is})$ if and only if $ah^{s}\subseteq h^{s}b$ for
some $b\in\Aa$, in which case $b=\sigma_{is}(a)$, thus $ah^{\a}\subseteq h^{\a}\sigma_{i\a}(a)$.
Now, follow the proof of \citep[Lemma 2.5]{Goldstein_Lindsay__Markov_sgs_KMS_symm_weight}:
since $a,\sigma_{i\a}(a)^{*}\in\mathcal{N}^{(1/s)}$, $ah^{\a}$ is
closable, $\overline{ah^{\a}}\subseteq h^{\a}\sigma_{i\a}(a)$, and
both sides are in $\pres{\tau}{}{\mathbb{A}}$, hence they are equal.
\end{proof}

The next lemma connects the map $\mathfrak{i}^{(2)}$ introduced above with the GNS map $\gnsmap$.

\begin{lem}
\label{lem:i_2_gns}\mbox{}
\begin{enumerate}
\item \label{enu:i_2_gns__1}For every $a=b^{*}c$ with $b,c\in\mathcal{N}^{(4)}$
(so $a\in\mathcal{M}^{(2)}$) such that $a\in D(\sigma_{-i/4})$ and
$\sigma_{-i/4}(a)\in\mathcal{N}$ we have $\mathfrak{i}^{(2)}(a)=\gnsmap(\sigma_{-i/4}(a))$.
\item \label{enu:i_2_gns__2}For every $a\in\mathcal{M}$ we have $\mathfrak{i}^{(2)}(a)=\nabla^{1/4}\gnsmap(a)$.
In particular, if $a\in\mathcal{M}\cap\mathcal{T}$ then $\mathfrak{i}^{(2)}(a)=\gnsmap(\sigma_{-i/4}(a))$.
\end{enumerate}
\end{lem}

\begin{proof}
\prettyref{enu:i_2_gns__1} We have $\mathfrak{i}^{(2)}(a)\supseteq h^{1/4}ah^{1/4}$
because \prettyref{prop:GL_i} \prettyref{enu:GL_i_Lem_2_9} implies
the relations $\mathfrak{i}^{(2)}(b^{*}c)\supseteq\mathfrak{j}^{(4)}(b)^{*}\mathfrak{j}^{(4)}(c)=h^{1/4}b^{*}\overline{ch^{1/4}}$.
On the other hand, since $\sigma_{-i/4}(a)\in\mathcal{N}\subseteq\mathcal{N}^{(4)}$
and $a^{*}\in\mathcal{M}^{(2)}\subseteq\mathcal{N}^{(4)}$, \prettyref{lem:h_impl_sigma}
implies that $h^{1/4}a=\overline{\sigma_{-i/4}(a)h^{1/4}}$, thus
$\mathfrak{i}^{(2)}(a)\supseteq\sigma_{-i/4}(a)h^{1/2}$. Hence \prettyref{rem:tau_meas_poly}
entails that $\mathfrak{i}^{(2)}(a)=\overline{\sigma_{-i/4}(a)h^{1/2}}=\mathfrak{j}^{(2)}(\sigma_{-i/4}(a))=\gnsmap(\sigma_{-i/4}(a))$.

\prettyref{enu:i_2_gns__2} Let $a\in\mathcal{M}$, and assume that
$a=b^{*}c$ for $b,c\in\mathcal{N}$. As in the proof of \prettyref{lem:Goldstein_Lindsay_i_1},
find bounded sequences $\left(b_{n}\right)_{n=1}^{\infty}$, $\left(c_{n}\right)_{n=1}^{\infty}$
in $\mathcal{T}$ that satisfy $b_{n}\xrightarrow[n\to\infty]{}b$,
$c_{n}\xrightarrow[n\to\infty]{}c$ in the strong operator topology and $\gnsmap(b_{n})\xrightarrow[n\to\infty]{}\gnsmap(b)$,
$\gnsmap(c_{n})\xrightarrow[n\to\infty]{}\gnsmap(c)$. Then the sequence
$\left(b_{n}^{*}c_{n}\right)_{n=1}^{\infty}$ in $\mathcal{M}_{\infty}$
converges ultraweakly to $a$ and $\gnsmap(b_{n}^{*}c_{n})=b_{n}^{*}\gnsmap(c_{n})\xrightarrow[n\to\infty]{}b^{*}\gnsmap(c)=\gnsmap(a)$
weakly. Since $J\nabla^{1/2}\gnsmap(b_{n}^{*}c_{n})=\gnsmap(c_{n}^{*}b_{n})=c_{n}^{*}\gnsmap(b_{n})$
for all $n\in\N$, the sequence $\left(\nabla^{1/2}\gnsmap(b_{n}^{*}c_{n})\right)_{n=1}^{\infty}$
is bounded, as was $\left(\gnsmap(b_{n}^{*}c_{n})\right)_{n=1}^{\infty}$, hence so is $\left(\nabla^{1/4}\gnsmap(b_{n}^{*}c_{n})\right)_{n=1}^{\infty}$
by the Phragmen\textendash Lindel\"{o}f three lines theorem. Thus,
$\nabla^{1/4}\gnsmap(b_{n}^{*}c_{n})\xrightarrow[n\to\infty]{}\nabla^{1/4}\gnsmap(a)$
weakly. Since $\nabla^{1/4}\gnsmap(b_{n}^{*}c_{n})=\mathfrak{i}^{(2)}(b_{n}^{*}c_{n})$
for all $n\in\N$ by \prettyref{enu:i_2_gns__1}, \prettyref{prop:GL_i}
\prettyref{enu:GL_i_Prop_2_12} implies that $\mathfrak{i}^{(2)}(a)=\nabla^{1/4}\gnsmap(a)$.
\end{proof}

Recall the (generalised) operators $k_\psi$ defined in Subsection \ref{sub:prelim_Haagerup_Lp}. The last statement below can be also deduced from a general argument concerning the $L^p$-extensions of automorphisms preserving the fixed reference weight; for extending yet more general maps to Haagerup $L^p$-spaces in the state context we refer to \citep[Section 5]{HJXReduction} (the arguments there can be also adapted to the weight case).

\begin{lem}
\label{lem:k_psi_mod_aut_grp}
\begin{enumerate}
\item \label{enu:k_psi_mod_aut_grp_1}For every normal weight $\psi$ on
$\Aa$ and $t\in\R$ we have $k_{\psi\circ\sigma_{t}}=h^{-it}k_{\psi}h^{it}$
in $\hat{\mathbb{A}}_{+}$.
\item \label{enu:k_psi_mod_aut_grp_2}$\tr$ is $\Ad{h^{it}}$-invariant.
\item \label{enu:k_psi_mod_aut_grp_3}Let $p\in[2,\infty)$, $q\in[1,\infty)$
and $t\in\R$. The sets $\mathcal{N}^{(p)},\mathcal{M}^{(q)},D(\overline{\mathfrak{i}^{(q)}})$
are invariant under $\sigma_{t}$, and for every $a\in\mathcal{N}^{(p)}$
and $b\in D(\overline{\mathfrak{i}^{(q)}})$ we have $\mathfrak{j}^{(p)}(\sigma_{t}(a))=h^{it}\mathfrak{j}^{(p)}(a)h^{-it}$
and $\overline{\mathfrak{i}^{(q)}}(\sigma_{t}(b))=h^{it}\overline{\mathfrak{i}^{(q)}}(b)h^{-it}$.
\item \label{enu:k_psi_mod_aut_grp_4}Let $p\in[1,\infty)$. For $t\in\R$
define $U(t):\Lp{\Aa}\to\Lp{\Aa}$ by $x\mapsto h^{it}xh^{-it}$,
$x\in\Lp{\Aa}$. Then $\left(U(t)\right)_{t\in\R}$ is a (well-defined)
$C_{0}$-group of surjective isometries.
\end{enumerate}
\end{lem}

\begin{proof}
\prettyref{enu:k_psi_mod_aut_grp_1} Let $\psi$ be a normal weight
on $\Aa$ and $t\in\R$. We claim that $\widetilde{\psi\circ\sigma_{t}}=\widetilde{\psi}\circ\Ad{h^{it}}$.
Once this is established, we get the following equalities in $\hat{\mathbb{A}}_{+}$:
\[
k_{\psi\circ\sigma_{t}}=\frac{\mathrm{d}(\widetilde{\psi\circ\sigma_{t}})}{\mathrm{d}\tau}=\frac{\mathrm{d}(\widetilde{\psi}\circ\Ad{h^{it}})}{\mathrm{d}\tau}=\Ad{h^{-it}}(\frac{\mathrm{d}\widetilde{\psi}}{\mathrm{d}\tau})=\Ad{h^{-it}}(k_{\psi}).
\]

To prove the claim, write $\pi$ for the canonical normal injection
of $\Aa$ into $\mathbb{A}$. Recall that $\theta=\left(\theta_{s}\right)_{s\in\R}$
is the action of $\R$ on $\mathbb{A}$ dual to $\sigma$, and denote
by $T$ the canonical n.s.f.~operator-valued weight from $\mathbb{A}$
to $\pi(\Aa)$ given by $Tx=\int_{\R}\theta_{s}(x)\d s$, $x\in\mathbb{A}_{+}$.
Then for every $x\in\mathbb{A}_{+}$, since $\theta_{s}(h^{it})=e^{-ist}h^{it}$,
we have 
\begin{equation}
T(h^{it}xh^{-it})=\int_{\R}\theta_{s}(h^{it}xh^{-it})\d s=\int_{\R}(h^{it}\theta_{s}(x)h^{-it})\d s=h^{it}(Tx)h^{-it}\label{eq:k_psi_mod_aut_grp_1}
\end{equation}
(the equalities being in $\hat{\mathbb{A}}_{+}$, thus in $\widehat{\pi(\Aa)}_{+}$).
By definition, we have $\widetilde{\psi\circ\sigma_{t}}=\left(\psi\circ\sigma_{t}\circ\pi^{-1}\right)^{\wedge}\circ T$,
where the hat symbol denotes the canonical extension to $\widehat{\pi(\Aa)}_{+}$. But
$\sigma_{t}\circ\pi^{-1}=\pi^{-1}\circ\Ad{h^{it}}|_{\pi(\Aa)}$, so
\prettyref{eq:k_psi_mod_aut_grp_1} implies that 
\[
\widetilde{\psi\circ\sigma_{t}}=(\psi\circ\pi^{-1}\circ\Ad{h^{it}})^{\wedge}\circ T=\widehat{(\psi\circ\pi^{-1})}\circ T\circ\Ad{h^{it}}=\widetilde{\psi}\circ\Ad{h^{it}}.
\]

\prettyref{enu:k_psi_mod_aut_grp_2} For each $\psi\in\Aa_{*}^{+}$,
$\tr(\Ad{h^{-it}}(k_{\psi}))=\tr(k_{\psi\circ\sigma_{t}})=(\psi\circ\sigma_{t})(\one)=\psi(\one)=\tr(k_{\psi})$.

\prettyref{enu:k_psi_mod_aut_grp_3} Let $a\in\mathcal{N}^{(p)}$.
We have $\sigma_{t}(a)=h^{it}ah^{-it}$. Thus 
\begin{equation}
h^{\frac{1}{p}}\sigma_{t}(a)^{*}=h^{\frac{1}{p}}h^{it}a^{*}h^{-it}=h^{it}h^{\frac{1}{p}}a^{*}h^{-it}.\label{eq:k_psi_mod_aut_grp_2}
\end{equation}
Since $a\in\mathcal{N}^{(p)}$, namely $h^{\frac{1}{p}}a^{*}$ is
$\tau$-measurable, so is $h^{it}h^{\frac{1}{p}}a^{*}h^{-it}$, proving
that $\sigma_{t}(a)\in\mathcal{N}^{(p)}$. Now \prettyref{eq:k_psi_mod_aut_grp_2}
means that $\mathfrak{j}^{(p)}(\sigma_{t}(a))^{*}=h^{it}\mathfrak{j}^{(p)}(a)^{*}h^{-it}$,
and taking adjoints gives $\mathfrak{j}^{(p)}(\sigma_{t}(a))=h^{it}\mathfrak{j}^{(p)}(a)h^{-it}$. 

Take $b\in\mathcal{M}^{(q)}$. From the foregoing it is obvious that
$\sigma_{t}(b)\in\mathcal{M}^{(q)}$ and $\mathfrak{i}^{(q)}(\sigma_{t}(b))=h^{it}\mathfrak{i}^{(q)}(b)h^{-it}$.
More generally, if $b\in D(\overline{\mathfrak{i}^{(q)}})$, let $\left(b_{\l}\right)_{\l\in\mathcal{I}}$
be a net in $\mathcal{M}^{(q)}$ such that $b_{\l}\xrightarrow[\l\in\mathcal{I}]{}b$
ultraweakly and $\mathfrak{i}^{(q)}(b_{\l})\xrightarrow[\l\in\mathcal{I}]{}\overline{\mathfrak{i}^{(q)}}(b)$
weakly. Then $\sigma_{t}(b_{\l})\xrightarrow[\l\in\mathcal{I}]{}\sigma_{t}(b)$
ultraweakly and $\mathfrak{i}^{(q)}(\sigma_{t}(b_{\l}))=h^{it}\mathfrak{i}^{(q)}(b_{\l})h^{-it}\xrightarrow[\l\in\mathcal{I}]{}h^{it}\overline{\mathfrak{i}^{(q)}}(b)h^{-it}$
weakly by \prettyref{enu:k_psi_mod_aut_grp_2} (it is clear that $h^{it}\Lq{\Aa}h^{-it}=\Lq{\Aa}$).
Hence $\sigma_{t}(b)\in D(\overline{\mathfrak{i}^{(q)}})$ and $\overline{\mathfrak{i}^{(q)}}(\sigma_{t}(b))=h^{it}\overline{\mathfrak{i}^{(q)}}(b)h^{-it}$.

\prettyref{enu:k_psi_mod_aut_grp_4} Fix $p\in[1,\infty)$ and $x\in\Lp{\Aa}$.
For every $t\in\R$ we have $\left|h^{it}xh^{-it}\right|^{p}=h^{it}\left|x\right|^{p}h^{-it}$,
so \prettyref{enu:k_psi_mod_aut_grp_2} entails that $\left\Vert h^{it}xh^{-it}\right\Vert _{p}=\left\Vert x\right\Vert _{p}$.
For every $t\in\R$ define $U(t):\Lp{\Aa}\to\Lp{\Aa}$ by $x\mapsto h^{it}xh^{-it}$.
By the foregoing, $U(t)$ is an isometry from $\Lp{\Aa}$ onto itself.
Clearly, $\left(U(t)\right)_{t\in\R}$ is a group. It is also point\textendash weakly
continuous. Indeed, it $p=1$, then for every $b\in\Aa$ we have $\tr(h^{it}xh^{-it}b)=\tr(xh^{-it}bh^{it})=\tr(x\sigma_{-t}(b))$,
and the function $t\mapsto\om(\sigma_{-t}(b))$ is continuous for
all $\om\in\Aa_{*}$. If $p\in\left(1,\infty\right)$, write $q$
for the conjugate exponent of $p$. For every $a\in\mathcal{M}^{(p)}$
and $b\in\mathcal{M}\subseteq\mathcal{M}^{(q)}$, we have $\tr\left(\mathfrak{i}^{(p)}(\sigma_{t}(a))\cdot\mathfrak{i}^{(q)}(b)\right)=\tr\left(\sigma_{t}(a)\cdot\mathfrak{i}^{(1)}(b)\right)$
by \prettyref{prop:GL_i} \prettyref{enu:GL_i_Prop_2_10}, so that
$\R\ni t\mapsto\tr\left(\mathfrak{i}^{(p)}(\sigma_{t}(a))\cdot\mathfrak{i}^{(q)}(b)\right)$
is continuous. Since $\mathfrak{i}^{(p)}(\sigma_{t}(a))=U(t)(\mathfrak{i}^{(p)}(a))$
for all $t\in \mathbb{R}$ and $a\in \mathcal{M}^{(p)}$, $\mathfrak{i}^{(p)}(\mathcal{M}^{(p)})$ is
dense in $\Lp{\Aa}$ and $\mathfrak{i}^{(q)}(\mathcal{M})$ is dense
in $\Lq{\Aa}$ (\prettyref{prop:GL_i} \prettyref{enu:GL_i_Prop_2_11}),
the group $\left(U(t)\right)_{t\in\R}$ is point\textendash weakly
continuous, thus point\textendash norm continuous.
\end{proof}

We need one more approximation result, identifying a convenient core for the embedding map $\overline{\mathfrak{i}^{(p)}}$.

\begin{prop}
\label{prop:i_p_core}For every $p\in[1,\infty)$, the subspace $\mathcal{M}_{\infty}$
is an ultraweak\textendash weak core for $\overline{\mathfrak{i}^{(p)}}$.
To elaborate, for every $b\in D(\overline{\mathfrak{i}^{(p)}})$ and
$\e>0$ there exists a net $\left(b_{\l}\right)_{\l\in\mathcal{I}}$
in $\mathcal{M}_{\infty}$ that is bounded by $\left\Vert b\right\Vert $
such that $\left(\mathfrak{i}^{(p)}(b_{\l})\right)_{\l\in\mathcal{I}}$
is bounded by $\bigl\Vert\overline{\mathfrak{i}^{(p)}}(b)\bigr\Vert+\e$,
$b_{\l}\xrightarrow[\l\in\mathcal{I}]{}b$ in the $*$-strong operator topology and $\mathfrak{i}^{(p)}(b_{\l})\xrightarrow[\l\in\mathcal{I}]{}\overline{\mathfrak{i}^{(p)}}(b)$
in norm. If $b$ is selfadjoint or positive, $\left(b_{\l}\right)_{\l\in\mathcal{I}}$
can be chosen to consist of such operators as well.
\end{prop}

\begin{proof}
By convexity, it suffices to prove the assertion for weak convergence of $\left(\mathfrak{i}^{(p)}(b_{\l})\right)_{\l\in\mathcal{I}}$.

\emph{Step 1}: We claim that for every $a\in D(\overline{\mathfrak{i}^{(p)}})$
there is a sequence $\left(a_{n}\right)_{n=1}^{\infty}$ in $D(\overline{\mathfrak{i}^{(p)}})$,
all of whose elements are entire analytic with respect to $\sigma$, such
that $\left\Vert a_{n}\right\Vert \leq\left\Vert a\right\Vert $ and
$\bigl\Vert\overline{\mathfrak{i}^{(p)}}(a_{n})\bigr\Vert_{p}\leq\bigl\Vert\overline{\mathfrak{i}^{(p)}}(a)\bigr\Vert_{p}$
for every $n\in\N$, and such that $a_{n}\xrightarrow[n\to\infty]{}a$
in the $*$-strong operator topology and $\overline{\mathfrak{i}^{(p)}}(a_{n})\xrightarrow[n\to\infty]{}\overline{\mathfrak{i}^{(p)}}(a)$
in norm, and if $a$ selfadjoint or positive, so are $\left(a_{n}\right)_{n=1}^{\infty}$.

Indeed, fix $a\in D(\overline{\mathfrak{i}^{(p)}})$. Let $n\in\N$ and write
$a_{n}:=\frac{\sqrt{n}}{\sqrt{\pi}}\int_{\R}e^{-nt^{2}}\sigma_{t}(a)\d t$,
where the integral converges in the $*$-strong operator topology. Then $a_{n}$ is entire analytic
with respect to $\sigma$. Furthermore, by \prettyref{lem:k_psi_mod_aut_grp}
\prettyref{enu:k_psi_mod_aut_grp_3} and \prettyref{enu:k_psi_mod_aut_grp_4},
$\bigl\Vert\overline{\mathfrak{i}^{(p)}}(\sigma_{t}(a))\bigr\Vert_{p}=\bigl\Vert\overline{\mathfrak{i}^{(p)}}(a)\bigr\Vert_{p}$
for all $t\in\R$ and $t\mapsto\overline{\mathfrak{i}^{(p)}}(\sigma_{t}(a))\in\Lp{\Aa}$
is norm continuous. Thus $\int_{\R}e^{-nt^{2}}\overline{\mathfrak{i}^{(p)}}(\sigma_{t}(a))\d t$
converges in norm. Hence $a_{n}\in D(\overline{\mathfrak{i}^{(p)}})$
and $\overline{\mathfrak{i}^{(p)}}(a_{n})=\frac{\sqrt{n}}{\sqrt{\pi}}\int_{\R}e^{-nt^{2}}\overline{\mathfrak{i}^{(p)}}(\sigma_{t}(a))\d t$.
In particular, $\bigl\Vert\overline{\mathfrak{i}^{(p)}}(a_{n})\bigr\Vert_{p}\leq\bigl\Vert\overline{\mathfrak{i}^{(p)}}(a)\bigr\Vert_{p}$.
The sequence $\left(a_{n}\right)_{n=1}^{\infty}$ has the desired
properties.

Before proceeding we make the following simple observation. Let $b\in D(\overline{\mathfrak{i}^{(p)}})$
and $e,f\in\mathcal{T}$. Then $f^{*}be\in\mathcal{M}$ satisfies
$\mathfrak{i}^{(p)}(f^{*}be)=\sigma_{\frac{i}{2p}}(f)^{*}\cdot\overline{\mathfrak{i}^{(p)}}(b)\cdot\sigma_{\frac{i}{2p}}(e)$.
Indeed, for every $d\in\mathcal{N}^{(2p)}$ we have $de\in\mathcal{N}\subseteq\mathcal{N}^{(2p)}$
and $\mathfrak{j}^{(2p)}(de)^{*}=h^{\frac{1}{2p}}e^{*}d^{*}=\overline{\sigma_{-\frac{i}{2p}}(e^{*})h^{\frac{1}{2p}}}d^{*}\supseteq\sigma_{-\frac{i}{2p}}(e^{*})h^{\frac{1}{2p}}d^{*}=\sigma_{-\frac{i}{2p}}(e^{*})\mathfrak{j}^{(2p)}(d)^{*}$
by \citep[Lemma 2.5]{Goldstein_Lindsay__Markov_sgs_KMS_symm_weight}
(or the more general \prettyref{lem:h_impl_sigma} above), hence $\mathfrak{j}^{(2p)}(de)^{*}=\sigma_{-\frac{i}{2p}}(e^{*})\cdot\mathfrak{j}^{(2p)}(d)^{*}$
by \prettyref{rem:tau_meas_poly}. Using \prettyref{prop:GL_i} \prettyref{enu:GL_i_Lem_2_9},
this entails that for all $b\in\mathcal{M}^{(p)}=D(\mathfrak{i}^{(p)})$
we have $\mathfrak{i}^{(p)}(f^{*}be)=\sigma_{\frac{i}{2p}}(f)^{*}\cdot\mathfrak{i}^{(p)}(b)\cdot\sigma_{\frac{i}{2p}}(e)$.
The observation follows for all $b\in D(\overline{\mathfrak{i}^{(p)}})$
from the definition of $\overline{\mathfrak{i}^{(p)}}$. 

\emph{Step 2}: We claim that for every $b\in D(\overline{\mathfrak{i}^{(p)}})$
that is entire analytic with respect to $\sigma$ and every $\delta>0$ there
exists a net $\left(b_{\l}\right)_{\l\in\mathcal{I}}$ in $\mathcal{M}_{\infty}$
that is bounded by $\left\Vert b\right\Vert $ such that $\left(\mathfrak{i}^{(p)}(b_{\l})\right)_{\l\in\mathcal{I}}$
is bounded by $(1+\delta)^{4}\bigl\Vert \overline{\mathfrak{i}^{(p)}}(b)\bigr\Vert _{p}$,
$b_{\l}\xrightarrow[\l\in\mathcal{I}]{}b$ in the $*$-strong operator topology and $\left(\mathfrak{i}^{(p)}(b_{\l})\right)_{\l\in I}\xrightarrow[\l\in\mathcal{I}]{}\overline{\mathfrak{i}^{(p)}}(b)$
weakly, and if $b$ is selfadjoint or positive, so are $\left(b_{\l}\right)_{\l\in\mathcal{I}}$.

Fix $b,\delta$ as above. By \prettyref{lem:Str_Terp_approx} there is a net $\left(e_{\l}\right)_{\l\in\mathcal{I}}$
in $\mathcal{T}$ that is bounded by $1$ and that converges
to $\one$ in the $*$-strong operator topology such that the nets $\bigl(\sigma_{\pm\frac{i}{2p}}(e_{\l})\bigr)_{\l\in\mathcal{I}}$
are bounded by $1+\delta$.
Then $\left(b_{\l}\right)_{\l\in\mathcal{I}}:=\left(e_{\l}^{*}e_{\l}be_{\l}^{*}e_{\l}\right)_{\l\in\mathcal{I}}$
is a net in $\mathcal{M}_{\infty}$ that is bounded by $\left\Vert b\right\Vert $
and that converges  to $b$ in the $*$-strong operator topology, and by the preceding paragraph,
$\left(\mathfrak{i}^{(p)}(b_{\l})\right)_{\l\in\mathcal{I}}$ is bounded
by $(1+\delta)^{4}\bigl\Vert\overline{\mathfrak{i}^{(p)}}(b)\bigr\Vert$.
If $p > 1$ write $q\in(1,\infty)$ for the conjugate exponent of $p$. For every
$c\in\mathcal{M}$,
\[
\tr\left(\mathfrak{i}^{(p)}(b_{\l})\cdot\mathfrak{i}^{(q)}(c)\right)=\tr\left(b_{\l}\cdot\mathfrak{i}^{(1)}(c)\right)\xrightarrow[\l\in\mathcal{I}]{}\tr\left(b\cdot\mathfrak{i}^{(1)}(c)\right)=\tr\bigl(\overline{\mathfrak{i}^{(p)}}(b)\cdot\mathfrak{i}^{(q)}(c)\bigr)
\]
by \prettyref{prop:GL_i} \prettyref{enu:GL_i_Prop_2_10}. As before,
density (\prettyref{prop:GL_i} \prettyref{enu:GL_i_Prop_2_11}) implies that $\mathfrak{i}^{(p)}(b_{\l})\xrightarrow[\l\in\mathcal{I}]{}\overline{\mathfrak{i}^{(p)}}(b)$
weakly. 
If $p = 1$, assume also that $\sigma_{\pm\frac{i}{2}}(e_{\l}) \xrightarrow[\l\in\mathcal{I}]{} \one$ in the $*$-strong operator topology (\prettyref{lem:Str_Terp_approx}). 
For every $\l\in\mathcal{I}$ we have
$\mathfrak{i}^{(1)}(b_{\l}) = 
\sigma_{-\frac{i}{2}}(e_{\l}^* e_{\l})\cdot\overline{\mathfrak{i}^{(1)}}(b)\cdot\sigma_{\frac{i}{2}}(e_{\l}^* e_{\l})$ in $\Lone{\Aa}$; equivalently, $\mathfrak{i}^{(1)}(b_{\l}) = 
(\sigma_{-\frac{i}{2}}(e_{\l}^* e_{\l})) (\overline{\mathfrak{i}^{(1)}}(b)) (\sigma_{\frac{i}{2}}(e_{\l}^* e_{\l}))$ as elements of $\Aa_*$.
Consequently, $\mathfrak{i}^{(1)}(b_{\l})\xrightarrow[\l\in\mathcal{I}]{}\overline{\mathfrak{i}^{(1)}}(b)$ in norm because $\sigma_{-\frac{i}{2}}(e_{\l}^* e_{\l}) \xrightarrow[\l \to \mathcal{I}]{}\one$ ultrastrongly.
This completes the proof. (Remark: without needing $\left(b_{\l}\right)_{\l\in\mathcal{I}}$
to be selfadjoint or positive if $b$ is, taking $\left(e_{\l}^{2}be_{\l}\right)_{\l\in\mathcal{I}}$
would suffice.)
\end{proof}

The last proposition gives a neat description of the `interval' that plays a key role when Dirichlet forms are introduced.

\begin{cor}
\label{cor:i_p_core}
\begin{enumerate}
\item \label{enu:i_p_core__1}For every $p\in[1,\infty)$ we have $D(\overline{\mathfrak{i}^{(p)}})\cap\Aa_{+}\subseteq D(\mathfrak{i}^{(p)})=\mathcal{M}^{(p)}$.
\item \label{enu:i_p_core__2}The closed convex set $\left[0,h^{1/2}\right]_{\Ltwo{\Aa}}=\mathfrak{i}^{(2)}\left(\left[0,\one\right]_{\mathcal{M}^{(2)}}\right)$
is equal to \[\overline{\nabla^{1/4}\gnsmap\left(\left[0,\one\right]_{\mathcal{M}_{\infty}}\right)}=\overline{\nabla^{1/4}\gnsmap\left(\left[0,\one\right]_{\mathcal{M}}\right)}.\]
\end{enumerate}
\end{cor}

\begin{proof}
\prettyref{enu:i_p_core__1} Let $b\in D(\overline{\mathfrak{i}^{(p)}})$
be positive and of norm $1$. By \prettyref{prop:i_p_core}, there
is a net $\left(b_{\l}\right)_{\l\in\mathcal{I}}$ of positive elements
in $\mathcal{M}_{\infty}\subseteq\mathcal{M}^{(p)}$ that are bounded
by $1$ such that $b_{\l}\xrightarrow[\l\in\mathcal{I}]{}b$ in the $*$-strong operator topology
and $\left(\mathfrak{i}^{(p)}(b_{\l})\right)_{\l\in I}\xrightarrow[\l\in\mathcal{I}]{}\overline{\mathfrak{i}^{(p)}}(b)$
weakly. The convex set \prettyref{eq:key_convex_set} is closed in
$\Lp{\Aa}$ in norm, equivalently weakly. Therefore, $\overline{\mathfrak{i}^{(p)}}(b)\in\mathfrak{i}^{(p)}(\left[0,\one\right]_{\mathcal{M}^{(p)}})$.
From injectivity of $\overline{\mathfrak{i}^{(p)}}$ we get $b\in\left[0,\one\right]_{\mathcal{M}^{(p)}}$.

\prettyref{enu:i_p_core__2} By \prettyref{lem:i_2_gns} \prettyref{enu:i_2_gns__2},
$\mathfrak{i}^{(2)}(a)=\nabla^{1/4}\gnsmap(a)$ for every $a\in\mathcal{M}$.
Hence we obtain the inclusions $\overline{\nabla^{1/4}\gnsmap\left(\left[0,\one\right]_{\mathcal{M}_{\infty}}\right)}\subseteq\overline{\nabla^{1/4}\gnsmap\left(\left[0,\one\right]_{\mathcal{M}}\right)}\subseteq\mathfrak{i}^{(2)}\left(\left[0,\one\right]_{\mathcal{M}^{(2)}}\right)$.
The reverse inclusion follows from \prettyref{prop:i_p_core} by convexity. 
\end{proof}
\begin{rem}
\label{rem:normal_2_integ}Let $p \in [1,\infty)$. If $T:\Aa\to\Aa$ is a normal operator
that is $p$-integrable with respect to $\varphi$, then $T$ leaves
$D(\overline{\mathfrak{i}^{(p)}})$ invariant and $\overline{\mathfrak{i}^{(p)}}(Ta)=\widetilde{T}^{(p)}(\overline{\mathfrak{i}^{(p)}}(a))$
for every $a\in D(\overline{\mathfrak{i}^{(p)}})$. Indeed, given
such $a$, find a net $\left(a_{\l}\right)_{\l\in\mathcal{I}}$ in
$\mathcal{M}$ such that $a_{\l}\xrightarrow[\l\in\mathcal{I}]{}a$
ultraweakly and $\mathfrak{i}^{(p)}(a_{\l})\xrightarrow[\l\in\mathcal{I}]{}\overline{\mathfrak{i}^{(p)}}(a)$
weakly by \prettyref{prop:i_p_core}. Then $Ta_{\l}\xrightarrow[\l\in\mathcal{I}]{}Ta$
ultraweakly and $\mathfrak{i}^{(p)}(Ta_{\l})=\widetilde{T}^{(p)}(\mathfrak{i}^{(p)}(a_{\l}))\xrightarrow[\l\in\mathcal{I}]{}\widetilde{T}^{(p)}(\overline{\mathfrak{i}^{(p)}}(a))$
weakly, so \prettyref{prop:GL_i} \prettyref{enu:GL_i_Prop_2_12}
gives the assertion.
\end{rem}

\subsection{\label{sub:prelim_LCQGs}Locally compact quantum groups}

We introduce locally compact quantum groups in the von Neumann algebraic
setting (\citep{Kustermans_Vaes__LCQG_von_Neumann}, see also \citep{Kustermans_Vaes__LCQG_C_star,Van_Daele__LCQGs}),
and describe some of their features.
\begin{defn}
A \emph{locally compact quantum group} %(in short, LCQG) 
is a virtual object studied via a pair
$\left(\mathsf{M},\Delta\right)$, where:
\begin{itemize}
\item $\mathsf{M}$ is a von Neumann algebra;
\item $\Delta:\mathsf{M}\to\mathsf{M}\tensorn\mathsf{M}$ is a \emph{co-multiplication}
(or \emph{co-product}), that is, a normal unital $*$-homomorphism
that is co-associative: $(\Delta\tensor\i)\circ\Delta=(\i\tensor\Delta)\circ\Delta$;
\item there exist n.s.f.~weights $\varphi,\psi$ on $\mathsf{M}$, called
the \emph{left and right Haar weights}, respectively, that satisfy
\[
\begin{gathered}\varphi((\om\tensor\i)\Delta(x))=\varphi(x)\om(\one)\quad\text{for all }\om\in\mathsf{M}_{*}^{+},x\in\mathsf{M}^{+}\text{ with }\varphi(x)<\infty,\\
\psi((\i\tensor\om)\Delta(x))=\psi(x)\om(\one)\quad\text{for all }\om\in\mathsf{M}_{*}^{+},x\in\mathsf{M}^{+}\text{ with }\psi(x)<\infty.
\end{gathered}
\]
\end{itemize}
We set $\Linfty{\G}:=\mathsf{M}$ and $\Lone{\G}:=\Linfty{\G}_{*}$.
\end{defn}

The left and right Haar weights are unique up to scaling. We write
$\Ltwo{\G}$ for the GNS Hilbert space of $\Linfty{\G}$ with respect
to $\varphi$, and always view $\Linfty{\G}$ as acting on $\Ltwo{\G}$.
Set $\nabla:=\nabla_{\varphi}$, $J:=J_{\varphi}$ and $\gnsmap:=\gnsmap_{\varphi}$.
There is an injective, positive, selfadjoint operator $\delta$ affiliated
with $\Linfty{\G}$, called the \emph{modular element}, such that
$\psi=\varphi_{\delta}$ in the sense of \citep{Vaes__Radon_Nikodym}.

Every locally compact quantum group $\G$ has a \emph{dual} locally compact quantum group, denoted by $\hat{\G}$. The
objects pertaining to it will be adorned with a `hat', e.g.~$\hat{\Delta},\hat{\varphi},\hat{\psi},\hat{\nabla},\hat{J},\hat{\gnsmap}$.
We will not elaborate on this construction, but mention that the GNS
Hilbert space of $\Linfty{\hat{\G}}$ with respect to $\hat{\varphi}$
can and will be naturally identified with $\Ltwo{\G}$. So we shall
always view $\Linfty{\hat{\G}}$ too as acting on $\Ltwo{\G}$. Importantly,
the `Pontryagin double dual property' holds: the dual of $\hat{\G}$
is $\G$.

There exists a distinguished (multiplicative) unitary $W\in\Linfty{\G}\tensorn\Linfty{\hat{\G}}$,
called the \emph{left regular representation}, which implements the
co-multiplication by $\Delta(x)=W^{*}(\one\tensor x)W$ for all $x\in\Linfty{\G}$.
The subspace $\Cz{\G}:=\overline{\{(\i\tensor\om)(W):\om\in\Lone{\hat{\G}}\}}^{\left\Vert \cdot\right\Vert }$
is a C$^{*}$-algebra, which is ultraweakly dense in $\Linfty{\G}$. Furthermore,
$\Delta(\Cz{\G})\subseteq\M{\Cz{\G}\tensormin\Cz{\G}}$, $W\in\M{\Cz{\G}\tensormin\Cz{\hat{\G}}}$
and $\hat{W}=\sigma(W^{*})$. The \emph{right regular representation}
is a (multiplicative) unitary $V\in\Linfty{\hat{\G}}'\tensorn\Linfty{\G}$
with similar properties, and in particular $\Delta(x)=V(x\tensor\one)V^{*}$
for all $x\in\Linfty{\G}$. In fact $\Linfty{\QG}'$, the commutant of $\Linfty{\QG}$ in $B(\Ltwo{\QG})$, admits its own co-product, arising from the \emph{commutant quantum group} $\QG'$ (so that $\Linfty{\QG'}:=\Linfty{\QG}'$) -- see \citep[Section 4]{Kustermans_Vaes__LCQG_von_Neumann} for details. All the objects associated with $\QG'$, such as the co-product, Haar weights, etc., will be adorned with primes.

An object of fundamental importance is the \emph{antipode} of $\G$.
It is a generally unbounded, ultraweakly closed, linear map $\Sant$ on
$\Linfty{\G}$ that is characterised by $\{(\i\tensor\om)(W):\om\in\Lone{\hat{\G}}\}$
being a $*$-ultrastrong core for $\Sant$ and the formula $\Sant\left((\i\tensor\om)(W)\right)=(\i\tensor\om)(W^{*})$
for every $\om\in\Lone{\hat{\G}}$. It has a `polar decomposition'
as $\Sant=\Rant\circ\tau_{-i/2}$, where the \emph{unitary antipode} $\Rant$
is a $*$-anti-automorphism of $\Linfty{\G}$ and $\tau_{-i}$ is the `analytic generator' of the \emph{scaling
group} $\tau=\left(\tau_{t}\right)_{t\in\R}$ of automorphisms
of $\Linfty{\G}$. The relavant objects are given by the formulas $\Rant(x)=\hat{J}x^{*}\hat{J}$ and
$\tau_{t}(x)=\hat{\nabla}^{it}x\hat{\nabla}^{-it}$ ($x\in\Linfty{\G}$,
$t\in\R$). Additionally, $\Rant$ and $\tau$ reduce to a $*$-anti-automorphism
and an automorphism group, respectively, of the C$^{*}$-algebra $\Cz{\G}$.
\begin{defn}
\label{def:rep}A unitary \emph{representation} of $\G$ on a Hilbert
space $\H$ is a unitary $U\in\M{\Cz{\G}\tensormin\K(\H)}$ that satisfies
$(\Delta\tensor\i)(U)=U_{13}U_{23}$, where we use the standard leg
numbering notation. A little more generally, for a C$^{*}$-algebra
$\mathsf{B}$, one defines similarly unitary representations $U\in\M{\Cz{\G}\tensormin\mathsf{B}}$
of $\G$.
\end{defn}

For instance, the left regular representation $W$ is a unitary representation
of $\G$ on $\Ltwo{\G}$. The trivial representation of $\G$ is $\one\in\M{\Cz{\G}}=\M{\Cz{\G}\tensor\C}$.
The above-mentioned property of the antipode extends as follows: for
every unitary representation $U$ of $\G$ on a Hilbert space $\H$
and $\om\in B(\H)_{*}$ we have $(\i\tensor\om)(U)\in D(\Sant)$ and $\Sant\left((\i\tensor\om)(U)\right)=(\i\tensor\om)(U^{*})$.

Locally compact quantum groups can be also studied via  a \emph{universal} C$^{*}$-algebra  $\CzU{\G}$ \citep{Kustermans__LCQG_universal},
equipped with a co-multiplication
$\Delta_{\mathrm{u}}:\CzU{\G}\to\M{\CzU{\G}\tensormin\CzU{\G}}$.
There is a unitary representation $\wW\in\M{\Cz{\G}\tensormin\CzU{\hat{\G}}}$
of $\G$, the right \emph{semi-universal} version of $W$, satisfying
$\overline{\{(\om\tensor\i)(\wW):\om\in\Lone{\G}\}}^{\left\Vert \cdot\right\Vert }=\CzU{\hat{\G}}$, determined by the following 
property: there is a $1-1$ correspondence between
unitary representations of $\G$ and representations of $\CzU{\hat{\G}}$
given as follows: a unitary representation $U$ of $\G$ on a Hilbert
space $\H$ is associated to the (unique) representation $\Phi$ of
$\CzU{\hat{\G}}$ on $\H$ by $U=(\i\tensor\Phi)(\wW)$, where we
view $\Phi$ as taking values in $\M{\K(\H)}\cong B(\H)$. There is also
a unitary $\Ww\in\M{\CzU{\G}\tensormin\Cz{\hat{\G}}}$, the left
semi-universal version of $W$, with a similar universality property,
and we have $\hat{\Ww}=\sigma(\wW^{*})$. There are also $\Vv,\vV$
for $V$.

Applying the (dual of the) above correspondence to the representation
$W$ and to the trivial representation gives, respectively, the \emph{reducing
morphism}, which is a surjective $*$-homomorphism $\Lambda:\CzU{\G}\to\Cz{\G}$
satisfying $(\Lambda\tensor\i)(\Ww)=W$, and the \emph{co-unit}, which
is a character $\epsilon\in\CzU{\G}^{*}$ satisfying $(\epsilon\tensor\i)(\Ww)=\one$.
We have $(\Lambda\tensor\Lambda)\circ\Delta_{\mathrm{u}}=\Delta\circ\Lambda$
and 
\begin{equation}
(\i\tensor\Lambda)(\Delta_{\mathrm{u}}(x))=\Ww^{*}(\one\tensor\Lambda(x))\Ww\qquad(\forall_{x\in\CzU{\G}}),\label{eq:Delta_u_Ww}
\end{equation}
where the right-hand side makes sense when considering $\Cz{\G}, \Cz{\hat{\G}}\subseteq B(\Ltwo{\G})$.
In this point of view we also have $\Ww\in\M{\CzU{\G}\tensormin\K(\Ltwo{\G})}$.

The unitary antipode and scaling group of $\G$, and thus also its antipode,
lift to objects acting on $\CzU{\G}$ denoted by $\Sant^{\mathrm{u}},\Rant^{\mathrm{u}},\tau^{\mathrm{u}}={(\tau_{t}^{\mathrm{u}})}_{t\in\R}$,
respectively. Again, $\Sant^{\mathrm{u}}=\Rant^{\mathrm{u}}\circ\tau_{-i/2}^{\mathrm{u}}$,
we have $\Lambda\circ \Rant^{\mathrm{u}}=\Rant\circ\Lambda$ and $\Lambda\circ\tau_{t}^{\mathrm{u}}=\tau_{t}\circ\Lambda$ for all $t\in\R$,
and the fundamental property of the antipode has a universal version:
\begin{equation}
(\i\tensor\om)(\Ww)\in D(\Sant^{\mathrm{u}})\text{ and }\Sant^{\mathrm{u}}\left((\i\tensor\om)(\Ww)\right)=(\i\tensor\om)(\Ww^{*})\qquad(\forall_{\om\in\Lone{\hat{\G}}}).\label{eq:S_univ_prop}
\end{equation}

The co-multiplications $\Delta$, $\Delta|_{\Cz{\G}}$ and $\Delta_{\mathrm{u}}$
induce on $\Lone{\G}$, $\Cz{\G}^{*}$ and $\CzU{\G}^{*}$, respectively,
\emph{convolution} products, for instance by $\om_{1}\conv\om_{2}:=(\om_{1}\tensor\om_{2})\circ\Delta$
($\om_{1},\om_{2}\in\Lone{\G}$), turning these spaces into completely
contractive Banach algebras. The co-unit $\epsilon$ is the unit of $\CzU{\G}^{*}$. 
The restriction map and the map of composing
with the reducing morphism allow us to embed $\Lone{\G}\hookrightarrow\Cz{\G}^{*}\hookrightarrow\CzU{\G}^{*}$
as Banach algebras. What is more, each `small' set is an ideal in
every `larger' one (see for example \citep{Hu_Neufang_Ruan__cb_mult_LCQG} and references therein). 
\begin{comment}
Let us give the details of why $\Lone{\G},\Cz{\G}^{*}$
are left ideals in $\CzU{\G}^{*}$: if $\mu\in\CzU{\G}^{*}$ and $\om\in\Cz{\G}^{*}$,
then by \prettyref{eq:Delta_u_Ww}, $\mu\conv\om=(\mu\tensor(\om\circ\Lambda))\circ\Delta_{\mathrm{u}}$,
so
\[
(\mu\conv\om)(x)=((\mu\tensor\om)\circ(\i\tensor\Lambda)\circ\Delta_{\mathrm{u}})(x)=(\mu\tensor\om)\left(\Ww^{*}(\one\tensor\Lambda(x))\Ww\right)\qquad(\forall_{x\in\CzU{\G}}).
\]
Consequently, $\mu\conv\om\in\CzU{\G}^{*}$ can be identified with
$\rho\in\Cz{\G}^{*}$ given by $\rho(y):=(\mu\tensor\om)\left(\Ww^{*}(\one\tensor y)\Ww\right)$,
$y\in\Cz{\G}$. Now suppose that $\om\in\Lone{\G}$. Assuming as we
may that $\mu$ is a state, write $(\H,\pi,\xi)$ for the GNS construction
of $\mu$, and (temporarily) let $\Ww_{\mu}:=(\pi\tensor\i)(\Ww)\in B(\H\tensor\Ltwo{\G})$.
Then $\rho(y)=(\om_{\xi}\tensor\om)\left(\Ww_{\mu}^{*}(\one\tensor y)\Ww_{\mu}\right)$
for all $y\in\Cz{\G}$. It is now clear that $\rho$ admits a normal
extension to $\Linfty{\G}$.
\end{comment}

For every Banach algebra $\mathsf{A}$, its dual $\mathsf{A}^{*}$
becomes an $\mathsf{A}$-bimodule as customary, namely, for $\theta\in\mathsf{A}^{*}$
and $a\in\mathsf{A}$, one defines $\theta a,a\theta\in\mathsf{A}^{*}$
by $(\theta a)(b):=\theta(ab)$ and $(a\theta)(b):=\theta(ba)$, $b\in\mathsf{A}$.
When specialising to one of the Banach algebras $\Lone{\G},\Cz{\G}^{*},\CzU{\G}^{*}$,
we denote these bimodule actions with a `$\cdot$' rather than just by
juxtaposition. For instance, by the definition of the convolution,
the actions of $\mu\in\CzU{\G}^{*}$ on $x\in\CzU{\G}\hookrightarrow\CzU{\G}^{**}$
are 
\[
\mu\cdot x=(\i\tensor\mu)(\Delta_{\mathrm{u}}(x)),x\cdot\mu=(\mu\tensor\i)(\Delta_{\mathrm{u}}(x))\in\M{\CzU{\G}}.
\]
In fact, we have $\CzU{\G}^{*} \cdot \CzU{\G} \subseteq \CzU{\G}$ and $\Lone{\G} \cdot \CzU{\G}$ is dense in $\CzU{\G}$; the same holds for the other module map and for $\Cz{\G}$ in place of $\CzU{\G}$. 
These assertions follow from the quantum cancellation rules, e.g.~$\clinspan{\Delta_{\mathrm{u}}(\CzU{\G})(\one \tensor \CzU{\G})} = \CzU{\G} \tensormin \CzU{\G}$.

\begin{example}
Commutative locally compact quantum groups $\G$, namely the ones whose $\Linfty{\G}$ algebra
is commutative, are in $1-1$ correspondence with locally compact
groups $G$. We have $\Linfty{\G}=\Linfty G$, $\CzU{\G}=\Cz{\G}=\Cz G$,
$(\Delta(f))(t,s)=f(ts)$ for $f\in\Linfty G$ and $t,s\in G$ (after
identifying $\Linfty G\tensorn\Linfty G$ with $\Linfty{G\times G}$)
and $\Ltwo{\G}=\Ltwo G$. The left and right Haar weights $\varphi,\psi$
are given by integration against the left and right Haar measures of $G$,
respectively. As $\sigma^{\varphi}$ is trivial, it follows immediately
from \prettyref{cor:i_p_core} \prettyref{enu:i_p_core__2} that
\[
\left[0,h_{\varphi}^{1/2}\right]_{\Ltwo G}=\mathfrak{i}_{\varphi}^{(2)}\left(\left[0,\one\right]_{\mathcal{M}^{(2,\varphi)}}\right)=\left\{ f\in\Ltwo G:0\le f\le1\text{ a.e.}\right\} .
\]
\end{example}

\begin{example}
Co-commutative locally compact quantum groups $\G$, namely the ones whose $\Lone{\G}$ algebra
is commutative, are in $1-1$ correspondence with duals $\hat{G}$
of locally compact groups $G$. We have $\Linfty{\hat{G}}=\VN G$,
the left von Neumann algebra of $G$ generated by the left translation
operators $\left\{ \l_{g}:g\in G\right\} $ on $\Ltwo G=\Ltwo{\hat{G}}$,
$\Cz{\hat{G}}$ is the reduced group C$^{*}$-algebra $\CStarR G$
of $G$, $\CzU{\hat{G}}$ is the full group C$^{*}$-algebra $\CStarF G$
of $G$, and $\Delta(\l_{g})=\l_{g}\tensor\l_{g}$ for all $g\in G$.
Denote the modular function of $G$ by $\delta$, and for $f:G\to\C$
define $f^{\#},f^{*}:G\to\C$ by $f^{\#}(s):=\delta(s)^{-1}\overline{f(s^{-1})}$
and $f^{*}(s):=\delta(s)^{-1/2}\overline{f(s^{-1})}$, $s\in G$.
Recall that the convolution of $f\in\Ltwo G$ and $g\in\Cc G$ is
given by $(f*g)(s):=\int_{G}f(t)g(t^{-1}s)\d t$ ($s\in G$). The
left and right Haar weights of $\hat{G}$ are both equal to the Plancherel
weight of $G$ \citep[Section VII.3]{Takesaki__book_vol_2}. It is
the n.s.f.~weight associated to the Tomita algebra $\Cc G\subseteq\Ltwo G$
with convolution as product and involution being $\Cc G\ni f\mapsto f^{\#}$.
The associated modular operator $\Delta$ is the operator of multiplication
by $\delta$ with maximal domain in $\Ltwo G$, and the modular conjugation $J$
is $\Ltwo G\ni f\mapsto f^{*}$. The left-bounded elements are all
$f\in\Ltwo G$ such that the map given by $\Cc G\ni g\mapsto f*g$
extends to a bounded operator $\pi_{l}(f)$ on $\Ltwo G$. With respect
to the Plancherel weight, from \prettyref{cor:i_p_core} \prettyref{enu:i_p_core__2}
we deduce that 
\[
\begin{split}\mathfrak{i}^{(2)}\left(\left[0,\one\right]_{\mathcal{M}^{(2)}}\right) & =\overline{\nabla^{1/4}\gnsmap\left(\left[0,\one\right]_{\mathcal{M}}\right)}^{\Ltwo G}\\
 & =\overline{\nabla^{1/4}\gnsmap\left(\left\{ \pi_{l}(f)^{*}\pi_{l}(f):f\in\Ltwo G\text{ is left bounded and }\left\Vert \pi_{l}(f)\right\Vert \le1\right\} \right)}^{\Ltwo G}.
\end{split}
\]
Let $f\in\Ltwo G$ be left bounded with $\left\Vert \pi_{l}(f)\right\Vert \le1$.
We can use \citep[Theorem VI.1.26]{Takesaki__book_vol_2} to find
a sequence $\left(f_{n}\right)_{n=1}^{\infty}$ in $\Cc G$ that converges
to $f$ in $\Ltwo G$ and such that $\left\Vert \pi_{l}(f_{n})\right\Vert \le1$
for all $n\in\N$ and $\pi_{l}(f_{n})\xrightarrow[n\to\infty]{}\pi_{l}(f)$
in the strong operator topology. As in the proof of \prettyref{lem:i_2_gns} \prettyref{enu:i_2_gns__2},
we then have $\nabla^{1/4}\gnsmap(\pi_{l}(f_{n})^{*}\pi_{l}(f_{n}))\xrightarrow[n\to\infty]{}\nabla^{1/4}\gnsmap(\pi_{l}(f)^{*}\pi_{l}(f))$
weakly. By convexity, 
\[
\mathfrak{i}^{(2)}\left(\left[0,\one\right]_{\mathcal{M}^{(2)}}\right)=\overline{\co \left\{ \delta^{1/4}(f^{\#}*f):f\in\Cc G,\left\Vert \pi_{l}(f)\right\Vert \le1\right\} }^{\Ltwo G}.
\]
\end{example}

We return to the general case. For $\mu\in\Cz{\hat{\G}}^{*}$, the
next lemma is just \citep[Lemma 3.2 and Remark 2]{Ng_Viselter__amenability_LCQGs_coreps}.
Let us give full details (using a different argument) for completeness.
\begin{lem}
\label{lem:antipode_W_extended}For every $\mu\in\CzU{\hat{\G}}^{*}$
we have 
\[
(\i\tensor\mu)(\wW)\in D(\Sant)\text{ and }\Sant((\i\tensor\mu)(\wW))=(\i\tensor\mu)(\wW^{*}).
\]
\end{lem}

\begin{proof}
We may suppose that $\mu$ is a state. Writing $(\H,\pi,\xi)$ for
the GNS construction of $\mu$, the operator $\wW_{\mu}:=(\i\tensor\pi)(\wW)\in\M{\Cz{\G}\tensormin\K(\H)}$
is a unitary representation of $\G$ on $\H$ and $(\i\tensor\mu)(\wW)=(\i\tensor\om_{\xi})(\wW_{\mu})$.
Thus this operator belongs to $D(\Sant)$ and $\Sant((\i\tensor\om_{\xi})(\wW_{\mu}))=(\i\tensor\om_{\xi})(\wW_{\mu}^{*})=(\i\tensor\mu)(\wW^{*})$.
\end{proof}
The following construction is useful. For $\mu\in\CzU{\G}^{*}$, let
\begin{equation}
\mu_{n}:=\frac{\sqrt{n}}{\sqrt{\pi}}\int_{\R}e^{-nt^{2}}\mu\circ\tau_{t}^{\mathrm{u}}\d t\in\CzU{\G}^{*}\qquad(n\in\N),\label{eq:CzU_star_tau_smear}
\end{equation}
with the integrals converging in the $w^{*}$-topology. Then for all $n\in\N$,
$\mu_{n}$ is entire analytic with respect to (the adjoint of) $\tau^{\mathrm{u}}$
and $\left\Vert \mu_{n}\right\Vert \le\left\Vert \mu\right\Vert $.
In addition, $\mu_{n}\xrightarrow[n\to\infty]{}\mu$ in the $w^{*}$-topology,
thus $\left\Vert \mu_{n}\right\Vert \xrightarrow[n\to\infty]{}\left\Vert \mu\right\Vert $.

The next definition and the proposition which follows it belong to the quantum group folklore.

\begin{defn}
A locally compact quantum group $\G$ is \emph{second countable} if $\Cz{\G}$ is separable.
\end{defn}

\begin{prop}
\label{prop:second_countable_LCQG}For a locally compact quantum group $\G$ the following
conditions are equivalent:
\begin{enumerate}
\item \label{enu:second_countable_LCQG__1}$\G$ is second countable;
\item \label{enu:second_countable_LCQG__2}$\hat{\G}$ is second countable;
\item \label{enu:second_countable_LCQG__3}$\CzU{\G}$ is separable;
\item \label{enu:second_countable_LCQG__4}$\CzU{\hat{\G}}$ is separable;
\item \label{enu:second_countable_LCQG__5}$\Ltwo{\G}$ is separable.
\end{enumerate}
\end{prop}

\begin{proof}
It is obviously sufficient to prove that \prettyref{enu:second_countable_LCQG__1},
\prettyref{enu:second_countable_LCQG__3} and \prettyref{enu:second_countable_LCQG__5}
are equivalent.

\prettyref{enu:second_countable_LCQG__1}$\implies$\prettyref{enu:second_countable_LCQG__5}
Since $\varphi$ is the $W^{*}$-lift of its restriction to $\Cz{\G}$,
which is a KMS weight (in the sense of \citep{Kustermans_Vaes__LCQG_C_star}), it follows from modular theory that if $A$
is a countable dense subset of $\Cz{\G}\cap\mathcal{N}_{\varphi}\cap\mathcal{N}_{\varphi}^{*}$,
then $J_{\varphi}AJ_{\varphi}\gnsmap_{\varphi}(A)$ is dense in $\Ltwo{\G}=\overline{\gnsmap_{\varphi}(\Cz{\G}\cap\mathcal{N}_{\varphi}\cap\mathcal{N}_{\varphi}^{*})}$ and the latter is also equal to $\overline{J_\varphi(\Cz{\G}\cap\mathcal{N}_{\varphi}\cap\mathcal{N}_{\varphi}^{*})J_\varphi \gnsmap_{\varphi}(\Cz{\G}\cap\mathcal{N}_{\varphi}\cap\mathcal{N}_{\varphi}^{*})}$.

\prettyref{enu:second_countable_LCQG__5}$\implies$\prettyref{enu:second_countable_LCQG__3}
This follows from the density of $\left\{ (\i\tensor\om_{\z,\eta})(\Ww):\z,\eta\in\Ltwo{\G}\right\} $
in $\CzU{\G}$.

\prettyref{enu:second_countable_LCQG__3}$\implies$\prettyref{enu:second_countable_LCQG__1}
Simply use the canonical reducing surjection $\CzU{\G}\to\Cz{\G}$.
\end{proof}
We will refer several times to the following particular classes of
locally compact quantum groups: compact (or dually, discrete), amenable, and co-amenable. Since
their definitions are not necessary for understanding the chief part
of this paper, we will not give them here, but rather refer the reader
to \citep{Maes_van_Daele__notes_CQGs,Woronowicz__CMP,Woronowicz__symetries_quantiques}
and to \citep{Bedos_Tuset_2003}.

\section{The left and right convolution operators associated with locally compact quantum groups}\label{sec:L_R_oper}

In this section we discuss various incarnations of convolution operators associated to quantum measures on locally compact quantum groups (i.e.\ continuous functionals on $\CzU{\G}$). After recalling and sometimes rephrasing known results in Subsection \ref{subsect:Basics}, essentially following \citep{Junge_Neufang_Ruan__rep_thm_LCQG, Daws__CPM_LCQGs_2012, Daws_Fima_Skalski_White_Haagerup_LCQG, Salmi_Skalski__idemp_states_LCQG_II}, in Subsection \ref{subsect:Symmetry} we introduce the conditions on the functional in question guaranteeing that the corresponding convolution operator possesses natural symmetry properties. Finally in Subsection \ref{subsect:Hilbert} we analyse the GNS- and KMS-implementations of  convolution operators on $L^2(\G)$. 

\subsection{Basic facts} \label{subsect:Basics}

Let $\G$ be a locally compact quantum group and $\mu\in\CzU{\G}^{*}$. Since $\Lone{\G}$,
when viewed canonically as a subspace of the completely contractive
Banach algebra $\CzU{\G}^{*}$, is an ideal, the operators $\Lone{\G}\to\Lone{\G}$
given by $\Lone{\G}\ni\om\mapsto\om\conv\mu$ and $\Lone{\G}\ni\om\mapsto\mu\conv\om$
have cb-norms dominated by $\left\Vert \mu\right\Vert $. Thus their
adjoints, denoted by $L_{\mu},R_{\mu}:\Linfty{\G}\to\Linfty{\G}$,
respectively, are normal and satisfy $\left\Vert L_{\mu}\right\Vert _{cb},\left\Vert R_{\mu}\right\Vert _{cb}\le\left\Vert \mu\right\Vert $.
We also write $\mu\cdot x:=L_{\mu}(x)$ and $x\cdot\mu:=R_{\mu}(x)$.
These make $\Linfty{\G}$ into a $\CzU{\G}^{*}$-bimodule, extending
its canonical structure as an $\Lone{\G}$-bimodule. We also have
the operators $R_{\mu}^{\mathrm{u}},L_{\mu}^{\mathrm{u}}:\CzU{\G}\to\CzU{\G}$
given by the $\CzU{\G}^{*}$-bimodule structure of $\CzU{\G}$, admitting also a simple direct description: for example $L_\mu^{\mathrm{u}} = (\textup{id} \otimes \mu)\circ \Delta_\mathrm{u}$. In fact in the literature one can find varying left/right conventions here (so that for example \citep{Salmi_Skalski__idemp_states_LCQG_II} calls the latter map a right convolution operator). Finally note
that we can recover $\mu$ by $\epsilon\circ R_{\mu}^{\mathrm{u}}=\mu=\epsilon\circ L_{\mu}^{\mathrm{u}}$.

We gather some information from the literature about these operators
in the subsequent results. Notice that a completely bounded map $T$ on
$\Linfty{\G}$ satisfies $\Delta\circ T=(T\tensor\i)\circ\Delta$
(resp.~$\Delta\circ T=(\i\tensor T)\circ\Delta$) if and only if
it commutes with the operators $L_{\om}$ (resp.~$R_{\om}$), $\om\in\Lone{\G}$.
We will see below that a similar, and even stronger, statement holds
for the maps $R_{\mu}^{\mathrm{u}},L_{\mu}^{\mathrm{u}}$. A bounded
linear map between two C$^{*}$-algebras is called \emph{strict} if
it is strictly continuous on bounded sets.
\begin{thm}
\label{thm:L_R}Let $\G$ be a locally compact quantum group.
\begin{enumerate}
\item \label{enu:L_R_maps_inj} The maps $\CzU{\G}^{*}\ni\mu\mapsto R_{\mu}$
and $\CzU{\G}^{*}\ni\mu\mapsto L_{\mu}$ are injective.
\item \label{enu:L_R_Haar_weights} For every state $\mu\in\CzU{\G}^{*}$
the left, respectively right, Haar weight is invariant under $R_{\mu}$,
respectively $L_{\mu}$.
\item \label{enu:L_R_universal}We have 
\[
\begin{split}\left\{ R_{\mu}^{\mathrm{u}}:\mu\in\CzU{\G}^{*}\right\}  & =\left\{ T^{\mathrm{u}}\in CB(\CzU{\G}):T^{\mathrm{u}}\tensor\i\text{ is strict and }\Delta_{\mathrm{u}}\circ T^{\mathrm{u}}=(T^{\mathrm{u}}\tensor\i)\circ\Delta_{\mathrm{u}}\right\} \\
 & =\left\{ T^{\mathrm{u}}\in B(\CzU{\G}):T^{\mathrm{u}}\text{ commutes with all maps }L_{\nu}^{\mathrm{u}},\nu\in\CzU{\G}^{*}\right\} .
\end{split}
\]
Furthermore, $\left\Vert R_{\mu}^{\mathrm{u}}\right\Vert _{cb}=\left\Vert R_{\mu}^{\mathrm{u}}\right\Vert =\left\Vert \mu\right\Vert $
for all $\mu\in\CzU{\G}^{*}$. 
\item \label{enu:L_R_CB_mult}There exist $1-1$ completely isometric correspondences
between:
\begin{itemize}
\item completely bounded, normal maps $T$ on $\Linfty{\G}$ that satisfy
$\Delta\circ T=(T\tensor\i)\circ\Delta$;
\item completely bounded maps $T'$ on $\Cz{\G}$ that commute with the
operators $L_{\om}$, $\om\in\Lone{\G}$;
\item completely bounded right module maps $M$ on $\Lone{\G}$ (that is, $M(\om_{1}\conv\om_{2})=M(\om_{1}) \conv \om_{2}$
for all $\om_{1},\om_{2}\in \Lone{\G}$).
\end{itemize}
They are given by $T'=T|_{\Cz{\G}}$ and $M=T_{*}$.
\item \label{enu:L_R_CP_mult}There exist $1-1$ order-preserving correspondences
between:
\begin{itemize}
\item positive functionals $\mu\in\CzU{\G}^{*}$;
\item completely positive, normal maps $T$ on $\Linfty{\G}$ that satisfy
$\Delta\circ T=(T\tensor\i)\circ\Delta$;
\item completely positive maps $T'$ on $\Cz{\G}$ that commute with the
operators $L_{\om}$, $\om\in\Lone{\G}$;
\item completely positive maps $T^{\mathrm{u}}$ on $\CzU{\G}$ that commute
with the operators $L_{\nu}^{\mathrm{u}}$, $\nu\in\CzU{\G}^{*}$;
\item right module maps $M$ on $\Lone{\G}$ with completely positive adjoints.
\end{itemize}
They are given by $T=R_{\mu}$, $T'=T|_{\Cz{\G}}$, $T^{\mathrm{u}} = R_\mu^\mathrm{u}$ and $M=T_{*}$,
and satisfy $\left\Vert \mu\right\Vert =\left\Vert T\right\Vert =\left\Vert T\right\Vert _{cb}=\left\Vert T'\right\Vert =\left\Vert T'\right\Vert _{cb}=\left\Vert T^{\mathrm{u}}\right\Vert =\left\Vert T^{\mathrm{u}}\right\Vert _{cb}=\left\Vert M\right\Vert =\left\Vert M\right\Vert _{cb}$. 
\item \label{enu:L_R_CB_mult_amen}Assume that $\hat{\G}$ is amenable.
There exist $1-1$ correspondences between: 
\begin{itemize}
\item functionals $\mu\in\CzU{\G}^{*}$;
\item completely bounded, normal maps $T$ on $\Linfty{\G}$ that satisfy
$\Delta\circ T=(T\tensor\i)\circ\Delta$;
\item completely bounded maps $T'$ on $\Cz{\G}$ that commute with the
operators $L_{\om}$, $\om\in\Lone{\G}$;
\item completely bounded right module maps $M$ on $\Lone{\G}$.
\end{itemize}
They are given by $T=R_{\mu}$ , $T'=T|_{\Cz{\G}}$ and $M=T_{*}$.

\end{enumerate}
In each of \prettyref{enu:L_R_universal}\textendash \prettyref{enu:L_R_CB_mult_amen}
a similar statement holds for the left actions of $\CzU{\G}^{*}$.
\end{thm}

\begin{proof}
\prettyref{enu:L_R_maps_inj} See \citep[Proposition 8.3 and its proof]{Daws__mult_self_ind_dual_B_alg}
(or \prettyref{lem:C_z_U_star_conv} below).

\prettyref{enu:L_R_Haar_weights} This is \citep[Lemma 3.4]{Kalantar_Neufang_Ruan__Poisson_bdry_LCQG}. 

\prettyref{enu:L_R_universal} The assertion is proved precisely as
\citep[Proposition 3.2]{Lindsay_Skalski__conv_semigrp_states}.

\prettyref{enu:L_R_CB_mult} Follows from \citep[Propositions 4.1 and 4.2]{Junge_Neufang_Ruan__rep_thm_LCQG},
see also \citep[Lemma 12]{Salmi_Skalski__idemp_states_LCQG_II} (and the preceding remark, replacing complete positivity by complete boundedness and ignoring the non-degeneracy condition).

\prettyref{enu:L_R_CP_mult} Follows from \prettyref{enu:L_R_universal},
\prettyref{enu:L_R_CB_mult} and \citep[Theorems 5.1 and 5.2]{Daws__CPM_LCQGs_2012}. 

\prettyref{enu:L_R_CB_mult_amen} Follows from \prettyref{enu:L_R_CB_mult}
and the left analogue of \citep[Proposition 5.10 and the discussion succeeding the definition of $\widetilde{\rho}$ in Section 3]{Crann__amn_cov_inj_2}.
\end{proof}
\begin{rem}
\mbox{}
\begin{enumerate}
\item By \citep[Theorem 7.2]{Crann__inner_amen_rel_hom}, the correspondence
in \prettyref{enu:L_R_CB_mult_amen} is isometric: $\left\Vert \mu\right\Vert =\left\Vert T\right\Vert _{cb}=\left\Vert M\right\Vert _{cb}$.
\item The result of \prettyref{enu:L_R_CB_mult_amen} was known before \citep{Crann__amn_cov_inj_2}
for co-amenable $\G$ \citep[Proposition 3.1 or Theorem 4.2]{Hu_Neufang_Ruan__cb_mult_LCQG}.
\end{enumerate}
\end{rem}

We now start discussing connections between various modes of convergence of different avatars of convolution operators (see also  \citep[Theorem 4.6]{Runde_Viselter_LCQGs_PosDef}).

\begin{lem}
\label{lem:C_z_U_star_conv}Let $\left(\mu_{i}\right)_{i\in\mathcal{I}}$
be a bounded net in $\CzU{\G}^{*}$, and let $\mu\in\CzU{\G}^{*}$.
Consider the following conditions:
\begin{enumerate}
\item \label{enu:C_z_U_star_conv__1}$\mu_{i}\xrightarrow[i\in\mathcal{I}]{}\mu$
in the $w^{*}$-topology;
\item \label{enu:C_z_U_star_conv__2}$R_{\mu_{i}}^{\mathrm{u}}(x)\xrightarrow[i\in\mathcal{I}]{}R_{\mu}^{\mathrm{u}}(x)$
weakly for every $x\in\CzU{\G}$;
\item \label{enu:C_z_U_star_conv__3}$R_{\mu_{i}}(x)\xrightarrow[i\in\mathcal{I}]{}R_{\mu}(x)$
ultraweakly for every $x\in\Cz{\G}$;
\item \label{enu:C_z_U_star_conv__4}$R_{\mu_{i}}(x)\xrightarrow[i\in\mathcal{I}]{}R_{\mu}(x)$
ultraweakly for every $x\in\Linfty{\G}$;
\item \label{enu:C_z_U_star_conv__5}$(R_{\mu_{i}})_{*}\xrightarrow[i\in\mathcal{I}]{}(R_{\mu})_{*}$
in the point\textendash norm topology.
\end{enumerate}
Then \prettyref{enu:C_z_U_star_conv__5}$\implies$\prettyref{enu:C_z_U_star_conv__4}$\implies$\prettyref{enu:C_z_U_star_conv__3}$\iff$\prettyref{enu:C_z_U_star_conv__2}$\iff$\prettyref{enu:C_z_U_star_conv__1},
and if $\left\Vert \mu_{i}\right\Vert =\left\Vert \mu\right\Vert $
for all $i\in\mathcal{I}$, then all these conditions are equivalent.
Similar statements hold for operators of the form $L_{\mu}$, $\mu\in\CzU{\G}^{*}$.
\end{lem}

\begin{proof}
The implications \prettyref{enu:C_z_U_star_conv__5}$\implies$\prettyref{enu:C_z_U_star_conv__4}$\implies$\prettyref{enu:C_z_U_star_conv__3}
are clear, and \prettyref{enu:C_z_U_star_conv__1} implies \prettyref{enu:C_z_U_star_conv__5}
under the additional norm assumption by \citep[Theorem 4.6, (i)$\implies$(vii)]{Runde_Viselter_LCQGs_PosDef}. 

Proving that \prettyref{enu:C_z_U_star_conv__3}$\iff$\prettyref{enu:C_z_U_star_conv__1}
can be done similarly to the proof of \citep[Corollary 4.5]{Runde_Viselter_LCQGs_PosDef}
as follows.
Let $\Lambda:\CzU{\G}\to\Cz{\G}$
be the reducing morphism. Then for every $a\in\CzU{\G}$, $\om\in\Lone{\G}$
and $\nu\in\CzU{\G}^{*}$ we have $\nu(\om\cdot a)=\om(R_{\nu}(\Lambda(a)))$,
where we used the action of $\Lone{\G}$ on $\CzU{\G}$. This implies
that \prettyref{enu:C_z_U_star_conv__1}$\implies$\prettyref{enu:C_z_U_star_conv__3}.
Conversely, assuming \prettyref{enu:C_z_U_star_conv__3}, we see
that $\mu_{i}(b)\xrightarrow[i\in\mathcal{I}]{}\mu(b)$ for all $b\in\Lone{\G}\cdot\CzU{\G}$.
Since $\Lone{\G}\cdot\CzU{\G}$ is dense in $\CzU{\G}$, the boundedness
of $\left(\mu_{i}\right)_{i\in\mathcal{I}}$ allows us to conclude
that $\mu_{i}\xrightarrow[i\in\mathcal{I}]{}\mu$ in the $w^{*}$-topology.

The implication \prettyref{enu:C_z_U_star_conv__1}$\implies$\prettyref{enu:C_z_U_star_conv__2}
is immediate, and the converse holds by the boundedness of $\left(\mu_{i}\right)_{i\in\mathcal{I}}$
because $\CzU{\G}^{*}\cdot\CzU{\G}\subseteq\CzU{\G}$ is dense in
$\CzU{\G}$.
\end{proof}

The next result follows from, and is in fact equivalent to, \citep[Theorem 2.6]{Daws_Fima_Skalski_White_Haagerup_LCQG}.
For clarity, we give its proof, whose second part is different from
that of \citep{Daws_Fima_Skalski_White_Haagerup_LCQG}. In the language of \cite{Caspers_Skalski__Haagerup_AP_Dirichlet_forms} we would say that we describe here the GNS-implementations of convolution operators.
\begin{prop}
\label{prop:R_L_phi_psi}For every $\mu\in\CzU{\G}^{*}$, $\mathcal{N}_{\varphi},\mathcal{M}_{\varphi}$
are invariant under $R_{\mu}$ and $\mathcal{N}_{\psi},\mathcal{M}_{\psi}$
are invariant under $L_{\mu}$. The maps $\gnsmap_{\varphi}(x)\mapsto\gnsmap_{\varphi}(R_{\mu}(x))$,
$x\in\mathcal{N}_{\varphi}$, and $\gnsmap_{\psi}(x)\mapsto\gnsmap_{\psi}(L_{\mu}(x))$,
$x\in\mathcal{N}_{\psi}$, are well defined and extend to bounded
operators on $\Ltwo{\G}$ denoted by $\widetilde{R}_{\mu}^{\varphi}$
and $\widetilde{L}_{\mu}^{\psi}$, respectively. In fact, we have
$\widetilde{R}_{\mu}^{\varphi}=(\mu\tensor\i)(\Ww^{*})\in\M{\Cz{\hat{\G}}}$
and $\widetilde{L}_{\mu}^{\psi}=(\i\tensor\mu)(\vV)\in\M{\Cz{\hat{\G}'}}$. Furthermore if $\mu \in L^1(\G)$, then $\widetilde{R}_{\mu}^{\varphi}\in \Cz{\hat{\G}}$,
and $\widetilde{L}_{\mu}^{\psi}\in\Cz{\hat{\G}'}$.
\end{prop}

\begin{proof}
Suppose that $\mu$ is a state. That $R_{\mu}(\mathcal{M}_{\varphi}^{+})\subseteq\mathcal{M}_{\varphi}^{+}$
follows readily from \prettyref{thm:L_R} \prettyref{enu:L_R_Haar_weights}. For every
$x\in\mathcal{N}_{\varphi}$ we have $R_{\mu}(x)^{*}R_{\mu}(x)\leq R_{\mu}(x^{*}x)$
by Kadison's Cauchy\textendash Schwarz inequality, which applies because
$R_{\mu}$ is a completely positive contraction. From \prettyref{thm:L_R} \prettyref{enu:L_R_Haar_weights}
we hence get $\varphi(R_{\mu}(x)^{*}R_{\mu}(x))\leq\varphi(R_{\mu}(x^{*}x))=\varphi(x^{*}x)$.
For general $\mu\in\CzU{\G}^{*}$, decompose $\mu$ as a linear combination
of states. 

The equality $\widetilde{R}_{\mu}^{\varphi}=(\mu\tensor\i)(\Ww^{*})$
can be easily deduced from \citep[Corollary 8.2, (1)]{Kustermans__LCQG_universal}.
An alternative proof is the following. For every $\om\in\Lone{\G}$
we have $\widetilde{R}_{\om}^{\varphi}=(\om\tensor\i)(W^{*})$ by
the very definition of the operator $W$. Therefore, under the embedding
$\Lone{\G}\hookrightarrow\CzU{\G}^{*}$, we have $\widetilde{R}_{\om}^{\varphi}=(\om\tensor\i)(\Ww^{*})$
for all $\om\in\Lone{\G}$. Consider now the anti-homomorphisms $\CzU{\G}^{*}\to B(\Ltwo{\G})$
given by $\mu\mapsto\widetilde{R}_{\mu}^{\varphi}$ and $\mu\mapsto(\mu\tensor\i)(\Ww^{*})$.
Using the facts that they coincide on the ideal $\Lone{\G}$ of $\CzU{\G}^{*}$
and that the algebra $\{\widetilde{R}_{\om}^{\varphi}:\om\in\Lone{\G}\}$
acts non-degenerately on $\Ltwo{\G}$, we infer that they coincide on
$\CzU{\G}^{*}$.

Since $\Ww\in\M{\CzU{\G}\tensor\Cz{\hat{\G}}}$, we have $\widetilde{R}_{\mu}^{\varphi}\in\M{\Cz{\hat{\G}}}$. Then the fact that  for $\mu \in \Lone{\QG}$ we have $\widetilde{R}_{\mu}^{\varphi}\in\Cz{\hat{\G}}$ is clear.

The proof for $L_{\mu}$ is similar.
\end{proof}

We are now ready to start discussing GNS-implementations of convolution operators with respect to the `wrong-sided' weights (i.e.\ for example $\psi$ for the right convolution operator).

\begin{prop}
\label{prop:R_L_phi_psi_inverted}\mbox{}
\begin{enumerate}
\item \label{enu:R_L_phi_psi_inverted__1}Let $\z,\eta$ be in $D(\delta^{-\frac{1}{2}})$
(resp.~$D(\delta^{\frac{1}{2}})$) and $\om:=\om_{\z,\eta}$. Then
$\mathcal{M}_{\psi}$ (resp.~$\mathcal{M}_{\varphi}$) is invariant
under $R_{\om}$ (resp.~$L_{\om}$). 
\item \label{enu:R_L_phi_psi_inverted__2}Let $\z$ be in $D(\delta^{-\frac{1}{2}})$
(resp.~$D(\delta^{\frac{1}{2}})$), $\eta\in\Ltwo{\G}$ and $\om:=\om_{\z,\eta}$.
Then $\mathcal{N}_{\psi}$ (resp.~$\mathcal{N}_{\varphi}$) is invariant
under $R_{\om}$ (resp.~$L_{\om}$). The map $\gnsmap_{\psi}(x)\mapsto\gnsmap_{\psi}(R_{\om}(x))$,
$x\in\mathcal{N}_{\psi}$, (resp.~$\gnsmap_{\varphi}(x)\mapsto\gnsmap_{\varphi}(L_{\om}(x))$,
$x\in\mathcal{N}_{\varphi}$) is well defined and extends to a bounded
operator on $\Ltwo{\G}$ denoted by $\widetilde{R}_{\om}^{\psi}$
(resp.~$\widetilde{L}_{\om}^{\varphi}$), which is equal to $\widetilde{R}_{\om_{\delta^{-\frac{1}{2}}\z,\eta}}^{\varphi}$
(resp.~$\widetilde{L}_{\om_{\delta^{\frac{1}{2}}\z,\eta}}^{\psi}$).
\end{enumerate}
\end{prop}

\begin{proof}
We prove only the first assertions.

\prettyref{enu:R_L_phi_psi_inverted__1} It is enough to assume $\z=\eta$.
By \citep[Result 7.6]{Kustermans_Vaes__LCQG_C_star}, in that case $\psi\left(R_{\om}(x)\right)=\psi(x)\om_{\delta^{-\frac{1}{2}}\z}(\one)$
for each $x\in\mathcal{M}_{\psi}^{+}$.

\prettyref{enu:R_L_phi_psi_inverted__2} Recall that $\psi=\varphi_{\delta}$.
Write $\mathcal{N}_{\psi,0}:=\{x\in\Linfty{\G}:x\delta^{\frac{1}{2}}\text{ is bounded and }\overline{x\delta^{\frac{1}{2}}}\in\mathcal{N}_{\varphi}\}$.
Then $\mathcal{N}_{\psi,0}\subseteq\mathcal{N}_{\psi}$, and for every
$x\in\mathcal{N}_{\psi,0}$ we have $\gnsmap_{\psi}(x)=\gnsmap_{\varphi}(\overline{x\delta^{\frac{1}{2}}})$.
Fix $\z\in D(\delta^{-\frac{1}{2}})$ and $\eta\in\Ltwo{\G}$. Let
$x\in\mathcal{N}_{\psi,0}$. The identity $\Delta(\delta)=\delta\tensor\delta$, which can be understood for example in the language of functional calculus for operators affiliated with a C$^*$-algebra, or directly as
 $W^{*}(\one\tensor\delta)W=\delta\tensor\delta$, yields $(\om_{\z,\eta}\tensor\i)(\Delta(x))\delta^{\frac{1}{2}}\subseteq(\om_{\delta^{-\frac{1}{2}}\z,\eta}\tensor\i)(\Delta(\overline{x\delta^{\frac{1}{2}}}))$.
Indeed, we have $\Delta(\overline{x\delta^{\frac{1}{2}}})=W^{*}(\one\tensor\overline{x\delta^{\frac{1}{2}}})W\supseteq W^{*}(\one\tensor x)WW^{*}(\one\tensor\delta^{\frac{1}{2}})W=\Delta(x)(\delta^{\frac{1}{2}}\tensor\delta^{\frac{1}{2}})$,
and for every $\a\in D(\delta^{\frac{1}{2}})$ and $\be\in\Ltwo{\G}$,
\[
\begin{split}\left\langle (\om_{\z,\eta}\tensor\i)(\Delta(x))\delta^{\frac{1}{2}}\a,\be\right\rangle  & =\left\langle \Delta(x)(\z\tensor\delta^{\frac{1}{2}}\a),\eta\tensor\be\right\rangle =\left\langle \Delta(x)(\delta^{\frac{1}{2}}\delta^{-\frac{1}{2}}\z\tensor\delta^{\frac{1}{2}}\a),\eta\tensor\be\right\rangle \\
 & =\left\langle \Delta(\overline{x\delta^{\frac{1}{2}}})(\delta^{-\frac{1}{2}}\z\tensor\a),\eta\tensor\be\right\rangle =\left\langle (\om_{\delta^{-\frac{1}{2}}\z,\eta}\tensor\i)(\Delta(\overline{x\delta^{\frac{1}{2}}}))\a,\be\right\rangle .
\end{split}
\]
By \prettyref{prop:R_L_phi_psi} we have $R_{\om_{\delta^{-\frac{1}{2}}\z,\eta}}(\overline{x\delta^{\frac{1}{2}}}) =  (\om_{\delta^{-\frac{1}{2}}\z,\eta}\tensor\i)(\Delta(\overline{x\delta^{\frac{1}{2}}}))\in\mathcal{N}_{\varphi}$,
thus $R_{\om_{\z,\eta}}(x)=(\om_{\z,\eta}\tensor\i)(\Delta(x))\in\mathcal{N}_{\psi,0}$.
We infer that
\[
\gnsmap_{\psi}(R_{\om_{\z,\eta}}(x))=\gnsmap_{\varphi}(R_{\om_{\delta^{-\frac{1}{2}}\z,\eta}}(\overline{x\delta^{\frac{1}{2}}}))=\widetilde{R}_{\om_{\delta^{-\frac{1}{2}}\z,\eta}}^{\varphi}\gnsmap_{\varphi}(\overline{x\delta^{\frac{1}{2}}})=\widetilde{R}_{\om_{\delta^{-\frac{1}{2}}\z,\eta}}^{\varphi}\gnsmap_{\psi}(x).
\]
From the definition of $\varphi_\delta$ \citep{Vaes__Radon_Nikodym}, \citep[Theorem VI.1.26]{Takesaki__book_vol_2}, and the ultrastrong\textendash norm
(in fact, weak operator\textendash weak) closedness of $\gnsmap_{\psi}$, for
each $x\in\mathcal{N}_{\psi}$ we have $R_{\om_{\z,\eta}}(x)\in\mathcal{N}_{\psi}$
and $\gnsmap_{\psi}(R_{\om_{\z,\eta}}(x))=\widetilde{R}_{\om_{\delta^{-\frac{1}{2}}\z,\eta}}^{\varphi}\gnsmap_{\psi}(x)$.
\prettyref{prop:R_L_phi_psi} ends the proof.
\end{proof}
Let $(\sigma_{t}^{\varphi,\mathrm{u}})_{t\in\R}$ and $(\sigma_{t}^{\psi,\mathrm{u}})_{t\in\R}$
be the universal modular automorphism groups of the universal left
and right Haar weights of $\G$, respectively \citep[Proposition 8.6 and Definition 8.1]{Kustermans__LCQG_universal}.
Recall that ${(\tau_{t}^{\mathrm{u}})}_{t\in\R}$ and $\Rant^{\mathrm{u}}$
are the universal scaling group and unitary antipode of $\G$.
By \citep[Propositions 7.2, 9.2, 9.3 and 10.4]{Kustermans__LCQG_universal},
for every $t\in\R$,
\begin{align}
\Delta_{\mathrm{u}}\circ\sigma_{t}^{\varphi,\mathrm{u}} & =(\tau_{t}^{\mathrm{u}}\tensor\sigma_{t}^{\varphi,\mathrm{u}})\circ\Delta_{\mathrm{u}}, & \Delta_{\mathrm{u}}\circ\sigma_{t}^{\psi,\mathrm{u}} & =(\sigma_{t}^{\psi,\mathrm{u}}\tensor\tau_{-t}^{\mathrm{u}})\circ\Delta_{\mathrm{u}},\nonumber \\
\Delta_{\mathrm{u}}\circ\tau_{t}^{\mathrm{u}} & =(\tau_{t}^{\mathrm{u}}\tensor\tau_{t}^{\mathrm{u}})\circ\Delta_{\mathrm{u}}=(\sigma_{t}^{\varphi,\mathrm{u}}\tensor\sigma_{-t}^{\psi,\mathrm{u}})\circ\Delta_{\mathrm{u}}, & \Delta_{\mathrm{u}}\circ \Rant^{\mathrm{u}} & =\sigma\circ(\Rant^{\mathrm{u}}\tensor \Rant^{\mathrm{u}})\circ\Delta_{\mathrm{u}}.\label{eq:universal_comult_idents}
\end{align}
From these identities we easily get the following lemma (arguing as in  \citep[proof of Proposition 4]{Salmi_Skalski__idemp_states_LCQG_II}).
\begin{lem}
\label{lem:L_R_mod_prop}For every $\mu\in\CzU{\G}^{*}$ and $t\in\R$
we have 
\begin{align*}
L_{\mu}\circ\sigma_{t}^{\psi} & =\sigma_{t}^{\psi}\circ L_{\mu\circ\tau_{-t}^{\mathrm{u}}}, & L_{\mu}\circ\sigma_{t}^{\varphi} & =\tau_{t}\circ L_{\mu\circ\sigma_{t}^{\varphi,\mathrm{u}}}, & L_{\mu}\circ\tau_{t} & =\tau_{t}\circ L_{\mu\circ\tau_{t}^{\mathrm{u}}}=\sigma_{t}^{\varphi}\circ L_{\mu\circ\sigma_{-t}^{\psi,\mathrm{u}}},\\
R_{\mu}\circ\sigma_{t}^{\varphi} & =\sigma_{t}^{\varphi}\circ R_{\mu\circ\tau_{t}^{\mathrm{u}}}, & R_{\mu}\circ\sigma_{t}^{\psi} & =\tau_{-t}\circ R_{\mu\circ\sigma_{t}^{\psi,\mathrm{u}}}, & R_{\mu}\circ\tau_{t} & =\tau_{t}\circ R_{\mu\circ\tau_{t}^{\mathrm{u}}}=\sigma_{-t}^{\psi}\circ R_{\mu\circ\sigma_{t}^{\varphi,\mathrm{u}}},\\
 &  & L_{\mu}\circ \Rant & =\Rant\circ R_{\mu\circ \Rant^{\mathrm{u}}}.
\end{align*}
\end{lem}

\subsection{Symmetry} \label{subsect:Symmetry}
In this short subsection we associate the symmetry properties of the Hilbert space incarnations of the convolution operators $L_\mu$ and $R_\mu$ with the behaviour of $\mu$ with respect to the universal unitary antipode.

\begin{prop}
\label{prop:L_psi_adjoint_pre}Let $\mu\in\CzU{\G}^{*}$. For all
$a,b\in\mathcal{M}_{\psi}$ we have 
\begin{equation}
(\mathfrak{i}_{\psi}^{(1)}(a))(L_{\mu}(b))=(\mathfrak{i}_{\psi}^{(1)}(b))(L_{\mu\circ \Rant^{\mathrm{u}}}(a)).\label{eq:L_psi_adjoint_pre}
\end{equation}
Similarly, for all $a,b\in\mathcal{M}_{\varphi}$ we have
$(\mathfrak{i}_{\varphi}^{(1)}(a))(R_{\mu}(b))=(\mathfrak{i}_{\varphi}^{(1)}(b))(R_{\mu\circ \Rant^{\mathrm{u}}}(a))$.
\end{prop}

\begin{proof}We establish only the first assertion.
Assume first that $\mu\in\Lone{\G}$. Suppose that $a,b\in\mathcal{M}_{\psi,\infty}\subseteq\mathcal{T}_{\psi}$.
By \prettyref{prop:GL_i} \prettyref{enu:GL_i_Prop_2_13} and \prettyref{prop:R_L_phi_psi},
\begin{align*}
(\mathfrak{i}_{\psi}^{(1)}(a))(\mu\cdot b)&=\mathfrak{i}_{\psi}^{(1)}(a)\left((\i\tensor\mu)\Delta(b)\right)=\psi\left((\i\tensor\mu)(\Delta(b))\sigma_{-i/2}^{\psi}(a)\right)\\&=\mu\left[(\psi\tensor\i)\left(\Delta(b)(\sigma_{-i/2}^{\psi}(a)\tensor\one)\right)\right]
\end{align*}
($\psi\tensor\i$ is defined as an operator-valued weight). A similar
computation and the von Neumann algebraic analogue of \citep[Proposition 5.20]{Kustermans_Vaes__LCQG_C_star}
give 
\[
\begin{split}(\mathfrak{i}_{\psi}^{(1)}(b))((\mu\circ \Rant)\cdot a) & =(\mu\circ \Rant)\left[(\psi\tensor\i)\left(\Delta(a)(\sigma_{-i/2}^{\psi}(b)\tensor\one)\right)\right]=\mu\left[(\psi\tensor\i)\left((\sigma_{i/2}^{\psi}(a)\tensor\one)\Delta(b)\right)\right]\\
 & =\mu\left[(\psi\tensor\i)\left(\Delta(b)(\sigma_{-i/2}^{\psi}(a)\tensor\one)\right)\right]=(\mathfrak{i}_{\psi}^{(1)}(a))(\mu\cdot b)
\end{split}
\]
as in \prettyref{eq:L_psi_adjoint_pre}. Suppose that $a\in\mathcal{M}_{\psi}$
and $b\in\mathcal{M}_{\psi,\infty}$. Let $\left(a_{n}\right)_{n=1}^{\infty}$
be a sequence in $\mathcal{M}_{\psi,\infty}$ that converges ultraweakly
to $a$. Then $\mu\cdot b\in\mathcal{M}_{\psi}$ by \prettyref{prop:R_L_phi_psi}, and from the foregoing and \prettyref{prop:GL_i}
\prettyref{enu:GL_i_Prop_2_10},
\[
\begin{split}(\mathfrak{i}_{\psi}^{(1)}(a))(\mu\cdot b) & =(\mathfrak{i}_{\psi}^{(1)}(\mu\cdot b))(a)=\lim_{n\to\infty}(\mathfrak{i}_{\psi}^{(1)}(\mu\cdot b))(a_{n})=\lim_{n\to\infty}(\mathfrak{i}_{\psi}^{(1)}(a_{n}))(\mu\cdot b)\\
 & =\lim_{n\to\infty}(\mathfrak{i}_{\psi}^{(1)}(b))((\mu\circ \Rant)\cdot a_{n})=(\mathfrak{i}_{\psi}^{(1)}(b))((\mu\circ \Rant)\cdot a)
\end{split}
\]
by normality of $L_{\mu\circ \Rant}$. One can now prove \prettyref{eq:L_psi_adjoint_pre}
for all $a,b\in\mathcal{M}_{\psi}$ in a similar fashion.

For the general case, let $\mu\in\CzU{\G}^{*}$ and $a,b\in\mathcal{M}_{\psi}$.
For every $\om\in\Lone{\G}$ we have $\om\cdot b\in\mathcal{M}_{\psi}$,
and as $\Lone{\G}$ is an ideal in $\CzU{\G}^{*}$, 
\begin{equation}
\begin{split}(\mathfrak{i}_{\psi}^{(1)}(a))(\mu\cdot(\om\cdot b)) & =(\mathfrak{i}_{\psi}^{(1)}(a))((\mu\conv\om)\cdot b)=(\mathfrak{i}_{\psi}^{(1)}(b))[((\mu\conv\om)\circ \Rant)\cdot a]\\
 & =(\mathfrak{i}_{\psi}^{(1)}(b))[((\om\circ \Rant)\conv(\mu\circ \Rant^{\mathrm{u}}))\cdot a]=(\mathfrak{i}_{\psi}^{(1)}(b))[(\om\circ \Rant)\cdot((\mu\circ \Rant^{\mathrm{u}})\cdot a)].
\end{split}
\label{eq:L_psi_adjoint_pre_2}
\end{equation}
But for all $c\in\mathcal{M}_{\psi}$, by \prettyref{prop:GL_i} \prettyref{enu:GL_i_Prop_2_10} and by what we have already established above,
\[
(\mathfrak{i}_{\psi}^{(1)}(b))((\om\circ \Rant)\cdot c)=(\mathfrak{i}_{\psi}^{(1)}(c))(\om\cdot b)=(\mathfrak{i}_{\psi}^{(1)}(\om\cdot b))(c),
\]
so that $(\mathfrak{i}_{\psi}^{(1)}(b))\conv(\om\circ \Rant)=\mathfrak{i}_{\psi}^{(1)}(\om\cdot b)$.
Together with \prettyref{eq:L_psi_adjoint_pre_2} we get 
\begin{equation}
(\mathfrak{i}_{\psi}^{(1)}(a))(\mu\cdot(\om\cdot b))=(\mathfrak{i}_{\psi}^{(1)}(\om\cdot b))((\mu\circ \Rant^{\mathrm{u}})\cdot a).\label{eq:L_psi_adjoint_pre_3}
\end{equation}
Let $\left(b_{\iota}\right)_{\iota\in\mathcal{I}}$ be a net in $\Lone{\G}\cdot\mathcal{M}_{\psi}\subseteq\mathcal{M}_{\psi}$
that converges ultraweakly to $b$. Then $(\mu\circ \Rant^{\mathrm{u}})\cdot a\in\mathcal{M}_{\psi}$, and from
\prettyref{eq:L_psi_adjoint_pre_3} and \prettyref{prop:GL_i} \prettyref{enu:GL_i_Prop_2_10} we deduce that
\[
\begin{split}(\mathfrak{i}_{\psi}^{(1)}(a))(\mu\cdot b) & =\lim_{\iota\in\mathcal{I}}(\mathfrak{i}_{\psi}^{(1)}(a))(\mu\cdot b_{\iota})=\lim_{\iota\in\mathcal{I}}(\mathfrak{i}_{\psi}^{(1)}(b_{\iota}))((\mu\circ \Rant^{\mathrm{u}})\cdot a)\\
 & =\lim_{\iota\in\mathcal{I}}\mathfrak{i}_{\psi}^{(1)}\left((\mu\circ \Rant^{\mathrm{u}})\cdot a\right)(b_{\iota})=\mathfrak{i}_{\psi}^{(1)}\left((\mu\circ \Rant^{\mathrm{u}})\cdot a\right)(b)=\mathfrak{i}_{\psi}^{(1)}(b)\left((\mu\circ \Rant^{\mathrm{u}})\cdot a\right).\qedhere
\end{split}
\]
\end{proof}

Recall Definition \ref{def:Markov,KMS}.

\begin{cor}
\label{cor:R_L_KMS_sym}Let $\mu\in\CzU{\G}^{*}$. Then $L_{\mu}$
is KMS-symmetric with respect to $\psi$, namely 
\[
(\mathfrak{i}_{\psi}^{(1)}(b))(L_{\mu}(a))=(\mathfrak{i}_{\psi}^{(1)}(a))(L_{\mu}(b))\qquad(\forall_{a,b\in\mathcal{M}_{\psi}}),
\]
if and only if $\mu=\mu\circ \Rant^{\mathrm{u}}$, if and only if $R_{\mu}$ is KMS-symmetric with respect to $\varphi$.
\end{cor}

The terminology introduced below follows the analogy between GNS- and KMS-implementations.

\begin{defn}
For a von Neumann algebra $\Aa$ and an n.s.f.~weight $\varphi$
on $\Aa$, an operator $P\in B(\Aa)$ is called \emph{GNS-symmetric
with respect to $\varphi$} if $P(\mathcal{N}_{\varphi})\subseteq\mathcal{N}_{\varphi}$
and 
\begin{equation}
\left\langle \gnsmap_{\varphi}(Pa),\gnsmap_{\varphi}(b)\right\rangle =\left\langle \gnsmap_{\varphi}(a),\gnsmap_{\varphi}(Pb)\right\rangle \qquad(\forall_{a,b\in\mathcal{N}_{\varphi}}).\label{eq:GNS_symmetry}
\end{equation}
Condition \prettyref{eq:GNS_symmetry} is equivalent to symmetry of
the respective GNS-implementation: the map over the pre-Hilbert space $\gnsmap_{\varphi}(\mathcal{N}_{\varphi})$ given by
$\gnsmap_{\varphi}(a)\mapsto\gnsmap_{\varphi}(Pa)$, $a\in\mathcal{N}_{\varphi}$.
\end{defn}

Analogously to the subalgebra $\LoneSharp{\G}\subseteq\Lone{\G}$
we define $\CzUSSharp{\G}\subseteq\CzU{\G}^{*}$ to be the set of
all $\mu\in\CzU{\G}^{*}$ such that there exists a (necessarily unique)
$\mu^{\sharp}\in\CzU{\G}^{*}$ satisfying $\mu^{\sharp}(x)=\overline{\mu}(\Sant^{\mathrm{u}}(x))$
for every $x\in D(\Sant^{\mathrm{u}})$. Then $\CzUSSharp{\G}$ is a subalgebra
of $\CzU{\G}^{*}$ with $\mu\mapsto\mu^{\sharp}$ being an involution. 
Note that $\CzUSSharp{\G}$ is $w^*$-dense in $\CzU{\G}^{*}$ by the construction \eqref{eq:CzU_star_tau_smear}, 
because if $\mu\in\CzU{\G}^{*}$ is entire analytic with respect to the adjoint of $\tau^\mathrm{u}$ then $\overline{\mu} \circ \Rant^\mathrm{u} \in \CzUSSharp{\G}$.

\begin{prop}
Let $\mu\in\CzU{\G}^{*}$. Then $R_{\mu}$, resp.~$L_{\mu}$, is
GNS-symmetric with respect to $\varphi$, resp.~$\psi$, if and only
if $\overline{\mu}\in\CzU{\G}_{\sharp}^{*}$ and $\left(\overline{\mu}\right)^{\sharp}=\overline{\mu}$,
resp.~$\mu\in\CzU{\G}_{\sharp}^{*}$ and $\mu^{\sharp}=\mu$. 
\end{prop}

\begin{proof}
By \prettyref{prop:R_L_phi_psi}, $R_{\mu}$, resp.~$L_{\mu}$, is
GNS-symmetric with respect to $\varphi$, resp.~$\psi$, if and only
if the operator $(\overline{\mu}\tensor\i)(\Ww)$, resp.~$(\i\tensor\mu)(\vV)$,
is selfadjoint. Using the (formal) identities $(\Sant^{\mathrm{u}}\tensor\i)(\Ww)=\Ww^{*}$
(see \prettyref{eq:S_univ_prop}) and $(\i\tensor \Sant^{\mathrm{u}})(\vV)=\vV^{*}$
and the fact that the subspaces $\{(\i\tensor\om)(\Ww):\om\in\Lone{\hat{\G}}\}$
and $\{(\rho\tensor\i)(\vV):\rho\in\Lone{\hat{\G}'}\}$ are cores
for $\Sant^{\mathrm{u}}$, one can mimic the proof of \citep[Proposition 3.1]{Kustermans__LCQG_universal}
to derive the assertion. 
\end{proof}
\begin{rem}
\label{rem:GNS_KMS_sym}If $\mu\in\CzU{\G}_{+}^{*}$, then $R_{\mu}$
is GNS-symmetric if and only if $L_{\mu}$ is. Furthermore, in this
case, both operators are KMS-symmetric (with respect to $\varphi,\psi$,
respectively). Indeed, under the GNS-symmetry assumption, for every $x\in D(\tau_{-i}^{\mathrm{u}})$
we have $\mu(x)=\mu(\Sant^{\mathrm{u}}(x))$ and $\Sant^{\mathrm{u}}(x)=(\Rant^{\mathrm{u}}\circ\tau_{-i/2}^{\mathrm{u}})(x)=(\tau_{-i/2}^{\mathrm{u}}\circ \Rant^{\mathrm{u}})(x)\in D(\tau_{-i/2}^{\mathrm{u}})$,
thus $\mu(x)=\mu((\Sant^{\mathrm{u}}\circ \Sant^{\mathrm{u}})(x))=\mu(\tau_{-i}^{\mathrm{u}}(x))$.
This implies by a standard argument that $\mu$ is invariant under
$\tau^{\mathrm{u}}$: every $x\in\CzU{\G}$ that is entire analytic with respect to $\tau^{\mathrm{u}}$ satisfies
$\mu (\tau_t^{\mathrm{u}}(x)) = \mu(\tau_{-i}^{\mathrm{u}}(\tau_t^{\mathrm{u}}(x))) = \mu(\tau_{t-i}^{\mathrm{u}}(x))$ for all $t \in \R$, thus 
$\mu (\tau_z^{\mathrm{u}}(x)) = \mu(\tau_{z-i}^{\mathrm{u}}(x))$
for all $z \in \C$ by analyticity; since the function $\C \ni z \mapsto \mu(\tau_z^{\mathrm{u}}(x))$ is bounded on $\{z \in \C : 0 \le \Im z \le 1\}$, it is bounded on all of $\C$,
thus it is constant; and the set of all $x$ as above is dense in $\CzU{\G}$ (see also \citep[Proposition 4.8]{Cipriani_Franz_Kula__sym_Levy_proc} or \citep[Proposition 3.5]{FranzSkalskiTomatsu}). Therefore, every $x\in D(\Sant^{\mathrm{u}})$ satisfies
$\mu(x)=\mu((\tau_{-i/2}^{\mathrm{u}}\circ \Rant^{\mathrm{u}})(x))=\mu(\Rant^{\mathrm{u}}(x))$,
yielding that $\mu=\mu\circ \Rant^{\mathrm{u}}$.
\end{rem}

\subsection{$L_{\mu}$ and $R_{\mu}$ at the Hilbert space level} \label{subsect:Hilbert}

In this subsection we characterise (\prettyref{prop:L_R_form_GNS}) the operators of the form $L_{\mu},R_{\mu}$
in terms of their GNS Hilbert space versions with respect to $\psi,\varphi$,
respectively, introduced in \prettyref{prop:R_L_phi_psi}. We then prove that $L_{\mu},R_{\mu}$
admit KMS Hilbert space versions with respect to $\psi,\varphi$ (\prettyref{lem:L_R_Hilbert_2}),
and characterise the operators of these forms using the KMS language
(\prettyref{prop:L_R_form_KMS}). This is done via the KMS Hilbert
space versions of $L_{\mu},R_{\mu}$ with respect to the `wrong' weight
for suitable functionals $\mu$. The subsection ends with two convergence
results, \prettyref{lem:C_z_U_star_conv_Hilbert} and \prettyref{cor:L_R_Hilbert_close_to_identity}.
\begin{prop}
\label{prop:L_R_form_GNS}Let $T:\Linfty{\G}\to\Linfty{\G}$ be a
completely positive, normal operator.
\begin{enumerate}
\item \label{enu:L_R_form_GNS__R}Suppose that $T(\mathcal{N}_{\varphi})\subseteq\mathcal{N}_{\varphi}$
and that the map $\gnsmap_{\varphi}(a)\mapsto\gnsmap_{\varphi}(Ta)$,
$a\in\mathcal{N}_{\varphi}$, is bounded. Denote its bounded extension
to $\Ltwo{\G}$ by $\widetilde{T}^{\varphi}$. Then $T$ has the form
$R_{\mu}$, $\mu\in\CzU{\G}_{+}^{*}$, if and only if $\widetilde{T}^{\varphi}\in\Linfty{\hat{\G}}$.
\item \label{enu:L_R_form_GNS__L}Suppose that $T(\mathcal{N}_{\psi})\subseteq\mathcal{N}_{\psi}$
and that the map $\gnsmap_{\psi}(a)\mapsto\gnsmap_{\psi}(Ta)$, $a\in\mathcal{N}_{\psi}$,
is bounded. Denote its bounded extension to $\Ltwo{\G}$ by $\widetilde{T}^{\psi}$.
Then $T$ has the form $L_{\mu}$, $\mu\in\CzU{\G}_{+}^{*}$, if and
only if $\widetilde{T}^{\psi}\in\Linfty{\hat{\G}}'$.
\end{enumerate}
If $\hat{\G}$ is amenable then the same assertions are true for normal, completely
bounded $T$ and $\mu\in\CzU{\G}^{*}$ (not necessarily positive).
\end{prop}

\begin{proof}
We prove only \prettyref{enu:L_R_form_GNS__R}, as the proof of \prettyref {enu:L_R_form_GNS__L} is very similar. The assertion follows
easily from Propositions \ref{prop:R_L_phi_psi} and \ref{prop:R_L_phi_psi_inverted}.
Let $A:=\{\om_{\z,\eta}:\z\in D(\delta^{\frac{1}{2}}),\eta\in\Ltwo{\G}\}$.
Since $\delta^{\frac{1}{2}}$ is densely defined, $A$ is dense in
$\Lone{\G}$. By \prettyref{thm:L_R} \prettyref{enu:L_R_CP_mult}
and \prettyref{enu:L_R_CB_mult_amen}, proving that $T$ has the desired
form amounts to showing that it commutes with all operators $L_{\om}$,
$\om\in\Lone{\G}$. By density and continuity, this holds if and only
if for every $a\in\mathcal{N}_{\varphi}$ and $\om\in A$, we have
$(T\circ L_{\om})(a)=(L_{\om}\circ T)(a)$, which is the same as $\widetilde{T}^{\varphi}\widetilde{L}_{\om}^{\varphi}=\widetilde{L}_{\om}^{\varphi}\widetilde{T}^{\varphi}$
on $\gnsmap_{\varphi}(\mathcal{N}_{\varphi})$, equivalently everywhere.
So $T$ has the desired form if and only if $\widetilde{T}^{\varphi}$ belongs to $\{\widetilde{L}_{\om}^{\varphi}:\om\in A\}'=\{\widetilde{L}_{\om_{\delta^{\frac{1}{2}}\z,\eta}}^{\psi}:\z\in D(\delta^{\frac{1}{2}}),\eta\in\Ltwo{\G}\}'=\{(\i\tensor\om_{\delta^{\frac{1}{2}}\z,\eta})(V):\z\in D(\delta^{\frac{1}{2}}),\eta\in\Ltwo{\G}\}'=\Linfty{\hat{\G}}''=\Linfty{\hat{\G}}$.
\end{proof}
\begin{rem}
\label{rem:CFK_L_map_GNS}The last result is an analogue of \citep[Theorem 7.3, (1)$\iff$(4)]{Cipriani_Franz_Kula__sym_Levy_proc}
for general, not necessarily compact, locally compact quantum groups. Let us explain this fact,
using the conventions of \citep{Cipriani_Franz_Kula__sym_Levy_proc}
freely (in particular, sesquilinear forms will be linear in the right
variable and we use the Sweedler notation). Let $\QG$ be a compact quantum group,  $h$ the Haar state on $\mathrm{C}(\QG)$, and $\Pol{\G}\subseteq\mathrm{C}(\G)$ the
canonical Hopf $*$-algebra. Assume that the linear map $L:\Pol{\G}\to \Pol{\G}$
from \citep[Theorem 7.3]{Cipriani_Franz_Kula__sym_Levy_proc} extends to a normal map on $\Linfty{\QG}$,
still denoted by $L$, and to a bounded map on $\Ltwo{\QG}$,
denoted by $\widetilde{L}$. Consider condition (4) of that Theorem,
saying that 
\begin{equation}
\mathcal{Q}(a,b)\one=(m_{*}\tensor\mathcal{Q})(\Delta(a),\Delta(b))\qquad(\forall_{a,b\in\Pol{\G}}),\label{eq:CFK_7_1}
\end{equation}
where $m_{*},\mathcal{Q}$ are the sesquilinear maps from $\Pol{\G}$ to
$\Pol{\G},\C$, respectively, given by $m_{*}(a,b)=a^{*}b$ and $\mathcal{Q}(a,b):=-h(a^{*}L(b))$
($a,b\in\Pol{\G}$). For $\z,\eta\in\Ltwo{\G}$, apply $\om_{\z,\eta}$
to both sides of \prettyref{eq:CFK_7_1}. Then
\[
\begin{split}((\om_{\z,\eta}m_{*})\tensor\mathcal{Q})(\Delta(a),\Delta(b)) & =((\om_{\z,\eta}m_{*})\tensor\mathcal{Q})(a_{(1)}\tensor a_{(2)},b_{(1)}\tensor b_{(2)})\\
 & =-\left\langle a_{(1)}\z,b_{(1)}\eta\right\rangle h(a_{(2)}^{*}L(b_{(2)}))=-\left\langle a_{(1)}\z,b_{(1)}\eta\right\rangle \left\langle \gnsmap_{h}(a_{(2)}),\widetilde{L}\gnsmap_{h}(b_{(2)})\right\rangle \\
 & =-\left\langle a_{(1)}\z\tensor\gnsmap_{h}(a_{(2)}),(\one\tensor\widetilde{L})(b_{(1)}\eta\tensor\gnsmap_{h}(b_{(2)}))\right\rangle .
\end{split}
\]
Using the identity $W^{*}(\xi\tensor\gnsmap_{h}(c))=c_{(1)}\xi\tensor\gnsmap_{h}(c_{(2)})$
for $c\in \Pol{\G}$ and $\xi\in\Ltwo{\G}$, we get 
\[
((\om_{\z,\eta}m_{*})\tensor\mathcal{Q})(\Delta(a),\Delta(b))=-\left\langle W^{*}(\z\tensor\gnsmap_{h}(a)),(\one\tensor\widetilde{L})W^{*}(\eta\tensor\gnsmap_{h}(b))\right\rangle .
\]
On the other hand, $\mathcal{Q}(a,b)\om_{\z,\eta}(\one)=-\left\langle \z\tensor\gnsmap_{h}(a),(\one\tensor\widetilde{L})(\eta\tensor\gnsmap_{h}(b))\right\rangle $.
So it is evident that \prettyref{eq:CFK_7_1} holds in and only if
$\one\tensor\widetilde{L}$ commutes with $W$, that is to say, $\widetilde{L}$
belongs to $\Linfty{\hat{\G}}'$.
\end{rem}

We now turn to $2$-integrability of the maps $R_{\mu}$ and $L_{\mu}$.
From \prettyref{cor:R_L_KMS_sym} and \prettyref{thm:GL_4_7_forward}
we know that if $\mu$ is an $\Rant^{\mathrm{u}}$-invariant state of
$\CzU{\G}$ then these two Markov operators are KMS-symmetric, thus
$2$-integrable, with respect to $\varphi$ and $\psi$, respectively.
It turns out that these assumptions are not necessary. Recall the notation introduced in Definition \ref{def:Markov,KMS} and in Proposition \ref{prop:R_L_phi_psi}, together with the notion of the commutant quantum group. 
\begin{lem}
\label{lem:L_R_Hilbert_2}Let $\mu\in\CzU{\G}^{*}$. 
\begin{enumerate}
\item \label{enu:L_R_Hilbert_2__1}We have $\widetilde{R}_{\mu}^{\varphi}\in D(\hat{\tau}_{-i/2})$,
$\widetilde{L}_{\mu}^{\psi}\in D(\hat{\tau}_{-i/2}')$ and $\Vert\hat{\tau}_{-i/4}(\widetilde{R}_{\mu}^{\varphi})\Vert\leq\max\bigl(\bigl\Vert\widetilde{R}_{\mu}^{\varphi}\bigr\Vert,\bigl\Vert\widetilde{R}_{\overline{\mu}}^{\varphi}\bigr\Vert\bigr)\leq\left\Vert \mu\right\Vert $,
$\Vert\hat{\tau}_{-i/4}'(\widetilde{L}_{\mu}^{\psi})\Vert\leq\max\bigl(\bigl\Vert\widetilde{L}_{\mu}^{\psi}\bigr\Vert,\bigl\Vert\widetilde{L}_{\overline{\mu}}^{\psi}\bigr\Vert\bigr)\leq\left\Vert \mu\right\Vert $.
\item \label{enu:L_R_Hilbert_2__2}The operator $R_{\mu}$, resp.~$L_{\mu}$,
is $2$-integrable with respect to $\varphi$, resp.~$\psi$, and
the corresponding Hilbert space operators are $\widetilde{R}_{\mu}^{(2,\varphi)}=\hat{\tau}_{-i/4}(\widetilde{R}_{\mu}^{\varphi})$
and $\widetilde{L}_{\mu}^{(2,\psi)}=\hat{\tau}_{-i/4}'(\widetilde{L}_{\mu}^{\psi})$.
\item \label{enu:L_R_Hilbert_2__3}The operators $\widetilde{R}_{\mu}^{(2,\varphi)},\widetilde{L}_{\mu}^{(2,\psi)}$
belong to $\M{\Cz{\hat{\G}}},\M{\Cz{\hat{\G}'}}$, respectively; and
if $\mu\in\Lone{\G}$, they belong to $\Cz{\hat{\G}},\Cz{\hat{\G}'}$, respectively. 
\item \label{enu:L_R_Hilbert_2__4}The following conditions are equivalent:
\begin{itemize}
	\item $\mu$ is invariant under $\tau^{\mathrm{u}}$;
	\item  $\widetilde{R}_{\mu}^{(2,\varphi)}=\widetilde{R}_{\mu}^{\varphi}$;
	 \item $\widetilde{L}_{\mu}^{(2,\psi)}=\widetilde{L}_{\mu}^{\psi}$.
\end{itemize}
Thus, if $\mu$ is positive, and if $R_{\mu}$ is GNS-symmetric with
respect to $\varphi$ \textendash{} equivalently, $L_{\mu}$ is GNS-symmetric
with respect to $\psi$ \textendash{} then $\widetilde{R}_{\mu}^{(2,\varphi)}=\widetilde{R}_{\mu}^{\varphi}$
and $\widetilde{L}_{\mu}^{(2,\psi)}=\widetilde{L}_{\mu}^{\psi}$.
\end{enumerate}
\end{lem}

\begin{proof}
\prettyref{enu:L_R_Hilbert_2__1} By \prettyref{prop:R_L_phi_psi},
$\widetilde{R}_{\mu}^{\varphi}=(\mu\tensor\i)(\Ww^{*})=(\i\tensor\mu)(\hat{\wW})$
and $\widetilde{L}_{\mu}^{\psi}=(\i\tensor\mu)(\vV)$. Hence, $\widetilde{R}_{\mu}^{\varphi}\in D(\hat{\Sant})$
and $\hat{\Sant}(\widetilde{R}_{\mu}^{\varphi})=(\i\tensor\mu)(\hat{\wW}^{*})=\bigl(\widetilde{R}_{\overline{\mu}}^{\varphi}\bigr)^{*}$
by \prettyref{lem:antipode_W_extended}. Using the equality $\hat{\Sant}=\hat{\Rant}\circ\hat{\tau}_{-i/2},$
we have $\widetilde{R}_{\mu}^{\varphi}\in D(\hat{\tau}_{-i/2})$ and
$\hat{\tau}_{-i/2}(\widetilde{R}_{\mu}^{\varphi})=\hat{\Rant}\bigl(\bigl(\widetilde{R}_{\overline{\mu}}^{\varphi}\bigr)^{*}\bigr)$.
Consider the strip $D:=\left\{ z\in\C:-\frac{1}{2}\le\Im z\le0\right\} $,
and define $f:D\to\Linfty{\hat{\G}}$ by $f(z):=\hat{\tau}_{z}(\widetilde{R}_{\mu}^{\varphi})$,
$z\in D$. Then $f$ is bounded and ultraweakly
continuous on $D$ and analytic on $D^{\circ}$. Furthermore, $\left\Vert f(t)\right\Vert =\bigl\Vert\widetilde{R}_{\mu}^{\varphi}\bigr\Vert$
and $\Vert f(t-\frac{i}{2})\Vert=\bigl\Vert\widetilde{R}_{\overline{\mu}}^{\varphi}\bigr\Vert$
for all $t\in\R$. The Phragmen\textendash Lindel\"{o}f three lines
theorem implies that $\left\Vert f(z)\right\Vert \leq\max\bigl(\bigl\Vert\widetilde{R}_{\mu}^{\varphi}\bigr\Vert,\bigl\Vert\widetilde{R}_{\overline{\mu}}^{\varphi}\bigr\Vert\bigr)\leq\left\Vert \mu\right\Vert $
for all $z\in D$, and in particular for $z=-\frac{i}{4}$. 

The assertion about $\widetilde{L}_{\mu}^{\psi}$ follows by using
the equalities $\vV=\wW_{\widehat{\G^{\op}}}$ and $\widehat{\G^{\op}}=\hat{\G}'$
\citep[Section 4]{Kustermans_Vaes__LCQG_von_Neumann} (indeed, it
is proved there that $V=W_{\widehat{\G^{\op}}}$, and equality for
the ambient semi-universal operators follows from \citep[Proposition 6.6]{Kustermans__LCQG_universal}).

\prettyref{enu:L_R_Hilbert_2__2} Denote by $\tau^{\mathrm{u}*}$
the automorphism group $\bigl((\tau_{t}^{\mathrm{u}})^{*}\bigr)_{t\in\R}$
of $\CzU{\G}^{*}$. We first prove the following claims: if $\mu\in\CzU{\G}^{*}$
is entire analytic with respect to $\tau^{\mathrm{u}*}$, then $R_{\mu}$
is $2$-integrable with respect to $\varphi$ and the corresponding
operator $\widetilde{R}_{\mu}^{(2,\varphi)}\in B(\Ltwo{\G})$ equals
$\widetilde{R}_{(\tau^{\mathrm{u}*})_{i/4}(\mu)}^{\varphi}=\hat{\tau}_{-i/4}(\widetilde{R}_{\mu}^{\varphi})$,
and $L_{\mu}$ is $2$-integrable with respect to $\psi$ and the
corresponding operator $\widetilde{L}_{\mu}^{(2,\psi)}\in B(\Ltwo{\G})$
equals $\widetilde{L}_{(\tau^{\mathrm{u}*})_{-i/4}(\mu)}^{\psi}=\hat{\tau}_{-i/4}'(\widetilde{L}_{\mu}^{\psi})$.

We verify only the claim about $L_{\mu}$, assuming that $\mu$ is as in the last paragraph. Write $\mu(\cdot):\C\to\CzU{\G}^{*}$
for the unique entire analytic function that satisfies $\mu(t)=\mu\circ\tau_{t}^{\mathrm{u}}$
for all $t\in\R$. By \prettyref{lem:L_R_mod_prop}, $L_{\mu(t)}\circ\sigma_{t}^{\psi}=\sigma_{t}^{\psi}\circ L_{\mu}$
for every $t\in\R$. Consequently, if $a\in \mathcal{T}_\psi \subseteq\Linfty{\G}$,
then also $L_{\mu}(a)\in \mathcal{T}_\psi$ and $(L_{\mu(z)}\circ\sigma_{z}^{\psi})(a)=(\sigma_{z}^{\psi}\circ L_{\mu})(a)$ for all $z\in \C$ (use \prettyref{prop:R_L_phi_psi}).
Thus, by \prettyref{prop:R_L_phi_psi} again and \prettyref{lem:i_2_gns} \prettyref{enu:i_2_gns__2},
for every $a\in\mathcal{M}_{\psi}\cap \mathcal{T}_\psi$ we have $L_{\mu}(a)\in\mathcal{M}_{\psi}\cap \mathcal{T}_\psi$ and
\begin{equation}
\begin{split}\mathfrak{i}_{\psi}^{(2)}(L_{\mu}(a)) & =\gnsmap_{\psi}((\sigma_{-i/4}^{\psi}\circ L_{\mu})(a))=\gnsmap_{\psi}((L_{\mu(-i/4)}\circ\sigma_{-i/4}^{\psi})(a))\\
 & =\widetilde{L}_{\mu(-i/4)}^{\psi}(\gnsmap_{\psi}(\sigma_{-i/4}^{\psi}(a)))=\widetilde{L}_{\mu(-i/4)}^{\psi}(\mathfrak{i}_{\psi}^{(2)}(a)).
\end{split}
\label{eq:L_R_Hilbert_2__L_psi}
\end{equation}
Let $a\in\mathcal{M}_{\psi}$ (or even $a\in\mathcal{M}_{\psi}^{(2)}$).
Pick a net $\left(a_{\l}\right)_{\l\in\mathcal{I}}$ in $\mathcal{M}_{\psi,\infty}\subseteq\mathcal{M}_{\psi}\cap\mathcal{T}_{\psi}$
such that $a_{\l}\xrightarrow[\l\in\mathcal{I}]{}a$ ultraweakly and
$\mathfrak{i}_{\psi}^{(2)}(a_{\l})\xrightarrow[\l\in\mathcal{I}]{}\mathfrak{i}_{\psi}^{(2)}(a)$
weakly using \prettyref{prop:i_p_core}. Then $L_{\mu}(a_{\l})\xrightarrow[\l\in\mathcal{I}]{}L_{\mu}(a)$
ultraweakly, and since $\widetilde{L}_{\mu(-i/4)}^{\psi}$ is bounded,
$\mathfrak{i}_{\psi}^{(2)}(L_{\mu}(a_{\l}))\xrightarrow[\l\in\mathcal{I}]{}\widetilde{L}_{\mu(-i/4)}^{\psi}(\mathfrak{i}_{\psi}^{(2)}(a))$
weakly by \prettyref{eq:L_R_Hilbert_2__L_psi}. Hence, \prettyref{prop:GL_i}
\prettyref{enu:GL_i_Prop_2_12} implies that $L_{\mu}(a)\in D(\overline{\mathfrak{i}_{\psi}^{(2)}})$
and $\overline{\mathfrak{i}_{\psi}^{(2)}}(L_{\mu}(a))=\widetilde{L}_{\mu(-i/4)}^{\psi}(\mathfrak{i}_{\psi}^{(2)}(a))$,
so that $L_{\mu}$ is $2$-integrable with respect to $\psi$ and
$\widetilde{L}_{\mu}^{(2,\psi)}=\widetilde{L}_{(\tau^{\mathrm{u}*})_{-i/4}(\mu)}^{\psi}$.
Using the formula $(\hat{\tau}_{t}'\tensor\tau_{-t}^{\mathrm{u}})(\vV)=\vV$
for every $t\in\R$ we deduce that $\widetilde{L}_{\mu}^{(2,\psi)}=\hat{\tau}_{-i/4}'((\i\tensor\mu)(\vV))=\hat{\tau}_{-i/4}'(\widetilde{L}_{\mu}^{\psi})$,
and the claim follows.

We now treat the general case.
Let $\mu\in\CzU{\G}^{*}$.
Consider the sequence $\left(\mu_{n}\right)_{n=1}^{\infty}$ in $\CzU{\G}^{*}$
defined by \prettyref{eq:CzU_star_tau_smear} and its properties listed in the succeeding
text. By the foregoing, for every $n\in\N$, $L_{\mu_{n}}$ is $2$-integrable
with respect to $\psi$ and $\widetilde{L}_{\mu_{n}}^{(2,\psi)}=\hat{\tau}_{-i/4}'(\widetilde{L}_{\mu_{n}}^{\psi})$.
Since $\widetilde{L}_{\mu_{n}}^{\psi}\xrightarrow[n\to\infty]{}\widetilde{L}_{\mu}^{\psi}$
ultraweakly, and by \prettyref{enu:L_R_Hilbert_2__1}, $(\hat{\tau}_{-i/4}'(\widetilde{L}_{\mu_{n}}^{\psi}))_{n\in\N}$
is bounded by $\left\Vert \mu\right\Vert $, we have $\hat{\tau}_{-i/4}'(\widetilde{L}_{\mu_{n}}^{\psi})\xrightarrow[n\to\infty]{}\hat{\tau}_{-i/4}'(\widetilde{L}_{\mu}^{\psi})$
ultraweakly. Let $a\in\mathcal{M}_{\psi}$. Then $L_{\mu_{n}}(a)\xrightarrow[n\to\infty]{}L_{\mu}(a)$
ultraweakly by \prettyref{lem:C_z_U_star_conv} (because $\left\Vert \mu_{n}\right\Vert \xrightarrow[n\to\infty]{}\left\Vert \mu\right\Vert $)
and $\mathfrak{i}_{\psi}^{(2)}(L_{\mu_{n}}(a))=\hat{\tau}_{-i/4}'(\widetilde{L}_{\mu_{n}}^{\psi})(\mathfrak{i}_{\psi}^{(2)}(a))\xrightarrow[n\to\infty]{}\hat{\tau}_{-i/4}'(\widetilde{L}_{\mu}^{\psi})(\mathfrak{i}_{\psi}^{(2)}(a))$
weakly. Thus, \prettyref{prop:GL_i} \prettyref{enu:GL_i_Prop_2_12}
implies that $\mathfrak{i}_{\psi}^{(2)}(L_{\mu}(a))=\hat{\tau}_{-i/4}'(\widetilde{L}_{\mu}^{\psi})(\mathfrak{i}_{\psi}^{(2)}(a))$.
This proves that $L_{\mu}$ is $2$-integrable with respect to $\psi$
and $\widetilde{L}_{\mu}^{(2,\psi)}=\hat{\tau}_{-i/4}'(\widetilde{L}_{\mu}^{\psi})$. 

\prettyref{enu:L_R_Hilbert_2__3} Recall that $\hat{\tau}$ restricts
to an automorphism group of the C$^{*}$-algebra $\Cz{\hat{\G}}$,
and also $\hat{\Rant}(\Cz{\hat{\G}})=\Cz{\hat{\G}}$. By \prettyref{prop:R_L_phi_psi} we have $\widetilde{R}_{\mu}^{\varphi}\in\M{\Cz{\hat{\G}}}$, so
also $\hat{\tau}_{-i/2}(\widetilde{R}_{\mu}^{\varphi})=\hat{\Rant}\bigl(\bigl(\widetilde{R}_{\overline{\mu}}^{\varphi}\bigr)^{*}\bigr)\in\M{\Cz{\hat{\G}}}$.
Therefore, \citep[Proposition 2.44]{Kustermans__one_param_rep} implies
that $\widetilde{R}_{\mu}^{\varphi}$ belongs to $D(\hat{\tau}_{-i/2})$
in the sense of the strict extension of $\hat{\tau}$ from $\Cz{\hat{\G}}$
to $\M{\Cz{\hat{\G}}}$, and the two meanings of $\hat{\tau}_{-i/2}(\widetilde{R}_{\mu}^{\varphi})$
coincide. As a result, $\widetilde{R}_{\mu}^{(2,\varphi)}=\hat{\tau}_{-i/4}(\widetilde{R}_{\mu}^{\varphi})\in\M{\Cz{\hat{\G}}}$.
Similarly, if $\mu\in\Lone{\G}$, then $\widetilde{R}_{\mu}^{\varphi},\hat{\Rant}\bigl(\bigl(\widetilde{R}_{\overline{\mu}}^{\varphi}\bigr)^{*}\bigr)\in\Cz{\hat{\G}}$,
and we deduce from \citep[Proposition 1.24]{Kustermans__one_param_rep}
the desired conclusion.

\prettyref{enu:L_R_Hilbert_2__4} If $\mu$ is invariant under $\tau^{\mathrm{u}}$,
then the claim in the proof of \prettyref{enu:L_R_Hilbert_2__2} implies
that $\widetilde{R}_{\mu}^{(2,\varphi)}=\widetilde{R}_{\mu}^{\varphi}$.
Conversely, by \prettyref{enu:L_R_Hilbert_2__2}, the assumption $\widetilde{R}_{\mu}^{(2,\varphi)}=\widetilde{R}_{\mu}^{\varphi}$
means that $\widetilde{R}_{\mu}^{\varphi}=(\mu\tensor\i)(\Ww^{*})$
is invariant under $\hat{\tau}$, and the identity $(\tau_{t}^{\mathrm{u}}\tensor\hat{\tau}_{t})(\Ww)=\Ww$
for every $t\in\R$ yields that $\mu$ is invariant under $\tau^{\mathrm{u}}$.
Lastly, by \prettyref{rem:GNS_KMS_sym}, this is the case when $\mu\in\CzU{\G}_{+}^{*}$
and $R_{\mu}$ is GNS-symmetric with respect to $\varphi$.
\end{proof}

For brevity we formulate the next result only for the left convolution operators; it also has a natural right version.
Recall that $\nu$ denotes the \emph{scaling constant} of $\G$. Define an injective, positive, selfadjoint operator $P$ on $\Ltwo{\G}$ by $P^{it}(\gnsmap_\varphi(a)) := \nu^{t/2} \gnsmap_\varphi(\tau_t(a))$ for all $t\in\R$ and $a\in \mathcal{N}_\varphi$. The operators $\delta,P$ commute and
$(\tau_t)_{t\in\R} = (\Ad{P^{it}})_{t\in\R}$.
\begin{lem}
\label{lem:L_R_Hilbert_2_inverted}Let $e(\cdot),f(\cdot)$ be the spectral measures of $P,\delta$, respectively. Define a subspace $Z$ of $\Ltwo{\G}$ by 
\[
Z:=\bigcup_{n,m=1}^\infty e\left(\left[1/n,n\right]\right) f\left(\left[1/m,m\right]\right) \Ltwo{\G}.
\]
\begin{enumerate}
\item \label{enu:i_2_phi_psi__1}$Z$ is dense in $\Ltwo{\G}$, and for all $z,w\in\C$, $Z \subseteq D(P^z \delta^w)$ and $P^z \delta^w Z$ is dense in $\Ltwo{\G}$.
\item \label{enu:i_2_phi_psi__2}For each $\z,\eta\in Z$, $L_{\om_{\z,\eta}}$ is $2$-integrable
with respect to $\varphi$ and $\widetilde{L}_{\om_{\z,\eta}}^{(2,\varphi)}=\widetilde{L}_{\om_{P^{-\frac{1}{4}}\delta^{\frac{1}{4}}\z,P^{\frac{1}{4}}\delta^{\frac{1}{4}}\eta}}^{\psi}$.
\end{enumerate}
\end{lem}

\begin{proof}
\prettyref{enu:i_2_phi_psi__1} Injectivity of $P$ and $\delta$ implies the first assertion.
Since, in addition, $P$ and $\delta$ also commute, the second assertion follows. Notice that 
for every $\xi\in Z$, the function $\C^{2}\ni(z,w)\mapsto P^{z}\delta^{w}\xi=\delta^{w}P^{z}\xi$
is (well defined and) entire analytic. 

\prettyref{enu:i_2_phi_psi__2} We first show that for every 
$a\in\mathcal{M}_{\varphi}\cap\mathcal{T}_{\varphi}$ we have 
\begin{equation}
\mathfrak{i}_{\varphi}^{(2)}(L_{\om_{\z,\eta}}(a))=\widetilde{L}_{\om_{P^{-\frac{1}{4}}\delta^{\frac{1}{4}}\z,P^{\frac{1}{4}}\delta^{\frac{1}{4}}\eta}}^{\psi}(\mathfrak{i}_{\varphi}^{(2)}(a)).\label{eq:L_R_Hilbert_2_inverted}
\end{equation}
So let $a\in\mathcal{M}_{\varphi}\cap\mathcal{T}_{\varphi}$. For every $t\in\R$, using that $\sigma_{t}^{\psi}(\cdot)=\delta^{it}\sigma_{t}^{\varphi}(\cdot)\delta^{-it}$,
$\Delta(\delta^{it})=\delta^{it}\tensor\delta^{it}$, 
$\sigma_s^\varphi(\delta^{it}) = \sigma_s^\psi(\delta^{it}) =\nu^{ist}\delta^{it}$ for $s\in\R$,
and $L_{\om}\circ\sigma_{t}^{\psi}=\sigma_{t}^{\psi}\circ L_{\om\circ\tau_{-t}}$
for $\om\in\Lone{\G}$ (\prettyref{lem:L_R_mod_prop}), we have 
\[
\begin{split}\sigma_{t}^{\varphi}(L_{\om_{\z,\eta}}(a)) & =\delta^{-it}\sigma_{t}^{\psi}(L_{\om_{\z,\eta}}(a))\delta^{it}=\sigma_{t}^{\psi}(\delta^{-it}L_{\om_{\z,\eta}}(a)\delta^{it})\\
 & =\sigma_{t}^{\psi}(L_{\om_{\delta^{-it}\z,\delta^{-it}\eta}}(\delta^{-it}a\delta^{it}))=L_{\om_{P^{-it}\delta^{-it}\z,P^{-it}\delta^{-it}\eta}}(\sigma_{t}^{\psi}(\delta^{-it}a\delta^{it})) \\
 &
 = L_{\om_{P^{-it}\delta^{-it}\z,P^{-it}\delta^{-it}\eta}}(\sigma_{t}^{\varphi}(a)).
\end{split}
\]
Hence, $L_{\om_{\z,\eta}}(a)$ is entire analytic with respect
to $\sigma^{\varphi}$, and for all $z\in\C$ 
\[
\sigma_{z}^{\varphi}(L_{\om_{\z,\eta}}(a))=L_{\om_{P^{-iz}\delta^{-iz}\z,P^{-i\overline{z}}\delta^{-i\overline{z}}\eta}}(\sigma_{z}^{\varphi}(a)).
\]
Therefore, using \prettyref{prop:R_L_phi_psi_inverted} \prettyref{enu:R_L_phi_psi_inverted__1} and \prettyref{enu:R_L_phi_psi_inverted__2} and \prettyref{lem:i_2_gns} \prettyref{enu:i_2_gns__2} we obtain
$L_{\om_{\z,\eta}}(a)\in\mathcal{M}_{\varphi}\cap\mathcal{T}_{\varphi}$
and
\[
\begin{split}
\mathfrak{i}_{\varphi}^{(2)}(L_{\om_{\z,\eta}}(a)) & =\gnsmap_{\varphi}(\sigma_{-i/4}^{\varphi}(L_{\om_{\z,\eta}}(a)))=\gnsmap_{\varphi}(L_{\om_{P^{-\frac{1}{4}}\delta^{-\frac{1}{4}}\z,P^{\frac{1}{4}}\delta^{\frac{1}{4}}\eta}}(\sigma_{-i/4}^{\varphi}(a))) \\
 &
 = \widetilde{L}_{\om_{P^{-\frac{1}{4}}\delta^{\frac{1}{4}}\z,P^{\frac{1}{4}}\delta^{\frac{1}{4}}\eta}}^\psi
 (\gnsmap_{\varphi}(\sigma_{-i/4}^{\varphi}(a)))
 = \widetilde{L}_{\om_{P^{-\frac{1}{4}}\delta^{\frac{1}{4}}\z,P^{\frac{1}{4}}\delta^{\frac{1}{4}}\eta}}^\psi (\mathfrak{i}_{\varphi}^{(2)}(a)),
\end{split}
\]
proving \prettyref{eq:L_R_Hilbert_2_inverted}. Relying on this, on \prettyref{prop:i_p_core} and on \prettyref{prop:GL_i}
\prettyref{enu:GL_i_Prop_2_12} again we conclude that \prettyref{eq:L_R_Hilbert_2_inverted}
holds for all $a\in\mathcal{M}_{\varphi}$.
\end{proof}

The following result provides a key characterisation of the KMS-implementations of the convolution operators.

\begin{prop}
\label{prop:L_R_form_KMS}Let $T:\Linfty{\G}\to\Linfty{\G}$ be a
completely positive, normal operator that is $2$-integrable with
respect to $\varphi$, resp.~$\psi$, and let $\widetilde{T}^{(2,\varphi)}$,
resp.~$\widetilde{T}^{(2,\psi)}$, be the associated bounded operator
on $\Ltwo{\G}$. Then $T=R_{\mu}$, resp.~$T=L_{\mu}$,
for some $\mu\in\CzU{\G}_{+}^{*}$, if and only if $\widetilde{T}^{(2,\varphi)}\in\Linfty{\hat{\G}}$,
resp.~$\widetilde{T}^{(2,\psi)}\in\Linfty{\hat{\G}}'$. If $\hat{\G}$
is amenable then the same assertions are true for normal, completely bounded
$T$ and $\mu\in\CzU{\G}^{*}$ (not necessarily positive).
\end{prop}

\begin{proof}
We prove the assertions concerning $R_{\mu}$. As in the proof
of \prettyref{prop:L_R_form_GNS}, by \prettyref{thm:L_R}, $T$ being of the desired form is
equivalent to $T$ commuting with the operators $L_{\om}$, $\om\in\Lone{\G}$.
We use \prettyref{lem:L_R_Hilbert_2_inverted}
and its notation. Put $A:=\left\{ \om_{\z,\eta}:\z,\eta\in Z\right\} $.
The set $A$ dense in $\Lone{\G}$ because $Z$ is dense in $\Ltwo{\G}$.
Hence, the operator $T$ has the desired form if and only if it commutes
with the operators $L_{\om}$, $\om\in A$. Let $\z,\eta\in Z$. Then
$T$ commutes with $L_{\om_{\z,\eta}}$ if and only if $TL_{\om_{\z,\eta}}=L_{\om_{\z,\eta}}T$
on $\mathcal{M}_{\varphi}^{(2)}$ by ultraweak density of $\mathcal{M}_{\varphi}^{(2)}$
in $\Linfty{\G}$. By \prettyref{rem:normal_2_integ}, injectivity
of $\overline{\mathfrak{i}^{(2)}}$ (\prettyref{prop:GL_i} \prettyref{enu:GL_i_Prop_2_12})
and continuity, this is equivalent to $\widetilde{T}^{(2,\varphi)}\widetilde{L}_{\om_{\z,\eta}}^{(2,\varphi)}=\widetilde{L}_{\om_{\z,\eta}}^{(2,\varphi)}\widetilde{T}^{(2,\varphi)}$
on $\mathfrak{i}^{(2)}(\mathcal{M}_{\varphi}^{(2)})$, thus everywhere on $\Ltwo{\QG}$
by density of this subspace (\prettyref{prop:GL_i} \prettyref{enu:GL_i_Prop_2_11}).
So $T$ has the desired form if and only if $\widetilde{T}^{(2,\varphi)}$
commutes with $B:=\{\widetilde{L}_{\om_{P^{-\frac{1}{4}}\delta^{\frac{1}{4}}\z,P^{\frac{1}{4}}\delta^{\frac{1}{4}}\eta}}^{\psi}:\z,\eta\in Z\}$.
Density of $P^{-\frac{1}{4}}\delta^{\frac{1}{4}}Z$
and of $P^{\frac{1}{4}}\delta^{\frac{1}{4}}Z$ in $\Ltwo{\G}$
and \prettyref{prop:R_L_phi_psi} imply that $B$ is dense in $\Linfty{\hat{\G}}'$ in the weak operator topology,
and the assertion follows.
\end{proof}
We now supplement \prettyref{lem:C_z_U_star_conv} by adding a few
more conditions.
\begin{lem}
\label{lem:C_z_U_star_conv_Hilbert}Let $\left(\mu_{i}\right)_{i\in\mathcal{I}}$
be a bounded net in $\CzU{\G}^{*}$, and let $\mu\in\CzU{\G}^{*}$.
Consider the following conditions:\renewcommand{\theenumi}{(a)}
\begin{enumerate}
\item \label{enu:C_z_U_star_conv_Hilbert__1}$\mu_{i}\xrightarrow[i\in\mathcal{I}]{}\mu$
in the $w^{*}$-topology;\renewcommand{\theenumi}{(b)}
\item \label{enu:C_z_U_star_conv_Hilbert__2}$\widetilde{R}_{\mu_{i}}^{\varphi}\xrightarrow[i\in\mathcal{I}]{}\widetilde{R}_{\mu}^{\varphi}$
ultraweakly in $\Linfty{\hat{\G}}$;\renewcommand{\theenumi}{(b')}
\item \label{enu:C_z_U_star_conv_Hilbert__2_strict}$\widetilde{R}_{\mu_{i}}^{\varphi}\xrightarrow[i\in\mathcal{I}]{}\widetilde{R}_{\mu}^{\varphi}$
strictly in $\M{\Cz{\hat{\G}}}$;\renewcommand{\theenumi}{(c)}
\item \label{enu:C_z_U_star_conv_Hilbert__3}$\widetilde{R}_{\mu_{i}}^{(2,\varphi)}\xrightarrow[i\in\mathcal{I}]{}\widetilde{R}_{\mu}^{(2,\varphi)}$
ultraweakly in $\Linfty{\hat{\G}}$;\renewcommand{\theenumi}{(c')}
\item \label{enu:C_z_U_star_conv_Hilbert__3_strict}$\widetilde{R}_{\mu_{i}}^{(2,\varphi)}\xrightarrow[i\in\mathcal{I}]{}\widetilde{R}_{\mu}^{(2,\varphi)}$
strictly in $\M{\Cz{\hat{\G}}}$.
\end{enumerate}
Then \prettyref{enu:C_z_U_star_conv_Hilbert__1}$\iff$\prettyref{enu:C_z_U_star_conv_Hilbert__2}$\iff$\prettyref{enu:C_z_U_star_conv_Hilbert__3},
and they are implied by either of \prettyref{enu:C_z_U_star_conv_Hilbert__2_strict},
\prettyref{enu:C_z_U_star_conv_Hilbert__3_strict}. If $\left\Vert \mu_{i}\right\Vert =\left\Vert \mu\right\Vert $
for all $i\in\mathcal{I}$, then all these conditions are equivalent.
Similar results hold for operators of the form $L_{\mu}$.

\end{lem}

\begin{proof}
The equivalence \prettyref{enu:C_z_U_star_conv_Hilbert__1}$\iff$\prettyref{enu:C_z_U_star_conv_Hilbert__2}
is trivial from \prettyref{prop:R_L_phi_psi} because $\bigl\{(\i\tensor\om)(\Ww):\om\in\Lone{\hat{\G}}\bigr\}$
is dense in $\CzU{\G}$. 

Since $\left(\mu_{i}\right)_{i\in\mathcal{I}}$ is bounded, so is
$\bigl(\widetilde{R}_{\mu_{i}}^{\varphi}\bigr)_{i\in\mathcal{I}}$,
and the implication \prettyref{enu:C_z_U_star_conv_Hilbert__2_strict}$\implies$\prettyref{enu:C_z_U_star_conv_Hilbert__2}
follows. Similarly, $\bigl(\widetilde{R}_{\mu_{i}}^{(2,\varphi)}\bigr)_{i\in\mathcal{I}}$
is bounded by \prettyref{lem:L_R_Hilbert_2}, and so \prettyref{enu:C_z_U_star_conv_Hilbert__3_strict}$\implies$\prettyref{enu:C_z_U_star_conv_Hilbert__3}.

\prettyref{enu:C_z_U_star_conv_Hilbert__3}$\implies$\prettyref{enu:C_z_U_star_conv_Hilbert__1}
Assume that $\widetilde{R}_{\mu_{i}}^{(2,\varphi)}\xrightarrow[i\in\mathcal{I}]{}\widetilde{R}_{\mu}^{(2,\varphi)}$
ultraweakly. By \prettyref{prop:GL_i} \prettyref{enu:GL_i_Prop_2_10},
for all $\nu\in\CzU{\G}^{*}$ and $x,y\in\mathcal{M}_{\varphi}$,
\begin{equation}
\bigl\langle\widetilde{R}_{\nu}^{(2,\varphi)}\mathfrak{i}_{\varphi}^{(2)}(x),\mathfrak{i}_{\varphi}^{(2)}(y)\bigr\rangle=\bigl\langle\mathfrak{i}_{\varphi}^{(2)}(R_{\nu}(x)),\mathfrak{i}_{\varphi}^{(2)}(y)\bigr\rangle=\tr\left(\mathfrak{i}_{\varphi}^{(2)}(R_{\nu}(x))\cdot\mathfrak{i}_{\varphi}^{(2)}(y^{*})\right)=\tr\left(R_{\nu}(x)\cdot\mathfrak{i}_{\varphi}^{(1)}(y^{*})\right).\label{eq:C_z_U_star_conv_Hilbert}
\end{equation}
Since $\bigl(R_{\mu_{i}}\bigr)_{i\in\mathcal{I}}$ is bounded and
$\mathfrak{i}_{\varphi}^{(1)}(\mathcal{M}_{\varphi})$ is dense in
$\Lone{\G}$ by \prettyref{prop:GL_i} \prettyref{enu:GL_i_Prop_2_11},
we deduce from \prettyref{eq:C_z_U_star_conv_Hilbert} that $R_{\mu_{i}}(x)\xrightarrow[i\in\mathcal{I}]{}R_{\mu}(x)$
ultraweakly for every $x\in\mathcal{M}_{\varphi}$. This implies that
$\mu_{i}\xrightarrow[i\in\mathcal{I}]{}\mu$ in the $w^{*}$-topology
by \prettyref{lem:C_z_U_star_conv} implication \prettyref{enu:C_z_U_star_conv__3}$\implies$\prettyref{enu:C_z_U_star_conv__1},
because $\mathcal{M}_{\varphi}\cap\Cz{\G}$ is dense in $\Cz{\G}$.

\prettyref{enu:C_z_U_star_conv_Hilbert__2}$\implies$\prettyref{enu:C_z_U_star_conv_Hilbert__3}
If $\widetilde{R}_{\mu_{i}}^{\varphi}\xrightarrow[i\in\mathcal{I}]{}\widetilde{R}_{\mu}^{\varphi}$
ultraweakly in $\Linfty{\hat{\G}}$, then the boundedness of $\bigl(\widetilde{R}_{\mu_{i}}^{(2,\varphi)}\bigr)_{i\in\mathcal{I}}$
implies that $\widetilde{R}_{\mu_{i}}^{(2,\varphi)}=\hat{\tau}_{-i/4}(\widetilde{R}_{\mu_{i}}^{\varphi})\xrightarrow[i\in\mathcal{I}]{}\hat{\tau}_{-i/4}(\widetilde{R}_{\mu}^{\varphi})=\widetilde{R}_{\mu}^{(2,\varphi)}$
ultraweakly (\prettyref{lem:L_R_Hilbert_2}).

Assume that $\left\Vert \mu_{i}\right\Vert =\left\Vert \mu\right\Vert $
for all $i\in\mathcal{I}$ and that $\mu_{i}\xrightarrow[i\in\mathcal{I}]{}\mu$,
equivalently $\overline{\mu_{i}}\xrightarrow[i\in\mathcal{I}]{}\overline{\mu}$,
in the $w^{*}$-topology. Since $\widetilde{R}_{\nu}^{\varphi}=((\overline{\nu}\tensor\i)(\Ww))^{*}$
for all $\nu\in\CzU{\G}^{*}$, it follows from \citep[Theorem 4.6, (i)$\implies$(ii)]{Runde_Viselter_LCQGs_PosDef}
that $\widetilde{R}_{\mu_{i}}^{\varphi}\xrightarrow[i\in\mathcal{I}]{}\widetilde{R}_{\mu}^{\varphi}$
strictly in $\M{\Cz{\hat{\G}}}$, showing that \prettyref{enu:C_z_U_star_conv_Hilbert__1}$\implies$\prettyref{enu:C_z_U_star_conv_Hilbert__2_strict}.
We now prove that \prettyref{enu:C_z_U_star_conv_Hilbert__1}$\implies$\prettyref{enu:C_z_U_star_conv_Hilbert__3_strict}.
For every $\om\in\Lone{\G}$ we have $\bigl\Vert\bigl(\widetilde{R}_{\mu_{i}}^{(2,\varphi)}-\widetilde{R}_{\mu}^{(2,\varphi)}\bigr)\widetilde{R}_{\om}^{(2,\varphi)}\bigr\Vert=\bigl\Vert\widetilde{R}_{\om\conv(\mu_{i}-\mu)}^{(2,\varphi)}\bigr\Vert\le\left\Vert \om\conv(\mu_{i}-\mu)\right\Vert \xrightarrow[i\in\mathcal{I}]{}0$
by \prettyref{lem:L_R_Hilbert_2} and \prettyref{lem:C_z_U_star_conv}
implication \prettyref{enu:C_z_U_star_conv__1}$\implies$\prettyref{enu:C_z_U_star_conv__5}
in its $L_{\nu}$-version, and similarly for the opposite product.
Writing $\hat{\tau}_{-i/4}$ for the generator in the C$^{*}$-algebraic
sense, the subspace $\bigl\{\widetilde{R}_{\om}^{\varphi} = (\om\tensor\i)(W^{*}):\om\in\Lone{\G}\bigr\}\subseteq D(\hat{\tau}_{-i/4})$,
being norm-dense in $\Cz{\hat{\G}}$ and invariant under the automorphism
group $\hat{\tau}$ of $\Cz{\hat{\G}}$, is a core for $\hat{\tau}_{-i/4}$.
As a result, $\bigl\{\widetilde{R}_{\om}^{(2,\varphi)}=\hat{\tau}_{-i/4}(\widetilde{R}_{\om}^{\varphi}):\om\in\Lone{\G}\bigr\}$
is norm-dense in $\Cz{\hat{\G}}$. Boundedness of $\bigl(\widetilde{R}_{\mu_{i}}^{(2,\varphi)}\bigr)_{i\in\mathcal{I}}$
implies that it converges strictly to $\widetilde{R}_{\mu}^{(2,\varphi)}$
in $\M{\Cz{\hat{\G}}}$.
\end{proof}
We now establish an estimate connecting the distance of the KMS-version of the convolution operator from the identity to the distance between the underlying state and the co-unit. The proof follows that of the main part of \citep[Theorem 6.1, (c)$\implies$(a)]{Daws_Skalski_Viselter__prop_T}, where a corresponding result was shown for maps of the forms $\widetilde{R}_{\mu}^{\varphi},\widetilde{L}_{\mu}^{\psi}$ in place of 
$\widetilde{R}_{\mu}^{(2,\varphi)},\widetilde{L}_{\mu}^{(2,\psi)}$. Roughly speaking, we work with `deformed representations'.
\begin{prop}
\label{prop:L_R_Hilbert_2_close_to_identity}Let $\mu$ be a state
of $\CzU{\G}$ and $0<\delta<\frac{1}{2}$ be such that $\bigl\Vert\widetilde{R}_{\mu}^{(2,\varphi)}-\one\bigr\Vert\le\delta$
or $\bigl\Vert\widetilde{L}_{\mu}^{(2,\psi)}-\one\bigr\Vert\le\delta$.
Then $\left\Vert \mu-\epsilon\right\Vert \le2\frac{\sqrt{2\delta}}{1-\sqrt{2\delta}}$.
\end{prop}

\begin{proof}
Let $(\H_{\mu},\pi_{\mu},\xi_{\mu})$ be the GNS construction for
$\mu$. Recall from \prettyref{lem:L_R_Hilbert_2} that for every
$\nu\in\CzU{\G}^{*}$ we have $\Vert\widetilde{R}_{\nu}^{(2,\varphi)}\Vert\le\left\Vert \nu\right\Vert $.
Let $\hat{\om}\in\Lone{\hat{\G}}$. The linear functional given by 
\[
B(\H_{\mu})_{*}\ni\rho\mapsto\hat{\om}(\widetilde{R}_{\rho\circ\pi_{\mu}}^{(2,\varphi)})=\hat{\om}(\hat{\tau}_{-i/4}(\widetilde{R}_{\rho\circ\pi_{\mu}}^{\varphi}))=(\hat{\om}\circ\hat{\tau}_{-i/4})[(\i\tensor(\rho\circ\pi_{\mu}))(\hat{\wW})]
\]
is of norm at most $\left\Vert \hat{\om}\right\Vert $, and thus defines
an element $x_{\hat{\om}}\in B(\H_{\mu})$ of norm at most $\left\Vert \hat{\om}\right\Vert $.
The map $\Lone{\hat{\G}}\ni\hat{\om}\mapsto x_{\hat{\om}}$ is evidently
linear. Let $\hat{\om}_{1},\hat{\om}_{2}\in\Lone{\hat{\G}}$. Assuming
that $\hat{\om}_{j}\in D((\hat{\tau}_{*})_{-i/4})$ and writing $\hat{\omega}_{j}'$
for the closure of $\hat{\om}_{j}\circ\hat{\tau}_{-i/4}$ ($j=1,2$),
for all $\rho\in B(\H_{\mu})_{*}$ we have 
\[
\begin{split}\rho(x_{\hat{\om}_{1}}x_{\hat{\om}_{2}}) & =(\rho\cdot x_{\hat{\om}_{1}})(x_{\hat{\om}_{2}})=(\hat{\om}_{2}\circ\hat{\tau}_{-i/4})\{(\i\tensor\rho)[(\one\tensor x_{\hat{\om}_{1}})(\i\tensor\pi_{\mu})(\hat{\wW})]\}\\
 & =\rho[x_{\hat{\om}_{1}}(\hat{\omega}_{2}'\tensor\pi_{\mu})(\hat{\wW})]=[(\hat{\omega}_{2}'\tensor\pi_{\mu})(\hat{\wW})\cdot\rho](x_{\hat{\om}_{1}})\\
 & =\hat{\om}_{1}'\{(\i\tensor\rho)[(\i\tensor\pi_{\mu})(\hat{\wW})(\one\tensor(\hat{\omega}_{2}'\tensor\pi_{\mu})(\hat{\wW}))]\}
\end{split}
\]
(unconventionally using `$\cdot$' instead of juxtaposition). The identity $(\hat{\Delta}\tensor\i)(\hat{\wW})=\hat{\wW}_{13}\hat{\wW}_{23}$
implies that
\[
(\i\tensor\pi_{\mu})(\hat{\wW})(\one\tensor(\hat{\omega}_{2}'\tensor\pi_{\mu})(\hat{\wW}))=(\i\tensor\hat{\omega}_{2}'\tensor\pi_{\mu})(\hat{\wW}_{13}\hat{\wW}_{23})=(\i\tensor\hat{\omega}_{2}'\tensor\pi_{\mu})(\hat{\Delta}\tensor\i)(\hat{\wW}).
\]
From the last two formulas we get 
\[
\begin{split}\rho(x_{\hat{\om}_{1}}x_{\hat{\om}_{2}}) & =\hat{\om}_{1}'[(\i\tensor\rho)(\i\tensor\hat{\omega}_{2}'\tensor\pi_{\mu})(\hat{\Delta}\tensor\i)(\hat{\wW})]\\
 & =(\hat{\om}_{1}'\conv\hat{\om}_{2}')[(\i\tensor(\rho\circ\pi_{\mu}))(\hat{\wW})].
\end{split}
\]
Using the equality $(\hat{\tau}_{t}\tensor\hat{\tau}_{t})\circ\hat{\Delta}=\hat{\Delta}\circ\hat{\tau}_{t}$
valid for all $t\in\R$, we obtain
\[
\rho(x_{\hat{\om}_{1}}x_{\hat{\om}_{2}})=((\hat{\om}_{1}\conv\hat{\om}_{2})\circ\hat{\tau}_{-i/4})[(\i\tensor(\rho\circ\pi_{\mu}))(\hat{\wW})]=\rho(x_{\hat{\om}_{1}\conv\hat{\om}_{2}}),
\]
proving that $x_{\hat{\om}_{1}}x_{\hat{\om}_{2}}=x_{\hat{\om}_{1}\conv\hat{\om}_{2}}$.
Since $\Lone{\hat{\G}}\ni\hat{\om}\mapsto x_{\hat{\om}}$ is continuous
and $D((\hat{\tau}_{*})_{-i/4})$ is dense in $\Lone{\hat{\G}}$,
we infer that $x_{\hat{\om}_{1}}x_{\hat{\om}_{2}}=x_{\hat{\om}_{1}\conv\hat{\om}_{2}}$
for every $\hat{\om}_{1},\hat{\om}_{2}\in\Lone{\hat{\G}}$.

By assumption, for every state $\hat{\om}\in\Lone{\hat{\G}}$,
\[
\left\Vert x_{\hat{\om}}\xi_{\mu}-\xi_{\mu}\right\Vert ^{2}\leq2(1-\Re\left\langle x_{\hat{\om}}\xi_{\mu},\xi_{\mu}\right\rangle )=2\Re\hat{\om}(\one-\widetilde{R}_{\mu}^{(2,\varphi)})\leq2\delta.
\]
As a result, the unique vector $\xi\in\H_{\mu}$ of minimal norm in
the (convex) closure $C$ of $\{x_{\hat{\om}}\xi_{\mu}:\hat{\om}\in\Lone{\hat{\G}}\text{ is a state}\}$
satisfies $\left\Vert \xi-\xi_{\mu}\right\Vert \leq\sqrt{2\delta}$.
By the foregoing, $C$ is globally invariant under each of the contractions
$x_{\hat{\om}}$ for a state $\hat{\om}\in\Lone{\hat{\G}}$, and so
$\xi$ is invariant under each of these operators. That is, for every
state $\hat{\om}\in\Lone{\hat{\G}}$, 
\[
\hat{\om}(\left\Vert \xi\right\Vert ^{2}\one)=\left\Vert \xi\right\Vert ^{2}=\om_{\xi}(x_{\hat{\om}})=(\hat{\om}\circ\hat{\tau}_{-i/4})[(\i\tensor(\om_{\xi}\circ\pi_{\mu}))(\hat{\wW})].
\]
Consequently, $\hat{\tau}_{-i/4}[(\i\tensor(\om_{\xi}\circ\pi_{\mu}))(\hat{\wW})]=\left\Vert \xi\right\Vert ^{2}\one$,
that is, $(\i\tensor(\om_{\xi}\circ\pi_{\mu}))(\hat{\wW})=\left\Vert \xi\right\Vert ^{2}\one$.
Hence $\om_{\frac{1}{\left\Vert \xi\right\Vert }\xi}\circ\pi_{\mu}=\epsilon$,
and as $\left\Vert \mu-\epsilon\right\Vert =\Vert(\om_{\frac{1}{\left\Vert \xi\right\Vert }\xi}-\om_{\xi_{\mu}})\circ\pi_{\mu}\Vert\leq
\| \frac{1}{\left\Vert \xi\right\Vert }\xi + \xi_{\mu}\| \| \frac{1}{\left\Vert \xi\right\Vert }\xi - \xi_{\mu}\|
 \leq 2\frac{\sqrt{2\delta}}{1-\sqrt{2\delta}}$,
the proof is complete.
\end{proof}
The following result is the norm analogue of Lemmas \ref{lem:C_z_U_star_conv}
and \ref{lem:C_z_U_star_conv_Hilbert}.
\begin{cor}
\label{cor:L_R_Hilbert_close_to_identity}Let $\left(\mu_{i}\right)_{i\in\mathcal{I}}$
be a net of states of $\CzU{\G}$. Then the following conditions are
equivalent:
\begin{enumerate}
\item \label{enu:L_R_Hilbert_close_to_identity__1}$\mu_{i}\xrightarrow[i\in\mathcal{I}]{}\epsilon$
in $\CzU{\G}^{*}$-norm;
\item \label{enu:L_R_Hilbert_close_to_identity__2}$R_{\mu_{i}}^{\mathrm{u}}\xrightarrow[i\in\mathcal{I}]{}\i$
in $B(\CzU{\G})$-norm;
\item \label{enu:L_R_Hilbert_close_to_identity__3}$\widetilde{R}_{\mu_{i}}^{\varphi}\xrightarrow[i\in\mathcal{I}]{}\one$
in $\Linfty{\hat{\G}}$-norm;
\item \label{enu:L_R_Hilbert_close_to_identity__4}$\widetilde{R}_{\mu_{i}}^{(2,\varphi)}\xrightarrow[i\in\mathcal{I}]{}\one$
in $\Linfty{\hat{\G}}$-norm.
\end{enumerate}
A similar statement holds true for maps of the form $L_{\mu}$.

\end{cor}

\begin{proof}
It is immediate that \prettyref{enu:L_R_Hilbert_close_to_identity__1}$\implies$\prettyref{enu:L_R_Hilbert_close_to_identity__2},
\prettyref{enu:L_R_Hilbert_close_to_identity__3}. From \prettyref{lem:L_R_Hilbert_2}
we have \prettyref{enu:L_R_Hilbert_close_to_identity__1}$\implies$\prettyref{enu:L_R_Hilbert_close_to_identity__4}. 

\prettyref{enu:L_R_Hilbert_close_to_identity__2}$\implies$\prettyref{enu:L_R_Hilbert_close_to_identity__1}
Use that $\epsilon\circ R_{\nu}^{\mathrm{u}}=\nu$ for every $\nu\in\CzU{\G}^{*}$. 

\prettyref{enu:L_R_Hilbert_close_to_identity__3}$\implies$\prettyref{enu:L_R_Hilbert_close_to_identity__1}
This follows from the proof of \citep[Theorem 6.1, (c)$\implies$(a)]{Daws_Skalski_Viselter__prop_T}
(cf.~the proof of \prettyref{prop:L_R_Hilbert_2_close_to_identity}
and ignore $\hat{\tau}$).

\prettyref{enu:L_R_Hilbert_close_to_identity__4}$\implies$\prettyref{enu:L_R_Hilbert_close_to_identity__1}
Use \prettyref{prop:L_R_Hilbert_2_close_to_identity}.
\end{proof}
\begin{question}
While it is obvious that the equivalent conditions of \prettyref{cor:L_R_Hilbert_close_to_identity}
imply that $R_{\mu_{i}}\xrightarrow[i\in\mathcal{I}]{}\i$ in $B(\Linfty{\G})$-norm,
we ask whether the converse is true. The answer is clearly affirmative
when $\G$ is co-amenable.
\end{question}

\section{Semigroups of convolution operators}\label{sec:semigroups}

In this section we discuss the key notion for our paper, namely that of convolution semigroups, and use the results of the last section and of the Appendix to obtain our central result, Theorem \ref{thm:corres_conv_smgrps_Dirichlet_forms}.

\begin{defn}[\citep{Lindsay_Skalski__conv_semigrp_states}]
Let $\G$ be a locally compact quantum group. A \emph{convolution semigroup of positive functionals
on $\CzU{\G}$} (or \emph{on $\G$}) is a family $\left(\mu_{t}\right)_{t\geq0}$
in $\CzU{\G}_{+}^{*}$ such that $\mu_{t}\conv\mu_{s}=\mu_{t+s}$
for all $t,s\geq0$ and $\mu_{0}=\epsilon$. It is called \emph{$w^{*}$-continuous} (or 
\emph{pointwise continuous})
if $\mu_{t}\xrightarrow[t\to0^{+}]{}\epsilon$ in the $w^{*}$-topology
of $\CzU{\G}^{*}$, and \emph{norm continuous} if $\mu_{t}\xrightarrow[t\to0^{+}]{}\epsilon$
in norm.
\end{defn}

The first result is essentially an immediate consequence of the facts established in the last section.

\begin{thm}
\label{thm:corres_conv_smgrps}Let $\G$ be a locally compact quantum group. There exist $1-1$
correspondences between:
\begin{enumerate}
\item \label{enu:corres_conv_smgrps__1}$w^{*}$-continuous convolution
semigroups $\left(\mu_{t}\right)_{t\geq0}$ of states of $\CzU{\G}$;
\item \label{enu:corres_conv_smgrps__2}$C_{0}$-semigroups $\left(T_{t}^{\mathrm{u}}\right)_{t\geq0}$
of completely positive maps of norm $1$ on $\CzU{\G}$ that commute
with the operators $L_{\nu}^{\mathrm{u}}$, $\nu\in\CzU{\G}^{*}$;
\item \label{enu:corres_conv_smgrps__3}semigroups $\left(T_{t}\right)_{t\geq0}$
of normal, unital, completely positive maps on $\Linfty{\G}$ that
are point--ultraweakly continuous at $0^{+}$, and that
satisfy $\Delta\circ T_{t}=(T_{t}\tensor\i)\circ\Delta$ for every
$t\geq0$;
\item \label{enu:corres_conv_smgrps__4}$C_{0}$-semigroups $\left(M_{t}\right)_{t\geq0}$
of norm $1$ right module maps on $\Lone{\G}$ with completely positive
adjoints.
\end{enumerate}
They are given by $T_{t}=R_{\mu_{t}}$, $T_{t}^{\mathrm{u}}=R_{\mu_{t}}^{\mathrm{u}}$
and $M_{t}=(T_{t})_{*}$. All semigroups are continuous in the respective
topologies on all of $\R_{+}$ (for instance, $\left(T_{t}\right)_{t\geq0}$
is a $C_{0}^{*}$-semigroup). Norm continuity of the semigroups in
\prettyref{enu:corres_conv_smgrps__1} and \prettyref{enu:corres_conv_smgrps__2}
are equivalent, and the same is true for the semigroups in \prettyref{enu:corres_conv_smgrps__3}
and \prettyref{enu:corres_conv_smgrps__4}, and the former implies
the latter. Similar statements hold for maps of the form $L_{\mu}$.
\end{thm}

\begin{proof}
The correspondences exist by \prettyref{thm:L_R} \prettyref{enu:L_R_CP_mult}
and \prettyref{lem:C_z_U_star_conv}. Since $C_{0}$-semigroups (such
as $\left(M_{t}\right)_{t\geq0}$) are point\textendash norm continuous
on all of $\R_{+}$, all semigroups considered in the theorem are
continuous in the respective topologies on all of $\R_{+}$. By \prettyref{cor:L_R_Hilbert_close_to_identity}
equivalence \prettyref{enu:L_R_Hilbert_close_to_identity__1}$\iff$\prettyref{enu:L_R_Hilbert_close_to_identity__2},
norm continuity of \prettyref{enu:corres_conv_smgrps__1} and \prettyref{enu:corres_conv_smgrps__2}
are the equivalent, and this obviously implies that of \prettyref{enu:corres_conv_smgrps__3},
equivalently \prettyref{enu:corres_conv_smgrps__4}. 
\end{proof}
\begin{rem}
\label{rem:corres_conv_smgrps}The above \prettyref{thm:corres_conv_smgrps}
could be stated slightly more generally, namely, replace `states'
by `contractive positive functionals' in \prettyref{enu:corres_conv_smgrps__1},
`unital' by `contractive' in \prettyref{enu:corres_conv_smgrps__3},
and `of norm $1$' by `contractive' in \prettyref{enu:corres_conv_smgrps__2}
and \prettyref{enu:corres_conv_smgrps__4} (where `contractive' means `of norm at most $1$' throughout). This statement easily
reduces to the one we used, because by \prettyref{lem:C_z_U_star_conv},
all assumptions imply that $\left(\mu_{t}\right)_{t\geq0}$ is $w^{*}$-continuous;
and since $\left(\mu_{t}\right)_{t\geq0}$ consists of contractions,
the semigroup $\left(\mu_{t}(\one)\right)_{t\geq0}=\left(\left\Vert \mu_{t}\right\Vert \right)_{t\geq0}=\left(\left\Vert T_{t}^{\mathrm{u}}\right\Vert \right)_{t\geq0}=\left(\left\Vert T_{t}\right\Vert \right)_{t\geq0}=\left(\left\Vert M_{t}\right\Vert \right)_{t\geq0}$
(having extended $\mu_t$ strictly to the multiplier algebra on the left) is continuous at $0^{+}$, hence it has the form $\left(e^{tc}\right)_{t\geq0}$,
$c\leq 0$. Thus, by normalising, we may assume that $\left(\mu_{t}\right)_{t\geq0}$
consists of states, so that \prettyref{lem:C_z_U_star_conv} implication
\prettyref{enu:C_z_U_star_conv__1}$\implies$\prettyref{enu:C_z_U_star_conv__5}
applies.
\end{rem}

We now prove \prettyref{thm:intro__corres_conv_smgrps_Dirichlet_forms}
presented in the Introduction.
\begin{thm}
\label{thm:corres_conv_smgrps_Dirichlet_forms}Let $\G$ be a locally compact quantum group.
There exist $1-1$ correspondences between:
\begin{enumerate}
\item \label{enu:corres_conv_smgrps_Dirichlet_forms__1}$w^{*}$-continuous
convolution semigroups $\left(\mu_{t}\right)_{t\geq0}$ of $\Rant^{\mathrm{u}}$-invariant
contractive positive functionals on $\CzU{\G}$;
\item \label{enu:corres_conv_smgrps_Dirichlet_forms__2}$C_{0}^{*}$-semigroups
$\left(T_{t}\right)_{t\geq0}$ of normal, contractive, completely positive
maps on $\Linfty{\G}$ that are KMS-symmetric with respect to $\varphi$
and satisfy $\Delta\circ T_{t}=(T_{t}\tensor\i)\circ\Delta$ for every
$t\geq0$;
\item \label{enu:corres_conv_smgrps_Dirichlet_forms__3}$C_{0}$-semigroups
$\left(S_{t}\right)_{t\ge0}$ of selfadjoint completely Markov operators
on $\Ltwo{\G}$ with respect to $\varphi$ that belong to $\Linfty{\hat{\G}}$;
\item \label{enu:corres_conv_smgrps_Dirichlet_forms__4}completely Dirichlet
forms $Q$ with respect to $\varphi$ that are invariant under $\mathcal{U}(\Linfty{\hat{\G}}')$.
\end{enumerate}
The correspondences are given by $T_{t}=R_{\mu_{t}}$, $S_{t}=\widetilde{R}_{\mu_{t}}^{(2,\varphi)}$
and \prettyref{cor:corres_compl_Markov}. Also, norm continuity of $\left(\mu_{t}\right)_{t\geq0}$,
of $\left(T_{t}\right)_{t\geq0}$ and of $\left(S_{t}\right)_{t\ge0}$
are equivalent, and these are also equivalent to boundedness of $Q$.
Similar statements hold for maps of the form $L_{\mu}$ with invariance
under $\mathcal{U}(\Linfty{\hat{\G}})$ in \prettyref{enu:corres_conv_smgrps_Dirichlet_forms__4}.
\end{thm}

\begin{proof}
The correspondence between \prettyref{enu:corres_conv_smgrps_Dirichlet_forms__1}
and \prettyref{enu:corres_conv_smgrps_Dirichlet_forms__2} is given
by \prettyref{thm:corres_conv_smgrps}, \prettyref{rem:corres_conv_smgrps} and \prettyref{cor:R_L_KMS_sym}.

The correspondence between \prettyref{enu:corres_conv_smgrps_Dirichlet_forms__2},
\prettyref{enu:corres_conv_smgrps_Dirichlet_forms__3} and \prettyref{enu:corres_conv_smgrps_Dirichlet_forms__4}
relies on \prettyref{cor:corres_compl_Markov}: completely Dirichlet
forms $Q$ with respect to $\varphi$ correspond to semigroups $(T_{t})_{t\ge0}$
of normal, completely positive contractions on $\Linfty{\G}$ that
are KMS-symmetric with respect to $\varphi$ and satisfy that $\R_{+}\ni t\mapsto T_{t}a$
is ultraweakly continuous for all $a\in\mathcal{M}_{\varphi}$, and
to $C_{0}$-semigroups $\left(S_{t}\right)_{t\ge0}$ of selfadjoint
completely Markov operators on $\Ltwo{\G}$ with respect to $\varphi$.
Denote the generator of $\left(S_{t}\right)_{t\ge0}=\bigl(\widetilde{T}_{t}^{(2,\varphi)}\bigr)_{t\ge0}$
by $-A$; then $A$ is a positive selfadjoint operator on $\Ltwo{\G}$
and $Q=\left\Vert A^{1/2}\cdot\right\Vert ^{2}$. By \prettyref{prop:L_R_form_KMS},
for each $t\ge0$, $T_{t}$ is of the form $R_{\mu_{t}}$ for $\mu_{t}\in\CzU{\G}_{+}^{*}$
(namely, $\Delta\circ T_{t}=(T_{t}\tensor\i)\circ\Delta$) if and
only if $\widetilde{T}_{t}^{(2,\varphi)}\in\Linfty{\hat{\G}}$. Thus,
the semigroup $(T_{t})_{t\ge0}$ has the form $\bigl(R_{\mu_{t}}\bigr)_{t\ge0}$
with $\left(\mu_{t}\right)_{t\geq0}$ a convolution semigroup of positive
functionals on $\CzU{\G}$ if and only if each of the elements of
the $C_{0}$-semigroup $\bigl(\widetilde{T}_{t}^{(2,\varphi)}\bigr)_{t\ge0}$
commutes with the unitaries of $\Linfty{\hat{\G}}'$, if and only
if $A$ (equivalently, $A^{1/2}$) commutes with the unitaries of
$\Linfty{\hat{\G}}'$, that is, $Q$ is invariant (as a quadratic
form) under all unitaries of $\Linfty{\hat{\G}}'$. When this is the
case, for each $t\ge0$, $T_{t}=R_{\mu_{t}}$ is contractive, thus
so is $\mu_{t}$ by \prettyref{thm:L_R} \prettyref{enu:L_R_CP_mult}. By \prettyref{lem:C_z_U_star_conv}
implication \prettyref{enu:C_z_U_star_conv__3}$\implies$\prettyref{enu:C_z_U_star_conv__1},
which applies because $\mathcal{M}_{\varphi}\cap\Cz{\G}$ is dense
in $\Cz{\G}$, the semigroup $\left(\mu_{t}\right)_{t\geq0}$ is $w^{*}$-continuous.
So by \prettyref{rem:corres_conv_smgrps}, one deduces that ${(T_{t})}_{t\ge0} = {(R_{\mu_{t}})}_{t\ge0}$ is a $C_{0}^{*}$-semigroup; and furthermore, we can and shall assume below that $\left(\mu_{t}\right)_{t\geq0}$ consists of states.

As for norm continuity, the $C_{0}$-semigroup $\left(S_{t}\right)_{t\ge0}=\bigl(\widetilde{T}_{t}^{(2,\varphi)}\bigr)_{t\ge0}=\bigl(\widetilde{R}_{\mu_{t}}^{(2,\varphi)}\bigr)_{t\ge0}$
is norm continuous if and only if $-A$, equivalently $Q$, is bounded,
and by \prettyref{cor:L_R_Hilbert_close_to_identity}, this is the same as $\left(\mu_{t}\right)_{t\geq0}$
being norm continuous. In this case,
${(R_{\mu_{t}})}_{t\ge0}$ is clearly norm continuous. Conversely, for
every $t\ge0$, by \prettyref{prop:R_L_phi_psi} and the fact that $R_{\mu_{t}}$
is a (bounded) normal KMS-symmetric Markov operator, $R_{\mu_{t}}$ is $1$-integrable
and $\widetilde{R}_{\mu_{t}}^{(1,\varphi)*}=R_{\mu_{t}}$ (see also
\prettyref{thm:GL_4_7_forward}).
Hence, $R_{\mu_{t}}-\i$ is adjoint preserving and $1$-integrable, and
$\widetilde{(R_{\mu_{t}}-\i)}^{(1,\varphi)*}=R_{\mu_{t}}-\i$. As
a result, \citep[Proposition 4.1]{Goldstein_Lindsay__Markov_sgs_KMS_symm_weight}
yields that $\bigl\Vert\widetilde{R}_{\mu_{t}}^{(2,\varphi)}-\one_{B(\Ltwo{\QG})}\bigr\Vert=\bigl\Vert\widetilde{(R_{\mu_{t}}-\i)}^{(2,\varphi)}\bigr\Vert\le\bigl\Vert R_{\mu_{t}}-\i\bigr\Vert$.
So norm continuity of ${(R_{\mu_{t}})}_{t\ge0}$ implies that of ${(\widetilde{R}_{\mu_{t}}^{(2,\varphi)})}_{t\ge0}$.
\end{proof}
We recall the notion of the generating functional.
\begin{defn}[{\citep[p.~333, Definition]{Lindsay_Skalski__conv_semigrp_states}}]
\label{def:gen_func}Let $\left(\mu_{t}\right)_{t\geq0}$ be a $w^{*}$-continuous
convolution semigroup of positive functionals on $\CzU{\G}$. Its \emph{generating
functional} is the (generally unbounded) functional $\gamma$ over
$\CzU{\G}$ given by 
\[
\gamma(x):=\lim_{t\to0^{+}}\frac{\mu_{t}(x)-\epsilon(x)}{t}
\]
with maximal domain $D(\gamma)$ consisting of all $x\in\CzU{\G}$
for which this limit exists. 
\end{defn}

By \citep[p.~333, Remarks]{Lindsay_Skalski__conv_semigrp_states},
$D(\gamma)$ is dense in $\CzU{\G}$, and by \citep[Theorem 3.7]{Lindsay_Skalski__conv_semigrp_states},
$\left(\mu_{t}\right)_{t\geq0}$ is uniquely determined by $\gamma$.
Observe that $\gamma$ is hermitian: for every $x\in D(\gamma)$
we have $x^{*}\in D(\gamma)$ and $\gamma(x^{*})=\overline{\gamma(x)}$,
and it is conditionally positive: $\gamma(x)\ge0$ for all
$x\in D(\gamma)\cap\CzU{\G}_{+}\cap\ker\epsilon$. 
If $\G$ is compact and $\left(\mu_{t}\right)_{t\geq0}$ consists of states, 
then additionally $\one\in D(\gamma)$ and $\gamma(\one)=0$.
\begin{lem}
\label{lem:gen_func_bounded}In the notation of \prettyref{def:gen_func},
$\gamma$ is bounded $\iff$ $\gamma$ is everywhere defined $\iff$
$\left(\mu_{t}\right)_{t\geq0}$ is norm continuous.
\end{lem}

\begin{proof}
This is mostly \citep[Theorem 3.7]{Lindsay_Skalski__conv_semigrp_states},
the only missing part is that $\gamma$ being everywhere defined implies
the other conditions. But when this happens, the restriction of $\gamma$
to $\ker\epsilon$ is positive, thus bounded. Therefore, $\gamma$
is bounded on $\CzU{\G}$.
\end{proof}
To describe the relation between the generators appearing in \prettyref{thm:corres_conv_smgrps_Dirichlet_forms}
we need the following elementary lemma.
\begin{lem}
\label{lem:A_sqroot_sgrp_description}Let $A$ be a (generally unbounded)
positive selfadjoint operator on a Hilbert space $\H$. Consider the
$C_{0}$-semigroup ${(e^{-tA})}_{t\ge0}$. Then 
\[
D(A^{1/2})=\bigl\{\z\in\H:\lim_{t\to0^{+}}\om_{\z}\bigl(\frac{1}{t}(\one-e^{-tA})\bigr)\text{ exists}\bigr\}
\]
and 
\[
\left\Vert A^{1/2}\z\right\Vert ^{2}=\lim_{t\to0^{+}}\om_{\z}\bigl(\frac{1}{t}(\one-e^{-tA})\bigr)\qquad(\forall_{\z\in D(A^{1/2})}).
\]
\end{lem}

\begin{proof}
Write $e(\cdot)$ for the spectral measure
of $A$. Let $\z\in\H$. For every $t>0$,
\[
\om_{\z}\bigl(\frac{1}{t}(\one-e^{-tA})\bigr)=\int_{[0,\infty)}\frac{1}{t}(1-e^{-t\l})\left\langle e(\mathrm{d}\l)\z,\z\right\rangle .
\]
If $\lim_{t\to0^{+}}\om_{\z}(\frac{1}{t}(\one-e^{-tA}))$ exists,
then by Fatou's lemma, $\int_{[0,\infty)}\l\left\langle e(\mathrm{d}\l)\z,\z\right\rangle $
is dominated by this limit, thus $\z\in D(A^{1/2})$. On the other
hand, if $\z\in D(A^{1/2})$, namely $\int_{[0,\infty)}\l\left\langle e(\mathrm{d}\l)\z,\z\right\rangle <\infty$,
then relying on the inequality $\frac{1}{t}(1-e^{-t\l})\le\l$
for $t>0$ and $\l\ge0$, Lebesgue's dominated convergence theorem
yields that $\lim_{t\to0^{+}}\om_{\z}(\frac{1}{t}(\one-e^{-tA}))=\int_{[0,\infty)}\l\left\langle e(\mathrm{d}\l)\z,\z\right\rangle =\left\Vert A^{1/2}\z\right\Vert ^{2}$.
\end{proof}

We are now ready to provide the connection mentioned before the last lemma.

\begin{prop}
\label{prop:conv_smgrps_two_generators}In the context of \prettyref{thm:corres_conv_smgrps_Dirichlet_forms},
let $\gamma$ denote the generating functional of $\left(\mu_{t}\right)_{t\geq0}$.
Then 
\[
D(Q)=\bigl\{\z\in\Ltwo{\G}:\tau_{i/4}^{\mathrm{u}}\bigl((\om_{\z}\tensor\i)(\hat{\wW})\bigr)\in D(\gamma)\bigr\}
\]
and 
\[
Q\z=-\gamma\bigl[\tau_{i/4}^{\mathrm{u}}\bigl((\om_{\z}\tensor\i)(\hat{\wW})\bigr)\bigr]\qquad(\forall_{\z\in D(Q)}).
\]
\end{prop}

\begin{proof}
Denote again the generator of $\bigl(\widetilde{R}_{\mu_{t}}^{(2,\varphi)}\bigr)_{t\ge0}$
by $-A$. By \prettyref{lem:A_sqroot_sgrp_description}, 
\[
D(Q)=D(A^{1/2})=\bigl\{\z\in\Ltwo{\G}:\lim_{t\to0^{+}}\om_{\z}\bigl(\frac{1}{t}(\one-e^{-tA})\bigr)\text{ exists}\bigr\}=\bigl\{\z\in\Ltwo{\G}:\lim_{t\to0^{+}}\om_{\z}\bigl(\widetilde{R}_{\frac{1}{t}(\epsilon-\mu_{t})}^{(2,\varphi)}\bigr)\text{ exists}\bigr\}
\]
and 
\[
Q\z=\left\Vert A^{1/2}\z\right\Vert ^{2}=\lim_{t\to0^{+}}\om_{\z}\bigl(\frac{1}{t}(\one-e^{-tA})\bigr)=\lim_{t\to0^{+}}\om_{\z}\bigl(\widetilde{R}_{\frac{1}{t}(\epsilon-\mu_{t})}^{(2,\varphi)}\bigr)\qquad(\forall_{\z\in D(Q)}).
\]
By \prettyref{def:gen_func}, it suffices to observe that for each
$\nu\in\CzU{\G}^{*}$ and $\hat{\om}\in\Lone{\hat{\G}}$ we have $\hat{\om}(\widetilde{R}_{\nu}^{(2,\varphi)})=\nu\bigl[\tau_{i/4}^{\mathrm{u}}\bigl((\hat{\om}\tensor\i)(\hat{\wW})\bigr)\bigr]$;
and this is indeed true because $\widetilde{R}_{\nu}^{(2,\varphi)}=\hat{\tau}_{-i/4}(\widetilde{R}_{\nu}^{\varphi})=\hat{\tau}_{-i/4}((\i\tensor\nu)(\hat{\wW}))$
by \prettyref{prop:R_L_phi_psi} and \prettyref{lem:L_R_Hilbert_2}
\prettyref{enu:L_R_Hilbert_2__2}, so the desired statement is a consequence of the fact
that $(\hat{\om}\tensor\i)(\hat{\wW})\in D(\tau_{i/4}^{\mathrm{u}})$
and the identity $(\hat{\tau}_{t}\tensor\tau_{t}^{\mathrm{u}})(\hat{\wW})=\hat{\wW}$
for every $t\in\R$.
\end{proof}
\begin{rem}[compare \prettyref{rem:CFK_L_map_GNS}]
The last result generalises the analysis in \citep[Theorem 7.1 and the preceding text]{Cipriani_Franz_Kula__sym_Levy_proc},
which treats the case when $\G$ is compact, and only gives a formula
for the values of $Q$ at $\mathfrak{i}^{(2)}\left(\Pol{\G}\right)$,
where $\Pol{\G}$ is the canonical Hopf $*$-algebra that is dense
in $\CC{\G}$.
\end{rem}

\section{Property (T) and the Haagerup Property}\label{sec:prop_T_Haagerup}

In this section we employ the results of the previous sections to
characterise, for \emph{arbitrary} second countable locally compact quantum groups, Property
(T) and the Haagerup Property in terms of $w^{*}$-continuous convolution
semigroups of states. For locally compact groups, these characterisations
are well-known and of fundamental importance, but in the quantum group
setting they were hitherto established only for discrete quantum groups.

Lack of Property (T) is characterised in \prettyref{sub:prop_T} by
the existence of certain unbounded generating functionals, that is:
certain $w^{*}$-continuous, but not norm continuous, convolution
semigroups of states (\prettyref{thm:prop_T}). This generalises \citep[Theorem 3]{Akemann_Walter__unb_neg_def_func},
\citep[Theorem 1.2]{Jolissaint__prop_T_pairs} and \citep[Theorems 2.10.4 and 2.12.4]{Bekka_de_la_Harpe_Valette__book}
for locally compact groups, and \citep[Theorem 5.1]{Kyed__cohom_prop_T_QG}
for unimodular discrete quantum groups (see also \citep[Theorem 6.8]{Daws_Skalski_Viselter__prop_T}).
In order to achieve this `algebra level' result, we first `go down'
to the Hilbert space level, where more techniques are available, and
then `go back up' to the algebra level. These transitions are possible
thanks to the Dirichlet forms machinery. Our approach is very different
from that of \citep{Kyed__cohom_prop_T_QG}, and is closer to those
of \citep{Akemann_Walter__unb_neg_def_func,Jolissaint__prop_T_pairs,Daws_Skalski_Viselter__prop_T}.

In the main result of \prettyref{sub:Haagerup}, \prettyref{thm:Haagerup},
we characterise the Haagerup Property by the existence of certain
$w^{*}$-continuous convolution semigroups whose corresponding KMS
Hilbert space level counterparts `vanish at infinity' in positive
time. This condition should be viewed as \emph{properness} of the
associated generating functional. Our result extends \citep[Theorem 10]{Akemann_Walter__unb_neg_def_func}
for locally compact groups and \citep[Theorem 7.18]{Daws_Fima_Skalski_White_Haagerup_LCQG}
for discrete quantum groups. The approach of our proof is different
from that of \citep{Akemann_Walter__unb_neg_def_func,Daws_Fima_Skalski_White_Haagerup_LCQG},
and instead relies on \prettyref{prop:Haagerup_semigroup}: roughly
speaking, instead of seeking a generator with the desirable properties,
we construct the semigroup directly, following the technique of \citep{Sauvageot__str_Feller_smgrps}, applied later for example in \cite{Caspers_Skalski__Haagerup_AP_Dirichlet_forms}. 
We note that \prettyref{thm:Haagerup} compares with the characterisation of
the Haagerup Property of von Neumann algebras obtained recently in \cite[Theorem 6.7]{Caspers_Skalski__Haagerup_AP_Dirichlet_forms}.

We need first to recall some terminology concerning quantum group representations.

\begin{defn}[{\citep[Definitions 3.3 and 4.1]{Daws_Fima_Skalski_White_Haagerup_LCQG}}]
Let $\G$ be a locally compact quantum group and $U\in\M{\Cz{\G}\tensormin\K(\H)}$ be a unitary
representation of $\G$ on a Hilbert space $\H$.
\begin{enumerate}
\item A vector $\z\in\H$ is \emph{invariant} under $U$ if $U(\eta\tensor\z)=\eta\tensor\z$
for all $\eta\in\Ltwo{\G}$.
\item $U$ \emph{has almost-invariant vectors} if there is a net $\left(\z_{i}\right)_{i\in\mathcal{I}}$
of unit vectors in $\H$ such that for all $\eta\in\Ltwo{\G}$, $\left\Vert U(\eta\tensor\z_{i})-\eta\tensor\z_{i}\right\Vert \xrightarrow[i\in\mathcal{I}]{}0$.
\item $U$ is \emph{mixing} if `it has $C_{0}$-coefficients', namely, if
for every $\z,\xi\in\H$, we have $(\i\tensor\om_{\z,\xi})(U)\in\Cz{\G}$.
\end{enumerate}
\end{defn}

\subsection{\label{sub:prop_T}Property (T)}

Property (T) for discrete quantum groups was first introduced in \citep{Fima__prop_T}; for locally compact quantum groups it was formally defined in \citep{Daws_Fima_Skalski_White_Haagerup_LCQG}. For recent developments we refer the reader to \citep{Daws_Skalski_Viselter__prop_T, BrannanKerr}.

\begin{defn}
A locally compact quantum group has \emph{Property (T)} if each of its unitary representations
that has almost-invariant vectors admits a non-zero invariant vector.
\end{defn}

The following lemma is elementary.
\begin{lem}
\label{lem:a_n_b_n_to_one}Let $\left(a_{i}\right)_{i\in\mathcal{I}}$
and $\left(b_{i}\right)_{i\in\mathcal{I}}$ be nets of contractive
operators on a uniformly convex Banach space (e.g., a Hilbert space).
If $\frac{1}{2}(a_{i}+b_{i})\xrightarrow[i\in\mathcal{I}]{}\one$, then
$a_{i}\xrightarrow[i\in\mathcal{I}]{}\one$ and $b_{i}\xrightarrow[i\in\mathcal{I}]{}\one$.
\end{lem}

We are ready to provide the following strengthening of \citep[Theorem 6.1]{Daws_Skalski_Viselter__prop_T}, involving the KMS-implementations.

\begin{thm}
\label{thm:prop_T_states}For a locally compact quantum group $\G$ the following conditions
are equivalent:\renewcommand{\theenumi}{(a)}
\begin{enumerate}
\item \label{enu:prop_T_states__1}$\hat{\G}$ has Property (T);\renewcommand{\theenumi}{(b)}
\item \label{enu:prop_T_states__2}for every net $\left(\mu_{i}\right)_{i\in\mathcal{I}}$
of states of $\CzU{\G}$, if $\left(\mu_{i}\right)_{i\in\mathcal{I}}$
converges to $\epsilon$ in the $w^{*}$-topology, then it converges
to $\epsilon$ in norm;\renewcommand{\theenumi}{(b')}
\item \label{enu:prop_T_states__2_R_inv}same as \prettyref{enu:prop_T_states__2}
but only for nets $\left(\mu_{i}\right)_{i\in\mathcal{I}}$ of $\Rant^{\mathrm{u}}$-invariant
states of $\CzU{\G}$;\renewcommand{\theenumi}{(c)}
\item \label{enu:prop_T_states__3}for every net $\left(\mu_{i}\right)_{i\in\mathcal{I}}$
of states of $\CzU{\G}$, if ${(\widetilde{R}_{\mu_{i}}^{\varphi})}_{i\in\mathcal{I}}$
converges to $\one$ in the strict topology of $\M{\Cz{\hat{\G}}}$,
then it converges to $\one$ in norm;\renewcommand{\theenumi}{(c')}
\item \label{enu:prop_T_states__3_R_inv}same as \prettyref{enu:prop_T_states__3}
but only for nets $\left(\mu_{i}\right)_{i\in\mathcal{I}}$ of $\Rant^{\mathrm{u}}$-invariant
states of $\CzU{\G}$;\renewcommand{\theenumi}{(d)}
\item \label{enu:prop_T_states__4}for every net $\left(\mu_{i}\right)_{i\in\mathcal{I}}$
of states of $\CzU{\G}$, if the net ${(\widetilde{R}_{\mu_{i}}^{(2,\varphi)})}_{i\in\mathcal{I}}$
in $\Linfty{\hat{\G}}$ converges in the weak (equivalently, strong) operator topology
to $\one$, then it converges to $\one$ in norm;\renewcommand{\theenumi}{(d')}
\item \label{enu:prop_T_states__4_R_inv}same as \prettyref{enu:prop_T_states__4}
but only for nets $\left(\mu_{i}\right)_{i\in\mathcal{I}}$ of $\Rant^{\mathrm{u}}$-invariant
states of $\CzU{\G}$.
\end{enumerate}
Moreover, if $\G$ is second countable, one can replace `net' by `sequence'
throughout. 
Similar assertions hold for operators of the form $L_{\mu}$, $\mu\in\CzU{\G}^{*}$.
\end{thm}

\begin{proof}
The equivalence of \prettyref{enu:prop_T_states__1}, \prettyref{enu:prop_T_states__2}
and \prettyref{enu:prop_T_states__3} was proved in \citep[Theorem 6.1]{Daws_Skalski_Viselter__prop_T}.
The equivalences \prettyref{enu:prop_T_states__2}$\iff$\prettyref{enu:prop_T_states__3}$\iff$\prettyref{enu:prop_T_states__4}
and \prettyref{enu:prop_T_states__2_R_inv}$\iff$\prettyref{enu:prop_T_states__3_R_inv}$\iff$\prettyref{enu:prop_T_states__4_R_inv}
are consequences of \prettyref{lem:C_z_U_star_conv_Hilbert} and \prettyref{cor:L_R_Hilbert_close_to_identity}.
The implications \prettyref{enu:prop_T_states__2}$\implies$\prettyref{enu:prop_T_states__2_R_inv},
\prettyref{enu:prop_T_states__3}$\implies$\prettyref{enu:prop_T_states__3_R_inv}
and \prettyref{enu:prop_T_states__4}$\implies$\prettyref{enu:prop_T_states__4_R_inv}
are trivial.

\prettyref{enu:prop_T_states__3_R_inv}$\implies$\prettyref{enu:prop_T_states__3}
Notice that $\hat{\Rant}(\widetilde{R}_{\nu}^{\varphi})=\widetilde{R}_{\nu\circ \Rant^{\mathrm{u}}}^{\varphi}$
for every $\nu\in\CzU{\G}^{*}$.
Let $\left(\mu_{i}\right)_{i\in\mathcal{I}}$ be a net of states of
$\CzU{\G}$ such that $(\widetilde{R}_{\mu_{i}}^{\varphi})_{i\in\mathcal{I}}$
converges to $\one$ strictly (i.e., in the strict topology of $\M{\Cz{\hat{\G}}}$).
Then $\Big(\widetilde{R}_{\frac{1}{2}(\mu_{i}+\mu_{i}\circ \Rant^{\mathrm{u}})}^{\varphi}\Big)_{i\in\mathcal{I}}={(\frac{1}{2}(\widetilde{R}_{\mu_{i}}^{\varphi}+\hat{\Rant}(\widetilde{R}_{\mu_{i}}^{\varphi})))}_{i\in\mathcal{I}}$
converges to $\one$ strictly because $\hat{\Rant}$ is unital and maps
$\Cz{\hat{\G}}$ into itself. Applying \prettyref{enu:prop_T_states__3_R_inv}
to the net ${(\frac{1}{2}(\mu_{i}+\mu_{i}\circ \Rant^{\mathrm{u}}))}_{i\in\mathcal{I}}$
of $\Rant^{\mathrm{u}}$-invariant states of $\CzU{\G}$, we deduce that
${(\frac{1}{2}(\widetilde{R}_{\mu_{i}}^{\varphi}+\hat{\Rant}(\widetilde{R}_{\mu_{i}}^{\varphi})))}_{i\in\mathcal{I}}$
converges to $\one$ in norm. It follows from \prettyref{lem:a_n_b_n_to_one}
that ${(\widetilde{R}_{\mu_{i}}^{\varphi})}_{i\in\mathcal{I}}$ converges
to $\one$ in norm, as desired.

The last assertion under the assumption that $\G$ is second countable
follows from the easy observation that if $U$ is a unitary representation
of $\hat{\G}$ on a Hilbert space $\H$ with (a net of) almost-invariant
vectors, then since $\Ltwo{\G}$ is separable (\prettyref{prop:second_countable_LCQG}),
one can find a \emph{sequence} of almost-invariant vectors.
\end{proof}
\begin{defn}
Let $\left(\mu_{t}\right)_{t\geq0}$ be a $w^{*}$-continuous convolution
semigroup of positive functionals on $\CzU{\G}$. We say that $\left(\mu_{t}\right)_{t\geq0}$
has\emph{ unbounded generator} if its generating functional (\prettyref{def:gen_func})
is unbounded.
\end{defn}

By \prettyref{lem:gen_func_bounded} and \prettyref{thm:corres_conv_smgrps_Dirichlet_forms},
when the convolution semigroup consists of \emph{$\Rant^{\mathrm{u}}$-invariant contractive}
positive functionals, it has unbounded generator if and only if the associated completely
Dirichlet form is unbounded. 

The following is the main result of this subsection.

\begin{thm}
\label{thm:prop_T}Let $\G$ be a second countable locally compact quantum group. Then $\hat{\G}$
does not have Property (T) $\iff$ there exists a $w^{*}$-continuous
convolution semigroup of $\Rant^{\mathrm{u}}$-invariant states of $\CzU{\G}$
that has unbounded generator (equivalently, whose associated completely
Dirichlet form is unbounded).
\end{thm}

\begin{proof}
$(\implies)$ Assume that $\hat{\G}$ does not have Property (T).
By \prettyref{thm:prop_T_states} implication \prettyref{enu:prop_T_states__4_R_inv}$\implies$\prettyref{enu:prop_T_states__1}, \prettyref{cor:R_L_KMS_sym} and \prettyref{thm:GL_4_7_forward},
there is a sequence $\left(\mu_{n}\right)_{n=1}^{\infty}$ of $\Rant^{\mathrm{u}}$-invariant
states of $\CzU{\G}$ such that the sequence $(a_{n})_{n=1}^{\infty}:=(\widetilde{R}_{\mu_{n}}^{(2,\varphi)})_{n=1}^{\infty}$
of selfadjoint contractions in $\Linfty{\hat{\G}}$ converges to $\one$
in the strong operator topology but not in norm. Clearly $\one-a_{n}\ge0$ for all $n\in\N$.
Let $\left(\z_{j}\right)_{j=1}^{\infty}$ be a dense sequence in $\Ltwo{\G}$.
By passing to a subsequence, one may assume that there exists $\e_{0}>0$
such that for every $n\in\N$,
\begin{equation}
\left\Vert (\one-a_{n})\z_{j}\right\Vert \le\frac{1}{2^{n}}\qquad(\forall_{1\leq j\leq n})\label{eq:prop_T_gen_func__1}
\end{equation}
and

\begin{equation}
\left\Vert \one-a_{n}\right\Vert \ge\e_{0}.\label{eq:prop_T_gen_func__2}
\end{equation}

Whenever $H$ is a (generally unbounded) positive selfadjoint operator
on a Hilbert space, write $Q_{H}$ for the associated closed densely-defined (non-negative)
quadratic form. For $n\in\N$, consider the positive operator $\Delta_{n}:=\sum_{k=1}^{n}k(\one-a_{k})\in B(\Ltwo{\G})$.
Then $\left(\Delta_{n}\right)_{n=1}^{\infty}$, equivalently $\left(Q_{\Delta_{n}}\right)_{n=1}^{\infty}$,
is increasing. Let $Q$ be the closed (non-negative) quadratic form
that is the pointwise limit of $\left(Q_{\Delta_{n}}\right)_{n=1}^{\infty}$
(see \citep[Theorem 4.32]{Davies__semigroups}). It follows from \prettyref{eq:prop_T_gen_func__1}
that the dense subset $\left\{ \z_{j}:j\in\N\right\} $ of $\Ltwo{\G}$
is contained in the domain of $Q$. For every $k\in\N$, $a_{k}$ being
selfadjoint, contractive, and Markov with respect to $\varphi$ implies that $Q_{\one-a_{k}}$
is Dirichlet with respect to $\varphi$ by \citep[Lemma 5.2]{Goldstein_Lindsay__Markov_sgs_KMS_symm_weight}.
Hence, for all $n\in\N$, $Q_{\Delta_{n}}$ is Dirichlet with respect
to $\varphi$, and therefore so is $Q$. More generally 
the operators $R_{\mu_{k}}$, $k\in\N$, are (KMS-symmetric) completely positive contractions, thus
the selfadjoint operators $a_{k}$ are completely Markov with respect to
$\varphi$ (see the terminology introduced in the Appendix and \prettyref{thm:GL_4_7_forward}), so as above, $Q$ is completely Dirichlet with respect to $\varphi$.
Additionally, for all $n\in\N$, $Q_{\Delta_{n}}$ is invariant under
$\mathcal{U}(\Linfty{\hat{\G}}')$ for $\left\{ a_{k}:k\in\N\right\} \subseteq\Linfty{\hat{\G}}$,
so $Q$ also has this property. Finally, it is clear from \prettyref{eq:prop_T_gen_func__2}
that $Q$ is unbounded. \prettyref{thm:corres_conv_smgrps_Dirichlet_forms} (and \prettyref{rem:corres_conv_smgrps})
produce the $w^{*}$-continuous convolution semigroup
of $\Rant^{\mathrm{u}}$-invariant states of $\CzU{\G}$ that is associated
with $Q$, which, by the remark before the theorem, has unbounded generator.

$(\impliedby)$ Assume that $\hat{\G}$ has Property (T), and let
$\left(\mu_{t}\right)_{t\geq0}$ be a $w^{*}$-continuous convolution
semigroup of (even not necessarily $\Rant^{\mathrm{u}}$-invariant) states
of $\CzU{\G}$. By \prettyref{thm:prop_T_states} implication \prettyref{enu:prop_T_states__1}$\implies$\prettyref{enu:prop_T_states__2},
$\mu_{t}\xrightarrow[t\to0^{+}]{}\epsilon$ in norm; that is, the
generator of $\left(\mu_{t}\right)_{t\geq0}$ is bounded (\prettyref{lem:gen_func_bounded}).
\end{proof}
\begin{rem}
Comparing the above proof to those of the classical results that we
extend, see e.g.~\citep[Theorem 3]{Akemann_Walter__unb_neg_def_func},
something seems to be `missing', namely picking an increasing
sequence $\left(K_{n}\right)_{n=1}^{\infty}$ of compact neighbourhoods
of the identity with $G=\bigcup_{n=1}^{\infty}K_{n}$ and choosing
the sequence of continuous positive-definite functions in a way that
makes the series converge uniformly on each of these compact sets,
so as to make the sum $\psi$ continuous. We, however, chose the dense
sequence $\left(\z_{n}\right)_{n=1}^{\infty}$ \emph{arbitrarily}.
The reason is that $\psi$ is conditionally negative definite, and
also measurable as the limit of a sequence of continuous functions;
therefore, it is equal in $\Linfty G$ to a \emph{continuous} conditionally
negative-definite function. This is because for every $t\ge0$, the
measurable function $e^{-t\psi}$ is positive definite by Sch\"onberg's
theorem, thus it is equal in $\Linfty G$ to a continuous positive-definite
function (see \citep{Devinatz__meas_pos_def_oper_func} for a nice
survey on this).
\end{rem}

\subsection{\label{sub:Haagerup}The Haagerup Property}

The following definition was introduced in \citep{Daws_Fima_Skalski_White_Haagerup_LCQG}.

\begin{defn}
A locally compact quantum group $\G$ has the \emph{Haagerup Property} if it admits a unitary
representation that is mixing and has almost-invariant vectors.
\end{defn}

In the following result, which extends \citep[Theorem 6.5 (i)$\iff$(iii)$\iff$(iv) and Proposition 6.12]{Daws_Fima_Skalski_White_Haagerup_LCQG}
by adding two more equivalent conditions, all approximate identities
are understood in the Banach algebraic sense.
\begin{prop}
\label{prop:Haagerup_prop}Let $\G$ be a locally compact quantum group. Then the following
conditions are equivalent:\renewcommand{\theenumi}{(a)}
\begin{enumerate}
\item \label{enu:Haagerup_prop__1}$\hat{\G}$ has the Haagerup Property;\renewcommand{\theenumi}{(b)}
\item \label{enu:Haagerup_prop__2}there is a net $\left(\mu_{i}\right)_{i\in\mathcal{I}}$
of states of $\CzU{\G}$ such that $\bigl(\widetilde{R}_{\mu_{i}}^{\varphi}\bigr)_{i\in\mathcal{I}}$
is an approximate identity for $\Cz{\hat{\G}}$;\renewcommand{\theenumi}{(b')}
\item \label{enu:Haagerup_prop__2_R_inv}same as \prettyref{enu:Haagerup_prop__2}
for a net of $\Rant^{\mathrm{u}}$-invariant states;\renewcommand{\theenumi}{(c)}
\item \label{enu:Haagerup_prop__3}there is a net $\left(\mu_{i}\right)_{i\in\mathcal{I}}$
of states of $\CzU{\G}$ such that $\bigl(\widetilde{R}_{\mu_{i}}^{(2,\varphi)}\bigr)_{i\in\mathcal{I}}$
is an approximate identity for $\Cz{\hat{\G}}$;\renewcommand{\theenumi}{(c')}
\item \label{enu:Haagerup_prop__3_R_inv}same as \prettyref{enu:Haagerup_prop__3}
for a net of $\Rant^{\mathrm{u}}$-invariant states.
\end{enumerate}
Moreover, if $\G$ is second countable, one can replace `net' by `sequence'
throughout. 
Similar assertions hold for operators of the form $L_{\mu}$, $\mu\in\CzU{\G}^{*}$.
\end{prop}

\begin{proof}
The equivalence \prettyref{enu:Haagerup_prop__1}$\iff$\prettyref{enu:Haagerup_prop__2}$\iff$\prettyref{enu:Haagerup_prop__2_R_inv}
is \citep[Theorem 6.5 (i)$\iff$(iii) and Proposition 6.12]{Daws_Fima_Skalski_White_Haagerup_LCQG}.

By \prettyref{lem:C_z_U_star_conv_Hilbert}, for a net $\left(\mu_{i}\right)_{i\in\mathcal{I}}$
of states of $\CzU{\G}$, $\widetilde{R}_{\mu_{i}}^{\varphi}\xrightarrow[i\in\mathcal{I}]{}\one$
strictly in $\M{\Cz{\hat{\G}}}$ if and only if $\widetilde{R}_{\mu_{i}}^{(2,\varphi)}\xrightarrow[i\in\mathcal{I}]{}\one$
strictly in $\M{\Cz{\hat{\G}}}$, if and only if $\mu_{i}\xrightarrow[i\in\mathcal{I}]{}\epsilon$
in the $w^{*}$-topology. Thus, to prove that \prettyref{enu:Haagerup_prop__2}$\iff$\prettyref{enu:Haagerup_prop__3}
and \prettyref{enu:Haagerup_prop__2_R_inv}$\iff$\prettyref{enu:Haagerup_prop__3_R_inv},
we only need to take care of the issue of belonging to $\Cz{\hat{\G}}$.

\prettyref{enu:Haagerup_prop__2}$\implies$\prettyref{enu:Haagerup_prop__3}
and \prettyref{enu:Haagerup_prop__2_R_inv}$\implies$\prettyref{enu:Haagerup_prop__3_R_inv}.
Let $\mu\in\CzU{\G}^{*}$ be hermitian and satisfy $\widetilde{R}_{\mu}^{\varphi}\in\Cz{\hat{\G}}$.
Then as in the proof of \prettyref{lem:L_R_Hilbert_2} \prettyref{enu:L_R_Hilbert_2__3},
we have $\hat{\tau}_{-i/2}(\widetilde{R}_{\mu}^{\varphi})=\hat{\Rant}\bigl(\bigl(\widetilde{R}_{\overline{\mu}}^{\varphi}\bigr)^{*}\bigr)=\hat{\Rant}\bigl(\bigl(\widetilde{R}_{\mu}^{\varphi}\bigr)^{*}\bigr)\in\Cz{\hat{\G}}$,
thus $\widetilde{R}_{\mu}^{\varphi}\in D(\hat{\tau}_{-i/2})$ in the
C$^{*}$-algebraic sense \citep[Proposition 1.24]{Kustermans__one_param_rep}.
Hence $\widetilde{R}_{\mu}^{(2,\varphi)}=\hat{\tau}_{-i/4}(\widetilde{R}_{\mu}^{\varphi})\in\Cz{\hat{\G}}$.
The desired implication follows.

\prettyref{enu:Haagerup_prop__3}$\implies$\prettyref{enu:Haagerup_prop__2}
and \prettyref{enu:Haagerup_prop__3_R_inv}$\implies$\prettyref{enu:Haagerup_prop__2_R_inv}.
Let $\mu\in\CzU{\G}^{*}$ be such that $\widetilde{R}_{\mu}^{(2,\varphi)}\in\Cz{\hat{\G}}$.
Then for $n\in\N$, the operator $x_{n}:=\frac{\sqrt{n}}{\sqrt{\pi}}\int_{\R}e^{-nt^{2}}\hat{\tau}_{t}(\widetilde{R}_{\mu}^{(2,\varphi)})\d t\in\Cz{\hat{\G}}$,
where the integral converges in norm, is entire analytic with respect to
$\hat{\tau}$ in the C$^{*}$-algebraic sense. Define $\mu_{n}\in\CzU{\G}^{*}$
as in \prettyref{eq:CzU_star_tau_smear}. If $\mu$ is $\Rant^{\mathrm{u}}$-invariant,
so is $\mu_{n}$ because of the commutativity of $\tau_{t}^{\mathrm{u}}$
and $\Rant^{\mathrm{u}}$. Since $\widetilde{R}_{\mu}^{(2,\varphi)}=\hat{\tau}_{-i/4}(\widetilde{R}_{\mu}^{\varphi})$
and $\hat{\tau}_{t}(\widetilde{R}_{\mu}^{\varphi})=\widetilde{R}_{\mu\circ\tau_{-t}^{\mathrm{u}}}^{\varphi}$
for every $t\in\R$, we deduce from the ultraweak closedness of $\hat{\tau}_{-i/4}$
that $x_{n}=\hat{\tau}_{-i/4}(\widetilde{R}_{\mu_{n}}^{\varphi})$,
now in the von Neumann algebraic sense (use \prettyref{lem:C_z_U_star_conv_Hilbert}).
Hence $\widetilde{R}_{\mu_{n}}^{\varphi}=\hat{\tau}_{i/4}(x_{n})\in\Cz{\hat{\G}}$.
Recall that $\left\Vert \mu_{n}\right\Vert \le\left\Vert \mu\right\Vert $
for all $n\in\N$ and $\mu_{n}\xrightarrow[n\to\infty]{}\mu$ in the
$w^{*}$-topology, thus $\left\Vert \mu_{n}\right\Vert \xrightarrow[n\to\infty]{}\left\Vert \mu\right\Vert $.

Let $\left(\mu_{i}\right)_{i\in\mathcal{I}}$ be a net of states as
in the assumption. For $i\in\mathcal{I}$ and $n\in\N$, write $\mu_{i,n}:=\left(\mu_{i}\right)_{n}$.
Since $\mu_{i}\xrightarrow[i\in\mathcal{I}]{}\epsilon$ in the $w^{*}$-topology
(see above), there is a subnet $\left(\nu_{j}\right)_{j\in\mathcal{J}}$
of $\bigl(\frac{1}{\left\Vert \mu_{i,n}\right\Vert }\mu_{i,n}\bigr)_{(i,n)\in\mathcal{I}\times\N}$
that also converges to $\epsilon$ in the $w^{*}$-topology. By the
foregoing, $\bigl(\widetilde{R}_{\nu_{j}}^{\varphi}\bigr)_{j\in\mathcal{J}}$
is an approximate identity for $\Cz{\hat{\G}}$.

The final assertion, related to second countability, is obvious.
\end{proof}
The next result is a generalisation of \citep[Proposition 6.5]{Caspers_Skalski__Haagerup_AP_Dirichlet_forms},
which in turn is based on \citep{Sauvageot__str_Feller_smgrps} and
\citep{Jolissaint_Martin__alg_vN_fin_Haagerup}. Essentially, for
a C$^{*}$-algebra $\Aalg$, we replace $B(\H)$ and
$\K(\H)$  appearing in those results by $\M{\Aalg}$ and $\Aalg$, respectively. We remark that the constructions in the proof
are identical to those in \citep{Caspers_Skalski__Haagerup_AP_Dirichlet_forms},
that is, they do not depend on $\Aalg$. However, to verify that the
operators belong to $\M{\Aalg}$ or $\Aalg$ we need to present all
steps. A \emph{strictly continuous semigroup} in $\M{\Aalg}$ is a
family $\left(S_{t}\right)_{t\ge0}$ in $\M{\Aalg}$ such that $S_{0}=\one_{\M{\Aalg}}$,
$S_{t}S_{s}=S_{t+s}$ for every $t,s\ge0$, and $\lim_{t\to0^{+}}S_{t}=\one_{\M{\Aalg}}$
strictly. 
\begin{prop}
\label{prop:Haagerup_semigroup}Let $\H$ be a Hilbert space and $\Aalg$
be a separable C$^{*}$-subalgebra of $B(\H)$ acting non-degenerately on $\H$. Form $\M{\Aalg}$
in $B(\H)$. Let $\left(C_{i}\right)_{i\in\mathcal{I}}$ be a family
of closed convex subsets of $\H$, and let $\left(T_{n}\right)_{n=1}^{\infty}$
be a sequence in $\M{\Aalg}$ with the following properties:
\begin{enumerate}
\item for every $n\in\N$, $T_{n}$ is a selfadjoint contraction;
\item for every $n\in\N$ and $i\in\mathcal{I}$, $C_{i}$ is invariant
under $T_{n}$;
\item $T_{n}\xrightarrow[n\to\infty]{}\one_{\M{\Aalg}}$ strictly.
\end{enumerate}
Then there is a strictly continuous semigroup $\left(S_{t}\right)_{t\ge0}$
of selfadjoint contractions in $\M{\Aalg}$ leaving each of the sets
$C_{i}$ invariant. If $T_{n}\in\Aalg$ for every $n\in\N$, then
we can arrange that $S_{t}\in\Aalg$ for every $t>0$.
\end{prop}

\begin{proof}
We follow closely the proof of \citep[Proposition 6.5]{Caspers_Skalski__Haagerup_AP_Dirichlet_forms},
explaining which changes should be made and where. Write $\one:=\one_{\M{\Aalg}}$.

\emph{Step 1}: we first show that we may assume, without loss of generality, that the
operators $\left(T_{n}\right)_{n=1}^{\infty}$ commute.

Indeed, replacing $\left(T_{n}\right)_{n=1}^{\infty}$ by a subsequence if
necessary, we pick a dense sequence $\left(a_{l}\right)_{l=1}^{\infty}$
in $\Aalg$ such that $\left\Vert T_{n}a-a\right\Vert \le2^{-n}\left\Vert a\right\Vert $
for all $n\in\N$ and $a\in\linspan\left\{ T_{j}^{k}a_{l}:1\le j,l\le n-1,0\le k\le n^{2}\right\} $.
For $n\in\N$, let 
\[
\theta_{n}:=\frac{1}{n}(T_{1}+\ldots+T_{n})\in\M{\Aalg}\text{ and }\Delta_{n}:=n\one-(T_{1}+\ldots+T_{n})=n(\one-\theta_{n})\in\M{\Aalg}_{+},
\]
and for $\a>0$, consider the positive contraction $R_{\a,n}:=(\one+\a\Delta_{n})^{-1}\in\M{\Aalg}$.
By \citep[Lemma 6.4]{Caspers_Skalski__Haagerup_AP_Dirichlet_forms},
$C_{i}$ is invariant under $R_{\a,n}$ for every $i\in\mathcal{I}$.
Since $\left\Vert \theta_{n}\right\Vert \le1$, we have 
\[
R_{\a,n}=((n\a+1)\one-n\a\theta_{n})^{-1}=\frac{1}{n\a+1}\left(\one-\frac{n\a}{n\a+1}\theta_{n}\right)^{-1}=\frac{1}{n\a+1}\sum_{k=0}^{\infty}\left(\frac{n\a}{n\a+1}\right)^{k}\theta_{n}^{k},
\]
where the sum converges in the norm of $\M{\Aalg}$. For every $\a>0$,
repeating the calculations in \citep[pp.~88--89]{Sauvageot__str_Feller_smgrps}
or \citep[pp.~43--44]{Jolissaint_Martin__alg_vN_fin_Haagerup} yields
that $\left(R_{\a,n}a\right)_{n=1}^{\infty}$ converges in norm for
all $a\in\left\{ a_{l}\right\} _{l=1}^{\infty}$, thus for all $a\in\Aalg$.
From selfadjointness of the operators $R_{\a,n}$, $n\in\N$, and
strict completeness of $\M{\Aalg}$, we deduce that $\left(R_{\a,n}\right)_{n=1}^{\infty}$
converges strictly in $\M{\Aalg}$ to some positive contraction $\rho_{\a}\in\M{\Aalg}$,
which evidently preserves all sets $C_{i}$, $i\in\mathcal{I}$. Moreover,
by the further arguments in \citep[p.~89]{Sauvageot__str_Feller_smgrps}
or \citep[p.~44]{Jolissaint_Martin__alg_vN_fin_Haagerup}, we have
$\rho_{\a}\xrightarrow[\a\to0^{+}]{}\one$ strictly and $\a\rho_{\a}-\be\rho_{\be}=(\a-\be)\rho_{\a}\rho_{\be}$
for $\a,\be>0$, hence the operators $\left(\rho_{\a}\right)_{\a>0}$
commute. Therefore, we may replace $\left(T_{n}\right)_{n=1}^{\infty}$
by $\bigl(\rho_{\frac{1}{n}}\bigr)_{n=1}^{\infty}$.

It remains to show that if the operators $\left(T_{n}\right)_{n=1}^{\infty}$
belong to $\Aalg$, so do $\left(\rho_{\a}\right)_{\a>0}$, thus do
$\bigl(\rho_{\frac{1}{n}}\bigr)_{n=1}^{\infty}$. For $n,m\in\N$,
$n<m$, put $\Delta_{n,m}:=\Delta_{m}-\Delta_{n}$. For $\a>0$ and
$n\in\N$, repeating the above reasoning with $\left(T_{m}\right)_{m=n+1}^{\infty}$
in lieu of $\left(T_{m}\right)_{m=1}^{\infty}$ implies that the sequence
$\bigl(\bigl(\one+\frac{\a}{n\a+1}\Delta_{n,m}\bigr)^{-1}\bigr)_{m=n+1}^{\infty}$
converges strictly to some positive contraction $\gamma_{\a,n}\in\M{\Aalg}$.
One shows as in \citep[p.~45]{Jolissaint_Martin__alg_vN_fin_Haagerup}
that $\rho_{\a}-(\one+\a\Delta_{n})^{-1}\gamma_{\a,n}=\frac{n\a}{n\a+1}\bigl(\rho_{\a}-(\one+\a\Delta_{n})^{-1}\bigr)\theta_{n}\gamma_{\a,n}$,
and as in \citep[p.~46]{Jolissaint_Martin__alg_vN_fin_Haagerup} uses
this to prove that the operator 
\[
\psi_{\a,n}:=\theta_{n}(\one+\a\Delta_{n})^{-1}\gamma_{\a,n}+\frac{n\a}{n\a+1}\bigl(\rho_{\a}-(\one+\a\Delta_{n})^{-1}\bigr)\theta_{n}\gamma_{\a,n}
\]
satisfies $\left\Vert \rho_{\a}-\psi_{\a,n}\right\Vert \le\frac{2}{n\a}$.
Since $(\one+\a\Delta_{n})^{-1},\rho_{\a},\gamma_{\a,n}\in\M{\Aalg}$
and by assumption $\theta_{n}\in\Aalg$, we conclude that $\rho_{\a}\in\Aalg$.

\emph{Step 2}: we then argue that the assertion holds under the assumption that the operators
$\left(T_{n}\right)_{n=1}^{\infty}$ commute.

To this end we follow \citep[proof of Lemme 2]{Jolissaint_Martin__alg_vN_fin_Haagerup}.
Let $\left(a_{l}\right)_{l=1}^{\infty}$ be a dense sequence in $\Aalg$.
Replacing $\left(T_{n}\right)_{n=1}^{\infty}$ by a subsequence if
necessary, we may assume that $\sum_{n=1}^{\infty}\left\Vert T_{n}a_{l}-a_{l}\right\Vert <\infty$
for each $l\in\N$. Given $n,m\in\N$, $n<m$, define $\theta_{n},\Delta_{n},\Delta_{n,m}$
as above, and given $t\ge0$, consider the positive contraction $S_{t,n}:=e^{-t\Delta_{n}}\in\M{\Aalg}$.
These operators preserve the sets $C_{i}$, $i\in\mathcal{I}$, by
\citep[Lemma 6.4]{Caspers_Skalski__Haagerup_AP_Dirichlet_forms}.
From the commutativity assumption we get $S_{t,m}=S_{t,n}e^{-t\Delta_{n,m}}$.
For every $t\ge0$, one checks as in \citep[p.~40]{Jolissaint_Martin__alg_vN_fin_Haagerup}
that $\left(S_{t,n}a\right)_{n=1}^{\infty}$ converges in norm for
all $a\in\left\{ a_{l}\right\} _{l=1}^{\infty}$, thus for all $a\in\Aalg$.
Consequently, $\left(S_{t,n}\right)_{n=1}^{\infty}$ converges strictly
to a positive contraction $S_{t}\in\M{\Aalg}$, which leaves invariant
each of the sets $C_{i}$, $i\in\mathcal{I}$. As in \citep[p.~41]{Jolissaint_Martin__alg_vN_fin_Haagerup},
$\left(S_{t}\right)_{t\ge0}$ is a strictly continuous semigroup (equivalently,
when viewing these operators as elements of $B(\Aalg)$, $\left(S_{t}\right)_{t\ge0}$
is a $C_{0}$-semigroup).

It remains to show that if the operators $\left(T_{n}\right)_{n=1}^{\infty}$
belong to $\Aalg$, so do $\left(S_{t}\right)_{t>0}$. The argument
starting in \citep[middle of p.~41]{Jolissaint_Martin__alg_vN_fin_Haagerup}
shows that there exists $K>0$ such that for all $t>0$ and $n\in\N$,
$\left\Vert (\one-\theta_{n})S_{t}\right\Vert \le\left\Vert (\one-\theta_{n})S_{t,n}\right\Vert \le\frac{K}{\sqrt{nt}}$.
As a result, $S_{t}=\lim_{n\to\infty}\theta_{n}S_{t}$, and since
$\theta_{n}\in\Aalg$ for all $n\in\N$ by assumption, we conclude
that also $S_{t}\in\Aalg$.
\end{proof}
In the next corollary, which extends \citep[Proposition 6.6]{Caspers_Skalski__Haagerup_AP_Dirichlet_forms},
we use the following notation. For $k\in\N$, write $\Ltwo{M_{k}}:=\Ltwo{M_{k},tr_{k}}$,
where $tr_{k}$ is the canonical (non-normalised) trace on $M_{k}$.
For a Hilbert space $\H$ and $T\in B(\H)$, we consider the amplification
$T^{(k)}:=\one\tensor T\in B\bigl(\Ltwo{M_{k}}\tensor\H\bigr)$.
\begin{cor}
\label{cor:Haagerup_semigroup_ampl}Let $\H$ be a Hilbert space and
$\Aalg$ be a separable C$^{*}$-subalgebra of $B(\H)$ acting non-degenerately on $\H$. View $\M{\Aalg}$
in $B(\H)$. Let $\left(k_{i}\right)_{i\in\mathcal{I}}$ be a family
in $\N$, and for each $i\in\mathcal{I}$, let $C_{i}$ be a closed
convex subset of $\Ltwo{M_{k_{i}}}\tensor\H$. Let $\left(T_{n}\right)_{n=1}^{\infty}$
be a sequence in $\M{\Aalg}$ with the following properties:
\begin{enumerate}
\item for every $n\in\N$, $T_{n}$ is a selfadjoint contraction;
\item for every $n\in\N$ and $i\in\mathcal{I}$, $C_{i}$ is invariant
under $T_{n}^{(k_{i})}$;
\item $T_{n}\xrightarrow[n\to\infty]{}\one_{\M{\Aalg}}$ strictly.
\end{enumerate}
Then there is a strictly continuous semigroup $\left(S_{t}\right)_{t\ge0}$
of selfadjoint contractions in $\M{\Aalg}$ such that for each $t\ge0$
and $i\in\mathcal{I},$ $C_{i}$ is invariant under $S_{t}^{(k_{i})}$.
If $T_{n}\in\Aalg$ for every $n\in\N$, then we can arrange that
$S_{t}\in\Aalg$ for every $t>0$.
\end{cor}

\begin{proof}
Observe that the constructions in the proof of \prettyref{prop:Haagerup_semigroup}
do not depend on the closed convex sets in question, and they are
`amplification invariant'. The corollary thus follows.
\end{proof}

We are ready to present the main result of this subsection.

\begin{thm}
\label{thm:Haagerup}Let $\G$ be a second countable locally compact quantum group. The following
conditions are equivalent:
\begin{enumerate}
\item \label{enu:Haagerup__1}$\hat{\G}$ has the Haagerup Property;
\item \label{enu:Haagerup__2}there exists a $w^{*}$-continuous convolution
semigroup $\left(\mu_{t}\right)_{t\geq0}$ of $\Rant^{\mathrm{u}}$-invariant
states of $\CzU{\G}$ such that $\widetilde{R}_{\mu_{t}}^{(2,\varphi)}\in\Cz{\hat{\G}}$
for every $t>0$.
\end{enumerate}
\end{thm}

\begin{proof}
The implication \prettyref{enu:Haagerup__2}$\implies$\prettyref{enu:Haagerup__1}
follows from \prettyref{prop:Haagerup_prop} implication \prettyref{enu:Haagerup_prop__3_R_inv}$\implies$\prettyref{enu:Haagerup_prop__1}
and \prettyref{lem:C_z_U_star_conv_Hilbert} implication \prettyref{enu:C_z_U_star_conv_Hilbert__1}$\implies$\prettyref{enu:C_z_U_star_conv_Hilbert__3_strict}.

\prettyref{enu:Haagerup__1}$\implies$\prettyref{enu:Haagerup__2}
By \prettyref{prop:Haagerup_prop} implication \prettyref{enu:Haagerup_prop__1}$\implies$\prettyref{enu:Haagerup_prop__3_R_inv}
there exists a sequence $\left(\mu_{n}\right)_{n=1}^{\infty}$ of
$\Rant^{\mathrm{u}}$-invariant states of $\CzU{\G}$ such that $\bigl(\widetilde{R}_{\mu_{n}}^{(2,\varphi)}\bigr)_{n=1}^{\infty}$
is an approximate identity for $\Cz{\hat{\G}}$. Each of the operators
$\widetilde{R}_{\mu_{n}}^{(2,\varphi)}$, $n\in\N$, is selfadjoint
and completely Markov (as defined in the Appendix) by \prettyref{cor:R_L_KMS_sym} and \prettyref{thm:GL_4_7_forward}. Applying \prettyref{cor:Haagerup_semigroup_ampl}
to this sequence with $\Aalg:=\Cz{\hat{\G}}$, $\H:=\Ltwo{\G}$, and
for all $i\in\N$, $k_{i}:=i$ and $C_{i}:=\mathfrak{i}_{tr_{i}\tensor\varphi}^{(2)}(\left[0,\one\right]_{\mathcal{M}^{(2,tr_{i}\tensor\varphi)}})\subseteq\Ltwo{M_{i}}\tensor\Ltwo{\G}$
(cf.~\prettyref{def:n_Markov}), we obtain a strictly continuous semigroup
$\left(S_{t}\right)_{t\ge0}$ in $\M{\Cz{\hat{\G}}}\subseteq\Linfty{\hat{\G}}\subseteq B(\Ltwo{\G})$
consisting of selfadjoint (contractive) completely Markov operators
on $\Ltwo{\G}$ with respect to $\varphi$ with $S_{t}\in\Cz{\hat{\G}}$
for $t>0$. Since $\left(S_{t}\right)_{t\ge0}$ is, in particular,
a $C_{0}$-semigroup on $\Ltwo{\G}$, \prettyref{thm:corres_conv_smgrps_Dirichlet_forms}
implies that there is a $w^{*}$-continuous convolution semigroup
$\left(\mu_{t}\right)_{t\geq0}$ of $\Rant^{\mathrm{u}}$-invariant states
of $\CzU{\G}$ such that $S_{t}=\widetilde{R}_{\mu_{t}}^{(2,\varphi)}$
for every $t\ge0$. This completes the proof.
\end{proof}

\section{Examples}\label{sec:examples}

This section will be devoted to discussing examples of convolution semigroups on a locally compact quantum group $\QG$. As mentioned in the introduction, the classical theory (i.e.\ the case where $\QG$ is a classical locally compact group) is as rich as the general theory of L\'evy processes, and we refer to \citep{Fukushima_Oshima_Takeda__Dirichlet_forms_Markov_proc, Deny__methods_hilbertiennes_thy_potentiel} or \cite{Applebaum__Levy_proc_stock_calc, Liao__book} for further information. Here we focus on describing in detail the dual case and then on providing a method of constructing genuinely quantum examples via cocycle twisting. Finally note that several interesting \emph{compact} quantum group examples are treated in \cite{Cipriani_Franz_Kula__sym_Levy_proc}.

\subsection{\label{sub:cocomm_QGs}Co-commutative quantum groups}

Let $\Gamma$ be a locally compact group and consider $\G:=\hat{\Gamma}$.
Then $\CzU{\G}=\CStarF{\Gamma}$, the full C$^{*}$-algebra of $\Gamma$.
We will tacitly use the natural identification between $\CStarF{\Gamma}^{*}$
and the Fourier\textendash Stieltjes algebra $B(\Gamma)$ (see \citep{Eymard}) given as
follows: when viewing $B(\Gamma)$ as contained algebraically in $\Cb{\Gamma}$,
the pairing of $\mu\in B(\Gamma)$ and the image of $f\in\Lone{\Gamma}$
in $\CStarF{\Gamma}$ is $\int_{\Gamma}f(\gamma)\mu(\gamma)\d\gamma$.
In other words, letting $\left(\lambda_{\gamma}^{\mathrm{u}}\right)_{\gamma\in\Gamma}$
in $\M{\CStarF{\Gamma}}$ be the universal representation of $\Gamma$,
the identification is such that $\mu(\lambda_{\gamma}^{\mathrm{u}})=\mu(\gamma)$
for every $\gamma\in\Gamma$. Then $\sigma(\Ww_{\hat{\Gamma}}^{*})=\wW_{\Gamma}\in\M{\Cz{\Gamma}\tensormin\CStarF{\Gamma}}$,
as a continuous function from $\Gamma$ to $\CStarF{\Gamma}$ with
the strict topology, is the map $\gamma\mapsto\lambda_{\gamma}^{\mathrm{u}}$.
Thus, each $\mu\in B(\Gamma)$ satisfies $(\mu\tensor\i)(\Ww_{\hat{\Gamma}}^{*})=\mu$.
The positive elements in $\CStarF{\Gamma}^{*}$ correspond to positive-definite
functions in $B(\Gamma)$. Since $\Rant^{\mathrm{u}}(\lambda_{\gamma}^{\mathrm{u}})=\lambda_{\gamma^{-1}}^{\mathrm{u}}$
for each $\gamma\in\Gamma$, and since every positive-definite function
$\mu$ satisfies $\mu(\gamma^{-1})=\overline{\mu(\gamma)}$ for each
$\gamma\in\Gamma$, we conclude that positive $\Rant^{\mathrm{u}}$-invariant
elements in $\CStarF{\Gamma}^{*}$ correspond to real-valued positive-definite
functions on $\Gamma$.

A continuous function $\theta:\Gamma\to\R$ is \emph{conditionally
negative definite} if it satisfies the following conditions: CND1.~$\theta(e)=0$; CND2.~$\theta(\gamma^{-1})=\theta(\gamma)$
for all $\gamma\in\Gamma$; and CND3.~for every $n\in\N$, $\gamma_{1},\ldots,\gamma_{n}\in\Gamma$
and $c_{1},\ldots,c_{n}\in\R$ with $\sum_{i=1}^{n}c_{i}=0$, we have
$\sum_{i,j=1}^{n}c_{i}c_{j}\theta(\gamma_{i}^{-1}\gamma_{j})\le0$.
Such functions admit non-negative values, and \emph{Sch\"onberg's theorem}
asserts that a continuous function $\theta:\Gamma\to\R$ satisfying
CND1 and CND2 is conditionally negative definite if and only if $e^{-t\theta}$
is positive definite for all $t\ge0$.

Recall that in the state space of $\CStarF{\Gamma}$, that is, in
the space of positive-definite functions $\mu\in B(\Gamma)$ with
$\mu(e)=1$, the $\sigma(\CStarF{\Gamma}^{*},\CStarF{\Gamma})$-topology
coincides with the topology of uniform convergence on compact sets
\citep[Theorem 13.5.2]{Dixmier__C_star_English}. We note further that $w^{*}$-continuous
convolution semigroups $\left(\mu_{t}\right)_{t\geq0}$ of $\Rant^{\mathrm{u}}$-invariant
states of $\CStarF{\Gamma}$ are in $1-1$ correspondence with conditionally
negative-definite functions $\theta$ on $\Gamma$. Indeed, if $\theta$
is given, then the semigroup $\left(e^{-t\theta}\right)_{t\ge0}$
has the desired properties by Sch\"onberg's theorem. Conversely, given
a semigroup $\left(\mu_{t}\right)_{t\geq0}$ as above, for every $\gamma\in\Gamma$,
$\left(\mu_{t}(\gamma)\right)_{t\ge0}$ is a continuous positive semigroup
in $(0,1]$, hence of the form $\left(e^{-t\theta(\gamma)}\right)_{t\ge0}$
for some continuous function $\theta:\Gamma\to\R_{+}$, which clearly
satisfies CND1 and CND2. By Sch\"onberg's theorem, $\theta$ is conditionally
negative definite.

Since $\G=\hat{\Gamma}$ has trivial scaling group, $\widetilde{R}_{\mu}^{(2)}=\widetilde{R}_{\mu}=(\mu\tensor\i)(\Ww_{\hat{\Gamma}}^{*})$
for all $\mu\in B(\Gamma)$ by \prettyref{lem:L_R_Hilbert_2} \prettyref{enu:L_R_Hilbert_2__4}
and \prettyref{prop:R_L_phi_psi} (since $\G$ is unimodular we omit
the Haar weight from the notation). As already explained, this means
that $\widetilde{R}_{\mu}^{(2)}=\mu$ as elements of $\Cb{\Gamma}\subseteq B(\Ltwo{\Gamma})$
(as multiplication operators). So for a conditionally negative-definite
function $\theta:\Gamma\to\R$ and the matching convolution semigroup
$\left(\mu_{t}\right)_{t\geq0}$ in $B(\Gamma)$, the associated semigroup
$\bigl(\widetilde{R}_{\mu_{t}}^{(2)}\bigr)_{t\ge0}$ in $B(\Ltwo{\Gamma})$
is $\bigl(M_{e^{-t\theta}}\bigr)_{t\ge0}$. It is now an exercise
to check that the generator of this $C_{0}$-semigroup is $-A$, where
$A$ is the positive selfadjoint operator $M_{\theta}$ over $\Ltwo G$
given by $f\mapsto\theta f$ with maximal domain (compare \prettyref{lem:A_sqroot_sgrp_description}).
As a result, the associated Dirichlet form is $Q:\Ltwo G\to\left[0,\infty\right]$
given by $Qf:=\int_{\Gamma}\theta(\gamma)\left|f(\gamma)\right|^{2}\mathrm{d}\gamma$,
$f\in\Ltwo G$.

\subsection{Convolution semigroups via cocycle twistings}

Let $\QG$ be an arbitrary locally compact quantum group.

\begin{defn}
\label{def:2cocycle} A unitary $\Omega\in\Linfty{\G}\tensorn\Linfty{\G}$
is said to be a \emph{2-cocycle} if it satisfies the following equality:
\[
(\one\tensor\Omega)(\i\tensor\Delta)(\Omega)=(\Omega\tensor\one)(\Delta\tensor\i)(\Omega).
\]
We say that $\Omega$ is \emph{trivial} if it is of the form $(U\tensor U)\Delta(U^{*})$
for some unitary $U\in\Linfty{\G}$. 
\end{defn}

The following result is due to De Commer.

\begin{thm}[{\citep[Theorem 9.1.4]{DeCommer__PhD}}] \label{thm:cocycle_twist}
Let $\QG$ be a locally compact quantum group and let $\Omega\in\Linfty{\G}\tensorn\Linfty{\G}$
be a 2-cocycle. Put $\mathsf{M}:=\Linfty{\G}$ and let for $m\in\mathsf{M}$
\[
\Delta_{\Omega}(m):=\Omega\Delta(m)\Omega^{*}.
\]
Then the pair $(\mathsf{M},\Delta_{\Omega})$ defines a locally compact
quantum group $\QG_{\Omega}$ (so in particular we have $\Linfty{\QG_{\Omega}}=\Linfty{\G}$). 
\end{thm}

It is easy to see that if $\Omega$ is trivial, then $\QG\cong\QG_{\Omega}$.

Now the usefulness of this theorem in our context is in that it allows us to
define interesting convolution semigroups via cocycle twistings. Note
that the cocycle twisting construction need not preserve the Haar weight \citep{Fima_Vainerman__twist_Rieffel_deform},
nor does it have to preserve compactness \citep{DeCommer__cocycle_twist_CQG}!
In particular we need not have a straightforward equality $\Ltwo{\QG}=\Ltwo{\QG_{\Omega}}$
or $\Cz{\QG}=\Cz{\QG_{\Omega}}$.
\begin{prop}
\label{prop:convsemigtwist}Consider a $w^*$-continuous convolution semigroup of states
$(\mu_{t})_{t\geq0}$ on a locally compact quantum group $\G$. Assume
that $\Omega\in\Linfty{\G}$ is a 2-cocycle and that the associated
(via \prettyref{thm:corres_conv_smgrps} \prettyref{enu:corres_conv_smgrps__3})
semigroup on $\Linfty{\G}$, $(T_{t})_{t\geq0}$, satisfies the condition
\[
(T_{t}\tensor\i)(\Omega)=\Omega.
\]
Then $(T_{t})_{t\geq0}$ arises also from a convolution semigroup
of states $(\widetilde{\mu}_{t})_{t\geq0}$ on $\QG_{\Omega}$. 
\end{prop}

\begin{proof}
By \prettyref{thm:corres_conv_smgrps} it suffices to show that the
intertwining relation 
\[
\Delta_{\Omega}\circ T_{t}=(T_{t}\tensor\i)\circ\Delta_{\Omega}
\]
holds for each $t\geq0$. But the multiplicative domain arguments
for the unital completely positive map $T_{t}\tensor\i$ show that for each $m\in\Linfty{\G}$
and $t\geq0$ we have 
\begin{align*}
(T_{t}\tensor\i)(\Delta_{\Omega}(m)) & =(T_{t}\tensor\i)(\Omega\Delta(m)\Omega^{*})=(T_{t}\tensor\i)(\Omega)((T_{t}\tensor\i)(\Delta(m)))(T_{t}\tensor\i)(\Omega^{*})\\
 & =\Omega\Delta(T_{t}m)\Omega^{*}=\Delta_{\Omega}(T_{t}m).\qedhere
\end{align*}
\end{proof}
Note that by the comments above we have no guarantee that this construction
preserves KMS-symmetry; it will however preserve that property
if the extra assumptions guarantee that the Haar weights of $\QG_{\Omega}$
coincide with these of $\QG$.

As we want to produce examples related to non-classical quantum groups, we
begin from an `easy' noncommutative $\Linfty{\G}$. Let then $\Gamma$
be a locally compact group and consider $\QG:=\hat{\Gamma}$. Here
on one hand we know how all the convolution semigroups of states on
$\QG$ look like \textendash{} they are associated to conditionally
negative-definite functions on $\Gamma$ (\prettyref{sub:cocomm_QGs}),
and on the other hand we have a natural construction of the cocycle
twists, as explained for example in \citep{Enock_Vainerman__deform_Kac_alg_abel_grp}.
\begin{prop}
\label{prop:dualtwist} Suppose that $\Gamma$ is a locally compact
group with a closed abelian subgroup $H$ (so that via the Takesaki\textendash Tatsuuma theorem from \citep{Takesaki_Tatsuuma}
we have a natural normal inclusion $\gamma:\VN H\hookrightarrow\VN{\Gamma}\cong\Linfty{\hat{\Gamma}}$,
respecting the co-products). Let $k:\hat{H}\times\hat{H}\to\bt$ be
a measurable 2-cocycle on $\hat{H}$, i.e.~a measurable $\bt$-valued
function such that for all $r,s,t\in\hat{H}$, 
\[
k(s,t)k(r,st)=k(r,s)k(rs,t).
\]
Then if we first use the isomorphism $\VN H\cong\Linfty{\hat{H}}$
and then embed the element $k$ into $\Linfty{\hat{\Gamma}}\tensorn\Linfty{\hat{\Gamma}}$,
it becomes a 2-cocycle on $\hat{\Gamma}$, denoted henceforth by $\Omega$.
Further if $\psi:\Gamma\to\bc$ is a conditionally negative-definite
function vanishing on $H$, then the family $(\exp(-t\psi))_{t\geq0}$
of elements of $B(\Gamma)$, viewed as a convolution semigroup of
states on $\hat{\Gamma}$, satisfies the assumptions of Proposition
\ref{prop:convsemigtwist}. In particular we obtain a new convolution
semigroup on $\hat{\Gamma}_{\Omega}$.
\end{prop}

\begin{proof}
The first part is essentially contained in \citep[Section 6]{Enock_Vainerman__deform_Kac_alg_abel_grp}
(the extra assumptions there are not necessary for us thanks to \citep[Theorem 9.1.4]{DeCommer__PhD}).
The construction is as follows: we view $k$ as an element of $\VN H\tensorn\VN H$
and then write $\Omega=(\gamma\tensor\gamma)(k)$. It is easy to check
(as in \citep{Enock_Vainerman__deform_Kac_alg_abel_grp}) that then
$\Omega$ is a 2-cocycle in the sense of \prettyref{def:2cocycle}.

To show the second part it suffices to see that for $t\geq0$ the
condition 
\[
(T_{t}\tensor\i)(\Omega)=\Omega
\]
takes the form 
\begin{equation}
(R_{\exp(-t\psi)}\tensor\i)\left((\gamma\tensor\gamma)(k)\right)=(\gamma\tensor\gamma)(k),\label{eq:cocyclecpdfunction}
\end{equation}
so in particular if $R_{\exp(-t\psi)}|_{\gamma(\VN H)}=\i_{\gamma(\VN H)}$,
then the condition above is satisfied. Now the latter is equivalent
to the fact that $\exp(-t\psi(h))=1$ for all $h\in H$; in other
words, to $\psi|_{H}=0$. 
\end{proof}
\begin{rem}
In fact the proof above shows that sometimes we can hope for the condition
\eqref{eq:cocyclecpdfunction} to hold even if $\psi$ does not vanish
everywhere on $H$. Suppose for simplicity that $H$ is discrete (so
that $\hat{H}$ is compact). To understand what \eqref{eq:cocyclecpdfunction}
says in terms of the relationship between $k$ and $\psi$ we need
to see that it holds if and only if the coefficients $c_{h_{1},h_{2}}$
and $d_{h_{1},h_{2}}$ ($h_{1},h_{2}\in H$), given by 
\[
c_{h_{1},h_{2}}:=\left\langle (R_{\exp(t\psi)}\tensor\i)\left((\gamma\tensor\gamma)(k)\right)(\delta_{e}\tensor\delta_{e}),\delta_{h_{1}}\tensor\delta_{h_{2}}\right\rangle ,
\]
\[
d_{h_{1},h_{2}}:=\left\langle (\gamma\tensor\gamma)(k)(\delta_{e}\tensor\delta_{e}),\delta_{h_{1}}\tensor\delta_{h_{2}}\right\rangle ,
\]
%(and more formally as suitable limits with respect to shifted approximate units in $L^1(\Gamma) \cap L^2(\Gamma)$), 
coincide. We have however 
\[
c_{h_{1},h_{2}}=\exp(t\psi(h_{1}))d_{h_{1},h_{2}}
\]
and 
\[
d_{h_{1},h_{2}}=\int_{\hat{H}}\int_{\hat{H}}k(\hat{h}_{1},\hat{h}_{2})\hat{h}_{1}(h_{1})\hat{h}_{2}(h_{2})\d\hat{h}_{1}\d\hat{h}_{2},
\]
where we use the explicit expressions for the isomorphism $\Linfty{\hat{H}}\cong\VN H$.
Thus the sufficient and necessary condition for \eqref{eq:cocyclecpdfunction}
to hold is that $\psi(h)=0$ for all $h\in H$ such that 
\[
\int_{\hat{H}}\int_{\hat{H}}k(\hat{h}_{1},\hat{h}_{2})\hat{h}_{1}(h_{1})\hat{h}_{2}(h_{2})\d\hat{h}_{1}\d\hat{h}_{2}\neq0.
\]
\end{rem}

Consider now the following example, based on \citep{Enock_Vainerman__deform_Kac_alg_abel_grp};
we will again begin from a general setup, and then pass to a more specific
context.
\begin{example}
Let $H$, $G$ be abelian groups, with $G$ acting on $H$ by automorphisms,
and let $\Gamma:=H\rtimes G$. Suppose further that $k:\hat{H}\times\hat{H}\to\bt$
is a measurable 2-cocycle and consider the quantum group $\QG_{\Omega}:=\widehat{\Gamma}_{\Omega}$
constructed out of the pair $(\Gamma,k)$ as in \prettyref{prop:dualtwist}.
Note that the quantised Heisenberg group of \citep{Enock_Vainerman__deform_Kac_alg_abel_grp}
is a special case of that construction, if we consider $\Gamma:=H_{n}(\br)=\br^{n+1}\rtimes\br^{n}$
($n\geq2$), and $k:\widehat{\br^{n+1}}\times\widehat{\br^{n+1}}\to\bt$
determined by a choice of $j,k\in\{1,\ldots,n\}$, $j\neq k$, and a parameter
$q_{jk}\in\br$, given by the formula 
\[
k(\hat{u},\hat{v},\hat{u}',\hat{v}'):=\exp(iq_{jk}\hat{u}\hat{u}'(\hat{v}_{j}\hat{v}'_{k}-\hat{v}_{k}\hat{v}'_{j}))\qquad(\hat{u},\hat{u}'\in\br,\hat{v},\hat{v}'\in\br^{n}).
\]
Return to the general case of $H\rtimes G$. Suppose that $\psi_{G}:G\to\bc$
is continuous, conditionally negative definite. It follows directly
from the definitions that $\psi:\Gamma\to\bc$ defined by 
\[
\psi(h,g):=\psi_{G}(g)\qquad(h\in H,g\in G),
\]
is a continuous, conditionally negative-definite function on $\Gamma$
such that $\psi|_{H}=0$. Thus it satisfies all the assumptions of
\prettyref{prop:dualtwist}. If we specify again to the context of
$\Gamma=H_{n}(\br)$ we obtain a natural construction associating
to each L\'{e}vy process on $\br^{n}$ a convolution semigroup on
the quantised Heisenberg group of \citep{Enock_Vainerman__deform_Kac_alg_abel_grp}.
\end{example}

The discussion in the above example can be summarised in the following
corollary. The second statement follows from the remarks above and
the fact that the (unitary) antipode does not change in this quantisation
of $H_{n}(\br)$ by \citep[Theorem 5.2]{Enock_Vainerman__deform_Kac_alg_abel_grp},
because $H_{n}(\br)$ is unimodular, so in particular its modular
function is identically $1$ on $\br^{n+1}$.
\begin{cor}
Let $n\in\bn$. Let $(\mu_{t})_{t\geq0}$ denote the family of distributions
of a L\'{e}vy process on $\br^{n}$. Given a choice of $j,k\in\{1,\ldots,n\}$, $j\neq k$,
and a parameter $q_{jk}\in\br$, the construction described in the
above example yields a convolution semigroup of operators on $\QH_{n}^{q}(\br)$,
the quantised Heisenberg group of \citep{Enock_Vainerman__deform_Kac_alg_abel_grp}.
If the distributions $\mu_{t}$ are symmetric, the convolution semigroup
obtained is KMS-symmetric. 
\end{cor}

Remark that it is proved in \citep{Enock_Vainerman__deform_Kac_alg_abel_grp}
that $\Linfty{\QH_{n}^{q}(\br)}\cong \Linfty{\br}\tensorn B(\Ltwo{\br^{n}})$.

\appendix

\section{\citep[Theorem 4.7, `conversely']{Goldstein_Lindsay__Markov_sgs_KMS_symm_weight}}

The purpose of this Appendix is to point out a mistake in the `conversely'
statement of \citep[Theorem 4.7]{Goldstein_Lindsay__Markov_sgs_KMS_symm_weight},
and propose two solutions under additional hypotheses, Propositions
\ref{prop:GL_4_7_fix_1} and \ref{prop:GL_4_7_fix_2}. %We use the notation of \citep{Goldstein_Lindsay__Markov_sgs_KMS_symm_weight} throughout. 
Consider a von Neumann algebra $\Aa$ acting (not
necessarily standardly) on a Hilbert space $\H$ and an n.s.f.~weight
$\varphi$ on $\Aa$, and we denote $\mathbb{A}:=\Aa\rtimes_{\sigma^{\varphi}}\R\subseteq B(\Ltwo{\R}\tensor\H)$.
When indicating that an operator is Markov, we suppress the n.s.f.~weight
in question when there is no confusion.
\begin{thm}[{\citep[Theorem 4.7, first assertion]{Goldstein_Lindsay__Markov_sgs_KMS_symm_weight}}]
\label{thm:GL_4_7_forward}Let $T$ be a KMS-symmetric Markov operator
on $\Aa$ with domain $\mathcal{M}$. Then $T$ is ultraweakly continuous
and $p$-integrable for $p\in\{1,2\}$. Moreover, $(\widetilde{T}^{(1)})^{*}\supseteq T$,
and $\widetilde{T}^{(2)}$ is a symmetric Markov operator on $\Ltwo{\Aa}$. 
\end{thm}

The converse statement claimed that, starting with a symmetric Markov
operator on $\Ltwo{\Aa}$ with domain $\linspan\left[0,h^{1/2}\right]_{\Ltwo{\Aa}}=\linspan\mathfrak{i}^{(2)}(\left[0,\one\right]_{\mathcal{M}^{(2)}})=\mathfrak{i}^{(2)}(\mathcal{M}^{(2)})$,
one can produce a corresponding KMS-symmetric Markov operator on $\Aa$.
The problem with the proof is that $\mathcal{M}^{(2)}$ might not
be closed under the absolute value map; this was used in \citep{Goldstein_Lindsay__Markov_sgs_KMS_symm_weight} to establish
that the operator $T_{0}$ is bounded.

\subsection{First (better) approach}

As in \prettyref{sub:prelim_Haagerup_Lp}, let $h$ be the canonical
closed operator affiliated with $\mathbb{A}$, let $\theta$ be the
action dual to $\sigma^{\varphi}$, and let $\tau$ be the n.s.f.~trace
on $\mathbb{A}$ satisfying $\frac{\mathrm{d}\widetilde{\varphi}}{\mathrm{d}\tau}=h$.
For $n\in\N$ fixed, consider the von Neumann algebra $M_{n}(\Aa)\cong M_{n}\tensor\Aa$
acting on the Hilbert space $\C^{n}\tensor\H$ and the n.s.f.~weight
$tr_{n}\tensor\varphi$ on $M_{n}(\Aa)$, where $tr_{n}$ is the canonical
(non-normalised) trace on $M_{n}$. Below we use tensor products of
unbounded operators on Hilbert spaces \citep[Section 11.2]{Kadison_Ringrose_2}.
Then $\sigma^{tr_{n}\tensor\varphi}=\i_{M_{n}}\tensor\sigma^{\varphi}$,
thus $\mathbb{A}_{n}:=M_{n}(\Aa)\rtimes_{\sigma^{tr_{n}\tensor\varphi}}\R$,
acting on $\Ltwo{\R}\tensor\C^{n}\tensor\H\cong\C^{n}\tensor\Ltwo{\R}\tensor\H$,
equals $M_{n}\tensor\mathbb{A}$, the associated operator $h_{n}$
equals $\one_{M_{n}}\tensor h$, and the associated n.s.f.~trace
$\tau_{n}$ on $\mathbb{A}_{n}$ satisfying $\frac{\mathrm{d}(\widetilde{tr_{n}\tensor\varphi})}{\mathrm{d}\tau_{n}}=h_{n}$
equals $tr_{n}\tensor\tau$ as $\widetilde{tr_{n}\tensor\varphi}=tr_{n}\tensor\widetilde{\varphi}$.
Fixing a system $\left(e_{ij}\right)_{1\le i,j\le n}$ of matrix units
for $M_{n}$, we thus have a $*$-algebras isomorphism $\pres{\tau_{n}}{}{\mathbb{A}_{n}}\cong M_{n}\tensor\pres{\tau}{}{\mathbb{A}}$
as follows: an element $a\in\pres{\tau_{n}}{}{\mathbb{A}_{n}}$ is
associated to the matrix $\left(a_{ij}\right)_{1\le i,j\le n}\in M_{n}\tensor\pres{\tau}{}{\mathbb{A}}$
such that $(e_{ii}\tensor\one)\cdot a\cdot(e_{jj}\tensor\one)=e_{ij}\tensor a_{ij}$
for all $1\le i,j\le n$, and conversely, a matrix $\left(a_{ij}\right)_{1\le i,j\le n}\in M_{n}\tensor\pres{\tau}{}{\mathbb{A}}$
is associated to $\dot{\sum}_{i,j=1}^{n}e_{ij}\tensor a_{ij}$, where
the upper dot stands for summation in the algebra $\pres{\tau_{n}}{}{\mathbb{A}_{n}}$,
namely, followed by taking the closure.
\begin{lem}
\begin{enumerate}
\item \label{enu:amplif_GL_1}For every $p\in[1,\infty)$, the subspace
$\Lp{M_{n}(\Aa)}\subseteq\pres{\tau_{n}}{}{\mathbb{A}_{n}}$ is identified
with the subspace $M_{n}\tensor\Lp{\Aa}\subseteq M_{n}\tensor\pres{\tau}{}{\mathbb{A}}$.
\item \label{enu:amplif_GL_2}For every $q\in[2,\infty)$ we have $\mathcal{N}^{(q,tr_{n}\tensor\varphi)}\cong M_{n}\tensor\mathcal{N}^{(q,\varphi)}$,
and for every $a=\left(a_{ij}\right)_{1\le i,j\le n}$ in this set,
the element $\mathfrak{j}_{tr_{n}\tensor\varphi}^{(q)}(a)\in\Lq{M_{n}(\Aa)}$
is identified with $\bigl(\mathfrak{j}_{\varphi}^{(q)}(a_{ij})\bigr)_{1\le i,j\le n}$.
\item \label{enu:amplif_GL_3}For every $p\in[1,\infty)$ we have $\mathcal{M}^{(p,tr_{n}\tensor\varphi)}\cong M_{n}\tensor\mathcal{M}^{(p,\varphi)}$,
and for every $a=\left(a_{ij}\right)_{1\le i,j\le n}$ in this set,
the element $\mathfrak{i}_{tr_{n}\tensor\varphi}^{(p)}(a)\in\Lp{M_{n}(\Aa)}$
is identified with $\bigl(\mathfrak{i}_{\varphi}^{(p)}(a_{ij})\bigr)_{1\le i,j\le n}$.
\end{enumerate}
\end{lem}

\begin{proof}
\prettyref{enu:amplif_GL_1} This follows because the action on $\mathbb{A}_{n}\cong M_{n}\tensor\mathbb{A}$
dual to $\sigma^{tr_{n}\tensor\varphi}=\i_{M_{n}}\tensor\sigma^{\varphi}$
is $\i_{M_{n}}\tensor\theta$. 

\prettyref{enu:amplif_GL_2} Let $a=\left(a_{ij}\right)_{1\le i,j\le n}\in M_{n}(\Aa)$.
If $a\in M_{n}\tensor\mathcal{N}^{(q,\varphi)}$, namely $h^{1/q}a_{ji}^{*}\in\pres{\tau}{}{\mathbb{A}}$
for all $1\le i,j\le n$, then 
\[
h_{n}^{1/q}a^{*}=(\one_{M_{n}}\tensor h^{1/q})\sum_{1\le i,j\le n}e_{ji}\tensor a_{ij}^{*}\supseteq\sum_{1\le i,j\le n}(\one_{M_{n}}\tensor h^{1/q})(e_{ji}\tensor a_{ij}^{*})=\sum_{1\le i,j\le n}e_{ji}\tensor h^{1/q}a_{ij}^{*}
\]
(the last equality checks easily). We have $e_{ji}\tensor h^{1/q}a_{ij}^{*}\in\pres{\tau_{n}}{}{\mathbb{A}_{n}}$
for all $1\le i,j\le n$, hence $h_{n}^{1/q}a^{*}\in\pres{\tau_{n}}{}{\mathbb{A}_{n}}$,
so that $a\in\mathcal{N}^{(q,tr_{n}\tensor\varphi)}$ from \prettyref{rem:N_q}.
Also, by \prettyref{rem:tau_meas_poly}, the operator $h_{n}^{1/q}a^{*}=\mathfrak{j}_{tr_{n}\tensor\varphi}^{(q)}(a)^{*}$
equals $\dot{\sum\limits _{1\le i,j\le n}}e_{ji}\tensor h^{1/q}a_{ij}^{*}$,
i.e., it is associated to the matrix $\left(h^{1/q}a_{ji}^{*}\right)_{1\le i,j\le n}=\bigl(\mathfrak{j}_{\varphi}^{(q)}(a_{ij})\bigr)_{1\le i,j\le n}^{*}$.

Conversely, if $a\in\mathcal{N}^{(q,tr_{n}\tensor\varphi)}$, then
for fixed $1\le i,j\le n$, the element $a_{i}:=(e_{ii}\tensor\one)a$
also belongs to $\mathcal{N}^{(q,tr_{n}\tensor\varphi)}$, for the
latter is a left ideal in $M_{n}(\Aa)$. That is, $D(h_{n}^{1/q}a_{i}^{*})$
is $\tau_{n}$-dense. Now $(e_{jj}\tensor\one)h_{n}^{1/q}=(e_{jj}\tensor\one)(\one\tensor h^{1/q})\subseteq e_{jj}\tensor h^{1/q}$,
thus $(e_{jj}\tensor\one)h_{n}^{1/q}a_{i}^{*}\subseteq(e_{jj}\tensor h^{1/q})a_{i}^{*}=e_{ji}\tensor h^{1/q}a_{ij}^{*}$,
so these operators also have $\tau_{n}$-dense domains. Since $\tau_{n}=tr_{n}\tensor\tau$,
standard trace calculations show that $D(h^{1/q}a_{ij}^{*})$ is $\tau$-dense,
proving that $a_{ij}\in\mathcal{N}^{(q,\varphi)}$ by \prettyref{rem:N_q}.

\prettyref{enu:amplif_GL_3} This follows readily from \prettyref{enu:amplif_GL_2}
and \prettyref{prop:GL_i} \prettyref{enu:GL_i_Lem_2_9}.
\end{proof}
For simplicity, we keep writing $\mathcal{M}^{(2)},\mathfrak{i}^{(2)}$
for $\mathcal{M}^{(2,\varphi)},\mathfrak{i}_{\varphi}^{(2)}$, etc.
Notice that the identification of the linear spaces $\Ltwo{M_{n}(\Aa)}$
and $M_{n}\tensor\Ltwo{\Aa}$ is actually an isomorphism of Hilbert
spaces: $\Ltwo{M_{n}(\Aa)}\cong\Ltwo{M_{n},tr_{n}}\tensor\Ltwo{\Aa}$.
\begin{defn}
\label{def:n_Markov}An everywhere-defined operator $T$ on $\Aa$
is called \emph{$n$-Markov} if its amplification $\i_{M_{n}}\tensor T$
is Markov on $M_{n}(\Aa)$, namely $T$ is $n$-positive and $n$-contractive.
A (not necessarily everywhere defined) operator $S$ on $\Ltwo{\Aa}$
will be called \emph{$n$-Markov with respect to $\varphi$} if $\one_{M_{n}}\tensor S$
(algebraic tensor product) is Markov on $\Ltwo{M_{n}(\Aa)}\cong\Ltwo{M_{n},tr_{n}}\tensor\Ltwo{\Aa}$
with respect to $tr_{n}\tensor\varphi$, namely $\mathfrak{i}^{(2)}(\mathcal{M}^{(2)})\subseteq D(S)$
and $\mathfrak{i}_{tr_{n}\tensor\varphi}^{(2)}(\left[0,\one\right]_{\mathcal{M}^{(2,tr_{n}\tensor\varphi)}})$
is preserved by $\one_{M_{n}}\tensor S$.
\end{defn}

Clearly, an $n$-Markov operator is Markov.
\begin{prop}
\label{prop:GL_4_7_fix_1}Let $S$ be a symmetric $2$-Markov operator
on $\Ltwo{\Aa}$ with domain $\mathfrak{i}^{(2)}(\mathcal{M}^{(2)})$.
Then there is a unique everywhere-defined, normal, KMS-symmetric Markov
operator $T$ on $\Aa$ that satisfies $\widetilde{T}^{(2)}\supseteq S$.
\end{prop}

\begin{proof}
As in the beginning of the proof of \citep[Theorem 4.7, `conversely']{Goldstein_Lindsay__Markov_sgs_KMS_symm_weight},
since $S$ is symmetric and Markov and $\mathfrak{i}^{(2)}$ is injective,
there is a linear map $T_{0}:\mathcal{M}^{(2)}\to\mathcal{M}^{(2)}$
given by $\mathfrak{i}^{(2)}(T_{0}a)=S\mathfrak{i}^{(2)}(a)$, $a\in\mathcal{M}^{(2)}$,
which is KMS-symmetric and preserves the set $\left[0,\one\right]_{\mathcal{M}^{(2)}}$
(in particular, it is positivity preserving). Analogously, using now
the full strength of $2$-Markovianity of $S$, the map $\one_{M_{2}}\tensor S$
on $\Ltwo{M_{2},tr_{2}}\tensor\mathfrak{i}^{(2)}(\mathcal{M}^{(2)})$
induces the linear map $\i_{M_{2}}\tensor T_{0}$ on $M_{2}\tensor\mathcal{M}^{(2)}\cong\mathcal{M}^{(2,tr_{2}\tensor\varphi)}$.
This map preserves the set $\left[0,\one\right]_{M_{2}\tensor\mathcal{M}^{(2)}}$,
hence it is positivity preserving. To fill the gap in \citep{Goldstein_Lindsay__Markov_sgs_KMS_symm_weight}
we need to show that $T_{0}$ is bounded.

Let $a\in\mathcal{M}^{(2)}$. By \prettyref{lem:Str_Terp_approx} there is a net $\left(e_{\l}\right)_{\l\in\mathcal{I}}$
in $\left[0,\one\right]_{\mathcal{M}_{\infty}}$ that converges
to $\one$ in the strong operator topology and such that $(\sigma_{-\frac{i}{4}}(e_{\l}))_{\l\in\mathcal{I}}$
is bounded. As in the proof of \prettyref{prop:i_p_core}, $\mathfrak{i}^{(2)}(e_{\l}a)\xrightarrow[\l\in\mathcal{I}]{}\mathfrak{i}^{(2)}(a)$
weakly. For every $\l\in\mathcal{I}$, since $T_{0}$ preserves $\left[0,\one\right]_{\mathcal{M}^{(2)}}$,
we have $T_{0}(e_{\l}^{2})\in\left[0,\one\right]_{\mathcal{M}^{(2)}}$,
$T_{0}(a^{*}a)\in\left[0,\left\Vert a\right\Vert ^{2}\one\right]_{\mathcal{M}^{(2)}}$,
and $T_{0}(e_{\l}a)^{*}=T_{0}(a^{*}e_{\l})$ because $T_{0}$ is positivity,
thus adjoint, preserving. In addition, $\left(\begin{smallmatrix}e_{\l} & a\\
0 & 0
\end{smallmatrix}\right)\in\mathcal{M}^{(2,tr_{2}\tensor\varphi)}$, thus $0\le\left(\begin{smallmatrix}e_{\l}^{2} & e_{\l}a\\
a^{*}e_{\l} & a^{*}a
\end{smallmatrix}\right)=\left(\begin{smallmatrix}e_{\l} & a\\
0 & 0
\end{smallmatrix}\right)^{*}\left(\begin{smallmatrix}e_{\l} & a\\
0 & 0
\end{smallmatrix}\right)\in\mathcal{M}^{(2,tr_{2}\tensor\varphi)}.$ Since $\i_{M_{2}}\tensor T_{0}$ is positivity preserving, $\left(\begin{smallmatrix}T_{0}(e_{\l}^{2}) & T_{0}(e_{\l}a)\\
T_{0}(a^{*}e_{\l}) & T_{0}(a^{*}a)
\end{smallmatrix}\right)\ge0$. By \citep[Lemma 5.2 (ii)]{Lance},
for every $\z,\eta\in\H$,
\[
\left|\left\langle T_{0}(e_{\l}a)\z,\eta\right\rangle \right|^{2}\leq\left\langle T_{0}(e_{\l}^{2})\eta,\eta\right\rangle \left\langle T_{0}(a^{*}a)\z,\z\right\rangle \le(\left\Vert a\right\Vert \left\Vert \z\right\Vert \left\Vert \eta\right\Vert )^{2}.
\]
Consequently, $\left(T_{0}(e_{\l}a)\right)_{\l\in\mathcal{I}}$ is
bounded by $\left\Vert a\right\Vert $, so by passing to a subnet,
we may assume that it converges ultraweakly to some $c\in\Aa$ with
$\left\Vert c\right\Vert \le\left\Vert a\right\Vert $. For every
$b\in\mathcal{M}$, by \prettyref{prop:GL_i} \prettyref{enu:GL_i_Prop_2_10},
\[
\tr\left(\mathfrak{i}^{(2)}(T_{0}(e_{\l}a))\cdot\mathfrak{i}^{(2)}(b)\right)=\tr\left(T_{0}(e_{\l}a)\cdot\mathfrak{i}^{(1)}(b)\right)\xrightarrow[\l\in I]{}\tr\left(c\cdot\mathfrak{i}^{(1)}(b)\right).
\]
Since $S$ is symmetric and adjoint preserving, the expression on
the left-hand side is also equal to
\begin{align*}
\tr\left(S(\mathfrak{i}^{(2)}(e_{\l}a))\cdot\mathfrak{i}^{(2)}(b)\right) & =\tr\left(\mathfrak{i}^{(2)}(e_{\l}a)\cdot S(\mathfrak{i}^{(2)}(b))\right)\xrightarrow[\l\in I]{}\tr\left(\mathfrak{i}^{(2)}(a)\cdot S(\mathfrak{i}^{(2)}(b))\right)\\&=\tr\left(S(\mathfrak{i}^{(2)}(a))\cdot\mathfrak{i}^{(2)}(b)\right)
 =\tr\left(\mathfrak{i}^{(2)}(T_{0}a)\cdot\mathfrak{i}^{(2)}(b)\right)=\tr\left(T_{0}a\cdot\mathfrak{i}^{(1)}(b)\right)
\end{align*}
(see \prettyref{eq:L2_symmetry_tr}). In conclusion, $c=T_{0}a$, and we deduce that $\left\Vert T_{0}a\right\Vert \leq\left\Vert a\right\Vert $,
as desired.
\end{proof}
\begin{defn}
\label{def:Markov_semigroup}Recall that a \emph{Markov semigroup
on $\Aa$} is a semigroup $\left(T_{t}\right)_{t\ge0}$ of everywhere-defined,
normal Markov operators (i.e., positive contractions) on $\Aa$ such
that $\R_{+}\ni t\mapsto T_{t}a$ is ultraweakly continuous for all
$a\in\mathcal{M}$; and a \emph{Markov semigroup on $\Ltwo{\Aa}$}
is a $C_{0}$-semigroup of Markov operators on $\Ltwo{\Aa}$ \citep[p.~62]{Goldstein_Lindsay__Markov_sgs_KMS_symm_weight} (see also Definition \ref{def:Markov,KMS}).
For $n\in\N$, we define $n$-Markov semigroups by replacing `Markov'
by `$n$-Markov'.
\end{defn}

Using \prettyref{thm:GL_4_7_forward} and \prettyref{prop:GL_4_7_fix_1}
we obtain a corrected replacement of \citep[Theorem 4.9]{Goldstein_Lindsay__Markov_sgs_KMS_symm_weight}. 
\begin{cor}
\label{cor:GL_4_9_fix}If $\left(T_{t}\right)_{t\ge0}$ is a KMS-symmetric
Markov semigroup on $\Aa$, then $\bigl(\widetilde{T}_{t}^{(2)}\bigr)_{t\ge0}$
is a symmetric Markov semigroup on $\Ltwo{\Aa}$. Conversely, if $\left(S_{t}\right)_{t\ge0}$
is a symmetric $2$-Markov semigroup on $\Ltwo{\Aa}$, then there
exists a KMS-symmetric Markov semigroup $\left(T_{t}\right)_{t\ge0}$
on $\Aa$ such that $S_{t}=\widetilde{T}_{t}^{(2)}$ for all $t\ge0$ (in particular,
$\left(S_{t}\right)_{t\ge0}$ consists of contractions).
\end{cor}

For a quadratic form $Q$ on a Hilbert space $\H$ and $n\in\N$,
define a quadratic form $Q^{(n)}$ on $\Ltwo{M_{n},tr_{n}}\tensor\H$
by 
\[
Q^{(n)}(\z):=\sum_{i,j=1}^{n}Q(\z_{ij})\quad\text{for }\left(\z_{ij}\right)_{1\leq i,j\leq n}=\z\in\Ltwo{M_{n},tr_{n}}\tensor\H.
\]
\begin{defn}
Let $n\in\N$. A quadratic form $Q$ on $\Ltwo{\Aa}$ is called \emph{$n$-Dirichlet
with respect to $\varphi$} if $Q^{(n)}$ is Dirichlet on $\Ltwo{M_{n}(\Aa)}\cong\Ltwo{M_{n},tr_{n}}\tensor\Ltwo{\Aa}$
with respect to $tr_{n}\tensor\varphi$. 
\end{defn}

Let $\left(S_{t}\right)_{t\ge0}$ be a $C_{0}$-semigroup of selfadjoint
contractions on a Hilbert space $\H$ and $Q$ be the associated (closed
densely-defined) quadratic form. Let $n\in\N$. Then the $C_{0}$-semigroup
$\left(\one_{M_{n}}\tensor S_{t}\right)_{t\ge0}$ of selfadjoint contractions
on $\Ltwo{M_{n},tr_{n}}\tensor\H$ has $Q^{(n)}$ as its associated
quadratic form. When $\H=\Ltwo{\Aa}$, \citep[Theorem 5.7]{Goldstein_Lindsay__Markov_sgs_KMS_symm_weight}
implies that $Q$ is $n$-Dirichlet if and only if $S_{t}$ is $n$-Markov
for all $t\ge0$. 

We use the terminology \emph{completely Markov}, resp.~\emph{completely Dirichlet}, to mean $n$-Markov, resp.~$n$-Dirichlet,
for every $n\in\N$. Finally, taking into account \prettyref{cor:GL_4_9_fix}
and the last paragraph, we obtain the following.
\begin{cor}
\label{cor:corres_compl_Markov}There are $1-1$ correspondences between
the following classes:
\begin{itemize}
\item KMS-symmetric completely Markov semigroups $\left(T_{t}\right)_{t\ge0}$
on $\Aa$;
\item symmetric completely Markov semigroups $\left(S_{t}\right)_{t\ge0}$
on $\Ltwo{\Aa}$;
\item completely Dirichlet forms $Q$ on $\Ltwo{\Aa}$.
\end{itemize}
They are given by: $S_{t}=\widetilde{T}_{t}^{(2)}$ for all $t\ge0$, and $\left(S_{t}\right)_{t\ge0}$
is the $C_{0}$-semigroup (of selfadjoint contractions) on $\Ltwo{\Aa}$
associated to $Q$.
\end{cor}

\subsection{Second approach}

As an alternative to \prettyref{prop:GL_4_7_fix_1} we give the next
result. However, it is less practical for our purposes because the
new additional assumption cannot be expressed in terms of quadratic
forms using \citep[Theorem 5.3]{Goldstein_Lindsay__Markov_sgs_KMS_symm_weight}
since convex sets like $\mathfrak{i}^{(2)}(\mathcal{M})$
and $\mathfrak{i}^{(2)}(\left[0,\one\right]_{\mathcal{M}})$ are usually
not closed in $\Ltwo{\Aa}$ by \prettyref{prop:i_p_core}.
\begin{prop}
\label{prop:GL_4_7_fix_2}Let $S$ be a symmetric Markov operator
on $\Ltwo{\Aa}$ with domain $\mathfrak{i}^{(2)}(\mathcal{M}^{(2)})$.
Assume that \emph{$\mathfrak{i}^{(2)}(\mathcal{M})$ is invariant
under $S$}. Then there is a unique everywhere-defined, normal, KMS-symmetric
Markov operator $T$ on $\Aa$ that satisfies $\widetilde{T}^{(2)}\supseteq S$.
\end{prop}

\begin{proof}
Recall that since $S$ is symmetric and Markov, there is a linear
map $T_{0}:\mathcal{M}^{(2)}\to\mathcal{M}^{(2)}$ given by $\mathfrak{i}^{(2)}(T_{0}a)=S\left(\mathfrak{i}^{(2)}(a)\right)$,
$a\in\mathcal{M}^{(2)}$, which is KMS-symmetric and preserves the
set $\left[0,\one\right]_{\mathcal{M}^{(2)}}$. By the additional
assumption, $\mathcal{M}$ is invariant under $T_{0}$. Fix $a\in\Aa_{+}$.
Let $\left(a_{\l}\right)_{\l\in\mathcal{I}}$ be a net in $\left[0,\left\Vert a\right\Vert \one\right]_{\mathcal{M}}$
that converges ultraweakly to $a$. Since $\left[0,\one\right]_{\mathcal{M}^{(2)}}$
is invariant under $T_{0}$, $\left(T_{0}a_{\l}\right)_{\l\in\mathcal{I}}$
is a net in $\left[0,\left\Vert a\right\Vert \one\right]_{\mathcal{M}^{(2)}}$,
so one can assume, by passing to a subnet if necessary, that it converges
ultraweakly to some $Ta\in\Aa_{+}$ of norm at most $\left\Vert a\right\Vert $.
For every $b\in\mathcal{M}$, the KMS-symmetry of $T_{0}$, \prettyref{prop:GL_i}
\prettyref{enu:GL_i_Prop_2_10}, and the fact that $T_{0}(b)\in\mathcal{M}$
imply that 
\begin{equation}
\begin{split}\tr\left(\mathfrak{i}^{(1)}(b)\cdot Ta\right) & =\lim_{\l\in\mathcal{I}}\tr\left(\mathfrak{i}^{(1)}(b)\cdot T_{0}(a_{\l})\right)=\lim_{\l\in\mathcal{I}}\tr\left(T_{0}(b)\cdot\mathfrak{i}^{(1)}(a_{\l})\right)\\
 & =\lim_{\l\in\mathcal{I}}\tr\left(\mathfrak{i}^{(1)}(T_{0}(b))\cdot a_{\l}\right)=\tr\left(\mathfrak{i}^{(1)}(T_{0}(b))\cdot a\right).
\end{split}
\label{eq:GL_Thm_4_7__corr_2}
\end{equation}
By density of $\mathfrak{i}^{(1)}(\mathcal{M})$ in $\Aa_{*}$ (\prettyref{prop:GL_i}
\prettyref{enu:GL_i_Prop_2_11}), $Ta$ does not depend on the choice
of $\left(a_{\l}\right)_{\l\in\mathcal{I}}$. The function $T:\Aa_{+}\to\Aa_{+}$
is evidently additive and positively homogeneous, so it extends to
a well-defined positivity preserving linear map $T:\Aa\to\Aa$, obviously
satisfying $\left\Vert T\right\Vert \leq4$.

Clearly $T_{0}|_{\mathcal{M}}\subseteq T$, but in fact $T_{0}\subseteq T$.
Indeed, if $a\in\mathcal{M}_{+}^{(2)}$ and $\left(a_{\l}'\right)_{\l\in\mathcal{I}}$
is a net in $\mathcal{M}$ that converges ultraweakly to $a$ such
that $\left(\mathfrak{i}^{(2)}(a_{\l}')\right)_{\l\in\mathcal{I}}$
converges weakly to $\mathfrak{i}^{(2)}(a)$ (\prettyref{prop:i_p_core}),
then by \prettyref{eq:GL_Thm_4_7__corr_2}, \prettyref{prop:GL_i}
\prettyref{enu:GL_i_Prop_2_10}, and symmetry of $S$ (see \prettyref{eq:L2_symmetry_tr}),
for every $b\in\mathcal{M}$, 
\[
\begin{split}\tr\left(\mathfrak{i}^{(1)}(b)\cdot Ta\right) & =\tr\left(\mathfrak{i}^{(1)}(T_{0}(b))\cdot a\right)=\lim_{\l\in\mathcal{I}}\tr\left(\mathfrak{i}^{(1)}(T_{0}(b))\cdot a_{\l}'\right)=\lim_{\l\in\mathcal{I}}\tr\left(\mathfrak{i}^{(2)}(T_{0}(b))\cdot\mathfrak{i}^{(2)}(a_{\l}')\right)\\
 & =\tr\left(\mathfrak{i}^{(2)}(T_{0}(b))\cdot\mathfrak{i}^{(2)}(a)\right)=\tr\left(\mathfrak{i}^{(2)}(b)\cdot\mathfrak{i}^{(2)}(T_{0}(a))\right)=\tr\left(\mathfrak{i}^{(1)}(b)\cdot T_{0}(a)\right).
\end{split}
\]
This proves that $T_{0}$ is bounded, and the proof is complete.
\end{proof}
\bibliographystyle{myamsalpha}
\bibliography{GenFunc}

\end{document}